\documentclass[bj,authoryear]{imsart}

\RequirePackage{amsthm,amsmath,amsfonts,amssymb}
\usepackage{subfigure}
\usepackage{tikz}
\RequirePackage{graphicx}%% uncomment this for including figures

\startlocaldefs
\theoremstyle{plain}

\newtheorem{theorem}{Theorem}
\newtheorem{lemma}{Lemma}
\newtheorem{proposition}{Proposition}
\newtheorem{condition}{Condition}
\theoremstyle{remark}
\newtheorem{definition}[theorem]{Definition}
\newtheorem*{example}{Example}
\newtheorem*{remark}{Remark}

\def \P{\mathrm{Pr}}
\def \N{\mathbb{N}}
\def \Z{\mathbb{Z}}
\def \R{\mathbb{R}}
\def \dd{\mathrm{d}}
\newcommand\independent{\protect\mathpalette{\protect\independenT}{\perp}}
\def\independenT#1#2{\mathrel{\rlap{$#1#2$}\mkern2mu{#1#2}}}
\newcommand{\E}[1]{\mathbb{E}\left[#1\right]}
\newcommand{\CE}[2]{\mathbb{E}\left[ #1 \mid #2 \right]}
\newcommand{\curly}[1]{\left\{#1\right\}}
\newcommand{\p}[1]{\left(#1\right)}
\newcommand{\floor}[1]{\left \lfloor #1 \right \rfloor }
\newcommand{\norm}[1]{\left\lVert#1\right\rVert}
\DeclareMathOperator{\Var}{var}
\DeclareMathOperator{\Cov}{cov}
\DeclareMathOperator*{\argmin}{arg\,min}
\allowdisplaybreaks

\endlocaldefs
\begin{document}

\begin{frontmatter}
\title{ }
\title{Conditional Extreme Value Estimation for Dependent Time Series}
\runtitle{Conditional Extreme Value Estimation for Dependent Time Series}
% \thankstext{T1}{Corresponding author: Theodor Henningsen, e‑mail: \href{mailto:th@math.ku.dk}{th@math.ku.dk}}

\begin{aug}
%%%%%%%%%%%%%%%%%%%%%%%%%%%%%%%%%%%%%%%%%%%%%%%
%% Only one address is permitted per author. %%
%% Only division, organization and e-mail is %%
%% included in the address.                  %%
%% Additional information can be included in %%
\author[A]{\fnms{Martin}~\snm{Bladt}\ead[label=e1]{martinbladt@math.ku.dk}}%,
\author[A]{\fnms{Laurits}~\snm{Glargaard}\ead[label=e2]{lauritsglargaard@mail.dk}}
\author[A]{\fnms{Theodor}~\snm{Henningsen \MakeLowercase{(corresponding author)}    }\ead[label=e3]{th@math.ku.dk}}
% \thankstext[cor1]{Corresponding author. E‑mail: \href{mailto:th@math.ku.dk}{th@math.ku.dk}}
%%%%%%%%%%%%%%%%%%%%%%%%%%%%%%%%%%%%%%%%%%%%%%
%% Addresses                                %%
%%%%%%%%%%%%%%%%%%%%%%%%%%%%%%%%%%%%%%%%%%%%%%
\address[A]{Department of Mathematical Sciences, University of Copenhagen, Denmark,\\
\printead[presep={\ }]{e1,e2,e3}}

\end{aug}

\begin{abstract}
We study the consistency and weak convergence of the conditional tail function and conditional Hill estimators under broad dependence assumptions for a heavy-tailed response sequence and a covariate sequence. Consistency is established under 
$\alpha$-mixing, while asymptotic normality follows from 
$\beta$-mixing and second-order conditions. A key aspect of our approach is its versatile functional formulation in terms of the conditional tail process. Simulations demonstrate its performance across dependence scenarios. We apply our method to extreme event modelling in the oil industry, revealing distinct tail behaviours under varying conditioning values.
\end{abstract}



\end{frontmatter}
%%%%%%%%%%%%%%%%%%%%%%%%%%%%%%%%%%%%%%%%%%%%%%
%% Please use \tableofcontents for articles %%
%% with 50 pages and more                   %%
%%%%%%%%%%%%%%%%%%%%%%%%%%%%%%%%%%%%%%%%%%%%%%
\tableofcontents

%%%%%%%%%%%%%%%%%%%%%%%%%%%%%%%%%%%%%%%%%%%%%%
%%%% Main text entry area:
\section{Introduction}

The analysis of extremal events in the presence of covariates is a central problem in modern risk management. A persistent challenge in this setting is the temporal dependence inherent in both the covariate process and the process of interest, which complicates statistical inference. This paper focuses on the estimators for a heavy-tailed mixing time series given another mixing covariate time series.

Functional limit theorems for the tail empirical process under dependence conditions have been studied in the literature. Early  results were established by~\citet{Rootzen_1995} for $\alpha$-mixing sequences and later extended by~\citet{Rootzen_2009} to more general dependence structures. In~\citet{Drees_2000}  the asymptotic normality of the POT estimator was shown under $\beta$-mixing conditions. Among the most widely used estimators for heavy-tailed data is the Hill estimator, originally introduced by~\citet{Hill1975}. Its consistency in dependent settings was first investigated by~\citet{Resnick_Starica_1995}, who analysed its behaviour under various dependence structures. Further work by~\citet{Resnick_Starica_1997} established the consistency of the tail empirical measure for autoregressive processes with heavy-tailed innovations. The asymptotic normality of the Hill estimator for i.i.d. data has been well documented, with significant contributions from~\citet{Hall_1982},~\citet{Davis_Resnick_1984},~\citet{Haeusler_Teugels_1985}, and~\citet{Csorgo_Mason_1985}. Extending these results to stationary sequences,~\citet{HSING1991117} proved a central limit theorem for the Hill estimator under $\beta$-mixing conditions, and a modern treatment can be found in~\cite{Kulik2020}.

In this paper, we investigate the consistency and weak convergence of the conditional kernel Hill estimator under general assumptions on both the response and covariate sequences. Specifically, we establish consistency under $\alpha$-mixing conditions and derive asymptotic normality under $\beta$-mixing conditions, with the latter results obtained under two general second-order conditions. A key feature of our approach is its functional formulation, which naturally leads to limit theorems for the conditional tail process. Through simulation studies, we investigate our method across different scenarios. Notably, a distinguishing advantage of our estimator is its simple (and easy to plug-in) variance structure, in contrast to the often intricate variance expressions encountered in the unconditional case. We apply our method to extreme event modelling in the oil industry where, interestingly, distinct tail behaviours arise under different conditioning values. 

Other work is related to ours. In the presence of covariates, the asymptotic properties of extreme value estimators have been investigated using regression-based approaches. In~\citet{Daouia2010} tail index estimation is studied under $\alpha$-mixing conditions using a smoothed Pickands estimator, a technique later generalized by~\citet{Daouia_2013, Daouia_Stupfler_Usseglio-Carleve_2023} to regression models within the general max-domain of attraction. Furthermore, tail index regression in a random covariate setting has been analysed using a maximum likelihood framework, as explored by~\citet{Wang01092009}. Finally, asymptotic normality of the conditional kernel Hill estimator has been explored through different techniques and for a special second-order regularly varying class in the preprint~\cite{GoegebeurGuillou2024}.

The remainder of the paper is structured as follows. Section~\ref{section:estimators} introduces the mathematical framework for heavy-tailed analysis in dependent settings, defining key concepts and estimators. Section~\ref{section:main} establishes sufficient conditions for the consistency of the estimators under $\alpha$-mixing dependence, as well as conditions for their asymptotic normality under $\beta$-mixing conditions. Section~\ref{section:verification} provides specific second-order conditions and explicit sequences which imply our results. Section~\ref{section:finite} evaluates the finite-sample performance and robustness of the estimator through simulation studies, while Section~\ref{section:real} applies our methodology to extreme event modelling in the oil industry. Finally, Section~\ref{section:conclusion} concludes. Proofs of all results are provided in a series of appendices, as well as a additional results regarding anti-clustering conditions for certain models, kernel density estimation for dependent data, and mixing and functional limit theorems.

\section{The estimators}\label{section:estimators}

In this section we introduce estimators for the conditional tail function and for the conditional Hill estimator, the latter having the same structure as in independent settings. It is only when proving its asymptotics that the temporal dependence of the data comes into play.

\subsection{Notation and setting} 
Let $\left\{\left(X_{j}, Y_{j}\right), j \in  \mathbb{Z} \right\}$ be a stationary time series defined on the probability space $\left(\Omega, \P , \mathcal{F}\right)$. Here, $\left\{Y_{j}, j \in  \mathbb{Z} \right\}$ is a positive process of interest and $\left\{X_{j}, j \in  \mathbb{Z} \right\}$ is a process of covariates, where $X_{j}, Y_{j} \in \mathbb{R}$, for each $j \in \mathbb{Z}$. Let the distribution function of $Y_{0}$ be denoted by $F$, the density of $X_0$ be denoted by $g$, which is assumed twice continuously differentiable, the joint density $g_{X_i, X_j}$ of $(X_i, X_j)$ exists, and is bounded for all $i,j \in \Z$, and let the distribution function of $Y_{0}$ given $X_{0}=x$ be denoted by $F^{x}$. 

Let $x\in\R$ with $g(x)>0$ be given. Our key assumption is that the conditional distribution of $Y_{0}$, given $X_{0} = x$, is regularly varying with tail index $1/\gamma(x) > 0 $, and that this holds in a neighbourhood around $x$, that is
\begin{align}\label{def:condRV}
    \overline{F}^{u}(y) =
    \P (Y_0 >   y \vert X_0 = u) = 
     L^u(y) \ y^{-\frac{1}{\gamma(u)}},\quad y>0,
\end{align}
for all $u$ in an open set $U\subseteq\mathbb{R}$ containing $x$, and where $L^{u}$ is slowly varying at infinity, that is $L^u(tv)/L^u(v)\to1$ as $v\to\infty$ for any $t>0$.

Let $K$ be a bounded density, denoted a kernel, supported on $[-1,1]$ and denote by $\left\{h_{n}, n \in \N \right\}$ the corresponding bandwidth (scale) parameter. Let $\left\{u^{x}_{n}, n \in \N\right\}$ be a scaling sequence diverging to infinity. 
The conditional survival function $ \overline{F}^{x}$ is estimated with the Nadaraya--Watson pathwise estimator as 
\begin{align*}
    \overline{F}^{x}_{n}(y) = \sum_{j = 1}^{n}\frac{K\left(\frac{x-X_j}{h_n}\right)}{ \sum_{j = 1}^{n}K\left(\frac{x-X_j}{h_n}\right)} 1_{\left\{Y_j > y\right\}},\quad y>0.
\end{align*}
Let the corresponding tail empirical distribution function, for deterministic threshold, be given by
\begin{align*}
    \widetilde{T}^{x}_{n}(s) =  
    \frac{ \overline{F}^{x}_{n}(su_{n}^{x})}
    { \overline{F}^{x}(u_{n}^{x})} = 
    \frac{\frac{1}{nh_{n} \overline{F}^{x}(u_{n}^{x})}
    \sum_{j=1}^{n}K\left(\frac{x-X_j}{h_n}\right)1_{\curly{Y_j > su_{n}^{x}}}}{\frac{1}{n h_n}\sum_{j=1}^{n}K\left(\frac{x-X_j}{h_n}\right)} =: 
    \frac{\widetilde{V}^{x}_{n}(s)}{g_{n}(x)},
\end{align*}
for $s >0$. The target of $\widetilde{T}^{x}_{n}$ is the conditional tail function $ s\mapsto T^{x}(s) = s^{-1/\gamma(x)}$. Indeed, this convergence, both in probability and weakly, will be shown under suitable conditions in this paper. However, the deterministic thresholds $u_{n}^{x}$ are not observable. Below we provide asymptotic results also for when these thresholds are estimated from data.

\subsection{Conditional tail function and conditional Hill estimators}
To transfer the tail function behaviour to the conditional tail index, the quantity of interest, the following identity is key:
\begin{align*}
    \gamma(x)  = \int_{1}^{\infty}s^{-1/\gamma(x)-1}\mathrm{d}s 
    = \int_{1}^{\infty}\frac{T^{x}(s)}{s}\mathrm{d}s.
\end{align*}
In this paper, we assume that an \emph{intermediate sequence} is given, that is a sequence $k_n\le n$ such that
\begin{align*}
k_{n} \to \infty, \ \ k_{n}/n \to 0, \ \ \text{as} \ n \to \infty.
\end{align*}
Let $u_{n}^{x} = F^{\leftarrow, x}\p{1-k_{n}/n}$ be an unobserved sequence of deterministic thresholds, diverging to infinity, where $F^{\leftarrow, x}$ denotes the generalised inverse of $F^{x}$. We may clearly estimate this sequence by its empirical counterpart $\left\{q_{n,k_{n}}^{x}, n \in \N \right\}$ given by $$q_{n,k_{n}}^{x} = F^{\leftarrow, x}_{n}\p{1-k_{n}/n},\quad n \in \N.$$ The conditional tail function is thus estimated by
\begin{align*}
    \widehat{T}^{x}_{n}(s) =  
    \frac{ \overline{F}^{x}_{n}(sq_{n,k_{n}}^{x})}
    { \overline{F}^{x}(u_{n}^{x})} 
    = \widetilde{T}^{x}_{n}\p{s\frac{q_{n,k_{n}}^{x}}{u_{n}^{x}}},
\end{align*}
which no longer depends on the sequence $u_{n}^{x}$, which can be seen from the definition of $\widetilde{T}^{x}_{n}$ and from the fact that $\overline{F}^{x}(u_{n}^{x})=k_n/n$.

Furthermore, we construct a conditional version of the Hill Estimator, given by
\begin{align*}
    \widehat{\gamma}_{n}\p{x} = 
    \int_{1}^{\infty}\frac{\widehat{T}_{n}^{x}(s)}{s}\mathrm{d}s
   =
     -\int_{1}^{\infty}\log\p{s}\widehat{T}_{n}^{x}(\mathrm{d}s)
   =
   \frac{n}{k_{n}} \frac{\sum_{j=1}^{n}K\p{\frac{x-X_{j}}{h_{n}}}\log_{+} \p{\frac{Y_j}{q_{n,k_{n}}^{x}}}}
   {\sum_{j=1}^{n}K\p{\frac{x-X_{j}}{h_{n}}}}.
\end{align*}
It is worth noticing that the effective sample size used in this estimator is proportional to $h_n k_n$ instead of $k_n$, which is also the rate for weak convergence. Consequently, the parameter $k_n$ has a different interpretation than in the unconditional case. In particular, $\widehat{\gamma}_{n}\p{x}$ can be plotted for the entire range $k_n\in\{2,3,\dots,n\}$.
\begin{example}
In case of $K$ being the uniform density on $[-1,1]$, then $\widehat{\gamma}_{n}\p{x}$ reduces to the Hill estimator applied to the subsample of the data having $X_i\in [x-h_n,x+h_n]$. Notice that the $k_n$ sequence no longer corresponds to the number of top order statistics in the subsample. Instead, let that number be defined as $\tilde k_n$. They are linked through the following relation:
$$k_n\sim\frac{n \tilde k_n}{\# \{X_i\in [x-h_n,x+h_n]\}}$$
\end{example}

\begin{figure}[hbt!]
\centering
\includegraphics[width=.75\textwidth,clip]{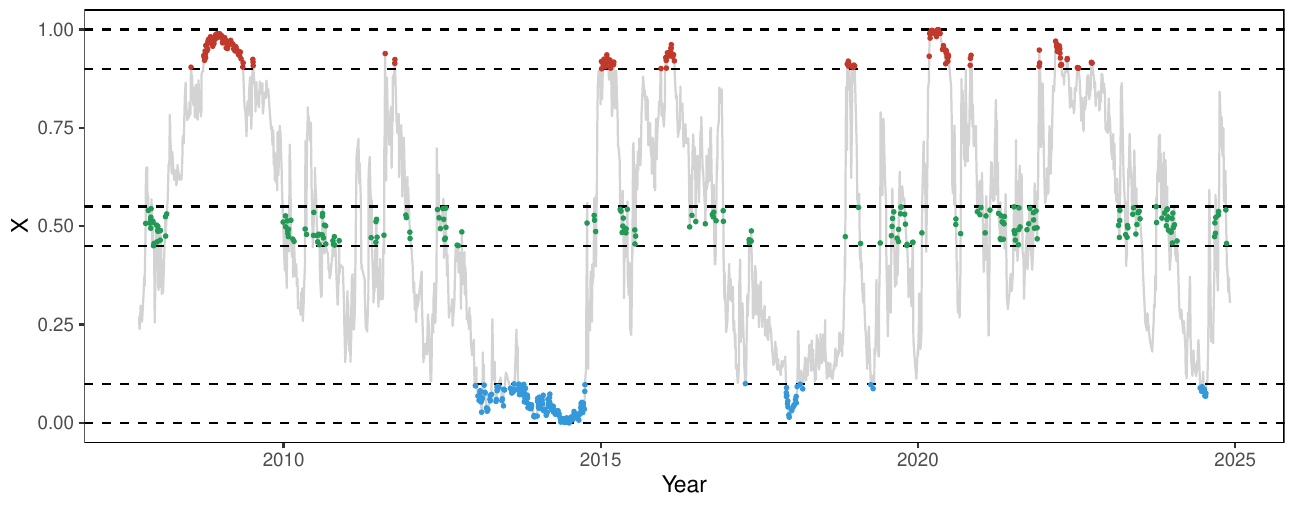}
\includegraphics[width=.75\textwidth,clip]{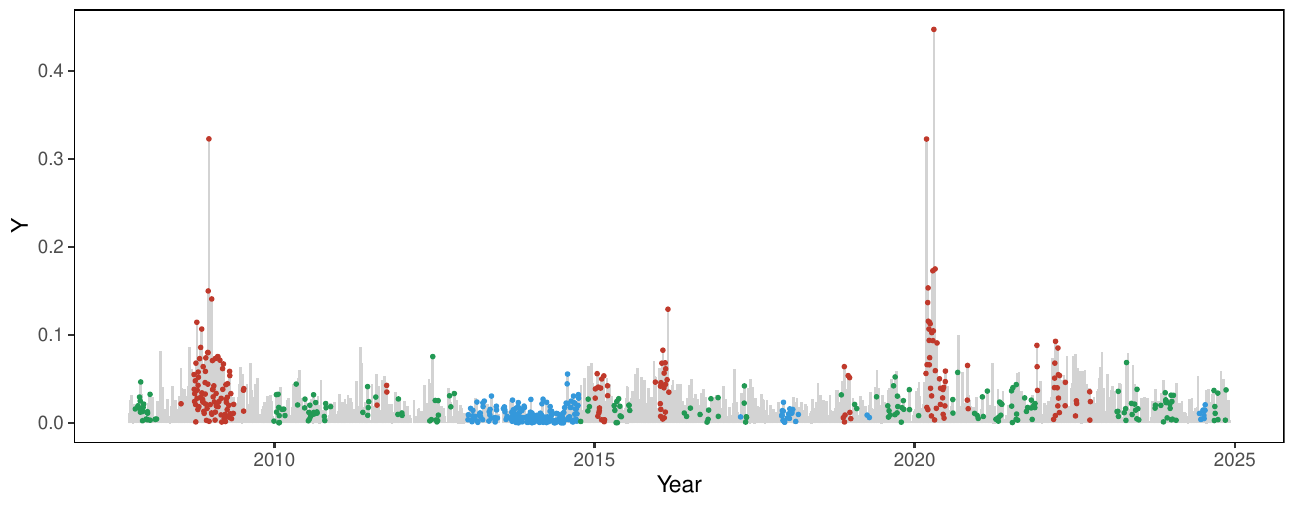}
\includegraphics[width=.75\textwidth,clip]{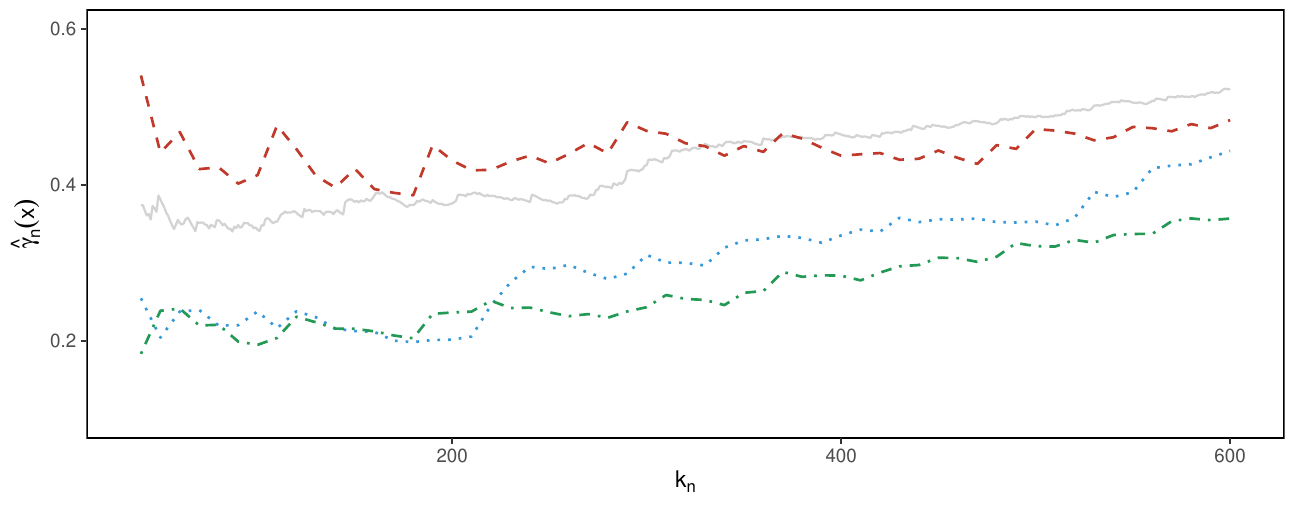}
\caption{Top panel: covariate process $\curly{X_j}$ given by the Crude Oil Volatility
Index (CBOE OVX). The effective sample is calculated using a uniform kernel centred at $x=0.05,\, 0.50,\, 0. 95$ (highlighted in red, green and blue, respectively) and with $h_n=0.05$. Middle panel: target process $\curly{Y_j}$ given by the absolute value of the negative log-returns of the West Texas Intermediate (WTI) grade crude oil spot prices with highlighted subsamples. Bottom panel: conditional Hill estimates $\widehat{\gamma}_{n}\p{x}$ in dotted blue, dashdotted green and dashed red, respectively; in solid grey we have the unconditional Hill estimator of the entire sample.}
\label{fig:simulatedts}
\end{figure}

Before delving into the mathematical treatment of our estimators, we provide an example of a conditional Hill plot. In Figure~\ref{fig:simulatedts} we have taken $\curly{Y_j}$ to be the (absolute value of) the negative log-returns of the West Texas Intermediate (WTI) grade crude oil spot prices from $2007$-$2024$. As a covariate $\curly{X_j}$, we take the Crude Oil Volatility
Index (CBOE OVX), transformed to have uniform marginals. For large $x$, we clearly observe a different, heavier tail regime. In contrast, for medium and small $x$, the tail regime is similar and lighter. Below, it's shown that the choice of kernel enters the asymptotic variance of the conditional Hill estimator through a certain integral. Transforming the covariate process to have uniform marginals is not necessary. It's nonetheless done here, in the simulations and real data analysis below to avoid sparse regions of data to ensure good practical performance.

\section{Main results}\label{section:main}

In this section we provide consistency and weak convergence results for the estimators of the conditional tail function and conditional tail index under suitable conditions. Most conditions are versions of their most general unconditional counterparts. Notably, only $\alpha$-mixing is required for consistency, while $\beta$-mixing is required for normality. Furthermore, second order conditions are formulated in a very general and abstract form. It is later seen that broad classes of regularly-varying tails satisfy them.  Unless otherwise specified, all statements below hold for every $x \in U$.

\subsection{Consistency}\label{section:consistency}

A fundamental underlying result is the uniform consistency of the conditional tail empirical function, for deterministic thresholds. For $m$-dependent sequences, this can be established under rather lax conditions, while for $\alpha$-mixing sequences we require additional dependence and speed conditions.

\begin{condition}\label{L_Lipschitz_assumption} The conditions for consistency are as follows.
\begin{enumerate}[label=\thecondition.\arabic*, ref=\thecondition.\arabic*]
    \item The map $z \mapsto 1/\gamma(z)$ is $c_{\gamma}$-Lipschitz continuous, \label{Assu:gamma_lip} 
    \item  $\lim_{n\to \infty}h_n\log(u_{n}^{x}) = 0$,                            \label{Assu:h_n_log_u_{n}^{x}}
    \item There exists $y_0 \in \R$ and $c_L>0$ such that                                 \label{Assu:L_lip}
    \begin{align*}
        \sup_{y \geq y_{0}}   \left| \frac{\log(L^{x}(y)) - \log(L^{u}(y))}{\log(y)} \right|
        \leq c_{L} \lvert x - u \rvert,
    \end{align*}
     uniformly in $x,u \in U$,
       \item  $\lim_{n\to \infty}n h_{n} \overline{F}^{x}(u_{n}^{x}) = \infty$,   
      \label{Assu:n_h_F_limit}
        \item It holds $g_{n}(x) \overset{\P}{\to} g(x)$, see Lemma~\ref{lemma:g_n_consistent_anti_clust},~\ref{lemma:g_n_consistent_alpha_mixing} or~\ref{lemma:g_n_consistent_m_dependent}, \label{Assu:g_n_consistent}
        \item For any $s,t > 0$ there exists a finite constant $M^x(s,t)$ not depending on $j$ such that \label{Assu:Expected_cross_indicator_limit}
        \begin{align*}
         \frac{\CE{1_{\left\{Y_0 > su_{n}^{x}\right\}}
      1_{\left\{Y_j > tu_{n}^{x}\right\}}}{X_{0} = x + h_{n}w_{1}, X_{j} = x + h_{n}w_{2}}}{\overline{F}^{x}(u_{n}^{x})} \to M^{x}(s,t),
        \end{align*}
         uniformly in $w_1, w_2 \in \left[-1,1\right]$.
\end{enumerate}

\end{condition}

\begin{theorem} \label{Empirical_tail_dist_consistency_theorem_m_depence}
    Assume~\eqref{def:condRV} and that conditions~\ref{Assu:gamma_lip}-\ref{Assu:g_n_consistent} hold and that the time series $\left\{\left(X_{j}, Y_{j}\right), j \in  \mathbb{Z} \right\}$ is stationary and m-dependent. Then
    \begin{align}\label{Empirical_tail_dist_consistency_equation}
        \widetilde{T}^{x}_{n}(s) \overset{\P}{\to} T^{x}(s) = s^{-1/\gamma(x)},\ s >0,
    \end{align}
    the convergence being uniform on compacts bounded away from zero.
\end{theorem}
\begin{theorem} \label{Empirical_tail_dist_consistency_theorem}
    Assume~\eqref{def:condRV}, condition~\ref{L_Lipschitz_assumption} and that the time series $\left\{\left(X_{j}, Y_{j}\right), j \in  \mathbb{Z} \right\}$ is stationary and $\alpha$-mixing with $\alpha(j) = \mathcal{O}(j^{-\eta})$, for $\eta > 2$. \\
     If there exists $c_{x} \in (0, \infty)$ such that
        \begin{align}\label{Assu:h_F_speed_relation}
            \frac{\overline{F}^{x}(u_{n}^{x}) + h_{n}^{c_{x}\eta -2}}{n h_{n}^{c_{x}} (\overline{F}^{x}(u_{n}^{x}))^{2}} \to 0,
        \end{align}
    then~\eqref{Empirical_tail_dist_consistency_equation} holds uniformly on compacts bounded away from zero.
\end{theorem}
As a first application we get that consistency of the conditional tail empirical function implies that the intermediate empirical quantile $q_{n,k_{n}}^{x}$ is equivalent to the deterministic $u_{n}^{x}$. This is of practical importance, since the former is estimable from data. Consistency of the conditional Hill estimator then directly follows.
\begin{proposition} \label{Empirical_quantile_consistency_prop}
    If~\eqref{Empirical_tail_dist_consistency_equation} holds, then $q_{n,k_{n}}^{x}/u_{n}^{x} \overset{\P}{\to} 1.$
\end{proposition}
\begin{theorem}\label{Hill_Consistency_thm}
    Assume conditions~\ref{Assu:gamma_lip}-\ref{Assu:L_lip},~\ref{Assu:g_n_consistent} and~\eqref{Empirical_tail_dist_consistency_equation}. Then $ \widehat{\gamma}_{n}\p{x} \overset{\P}{\to} \gamma(x)$.
\end{theorem}
\subsection{Central limit theorems}
We now consider the univariate conditional tail empirical process for deterministic thresholds
\begin{align*}
    \widetilde{\mathbb{T}}^{x}_{n}(s) =  \sqrt{k_{n}h_{n}}\curly{\widetilde{T}^{x}_{n}(s) - T^{x}(s)}, \ s > 0,
\end{align*}
for which a functional central limit is established in the case of $\beta$-mixing. The fundamental tool to achieve this goal is the blocking method. Let $\curly{r_n, n \in \N}$ be an intermediate sequence of block sizes and set $m_n = \floor{n/r_n}$ which is the number of complete blocks in the sample. 

Notice that the speed of convergence is determined by the effective sample size $nh_{n}\overline{F}^{x}(u_{n}^{x})=h_nk_n$, which is slower than in the unconditional case, which is simply $k_n$. This price to pay, however, results in a variance which agrees with the i.i.d. conditional case, in contrast to the more involved expression for unconditional time series provided as an infinite sum, e.g. in~\citet[Eq. (9.5.4)]{Kulik2020}.

For $s_{0} \in (0, 1)$, define 
\begin{align*}
    B_{n}^{x}(s_{0}) = \sup_{t \geq s_{0}}\left\vert   \E{\widetilde{V}^{x}_{n}(t)}  - t^{-1/\gamma(x)}g(x)  \right\vert.
\end{align*}
\begin{condition} \label{Normality_condition}
The conditions for normality are as follows.
\begin{enumerate}[label=\thecondition.\arabic*, ref=\thecondition.\arabic*]
    \item  \label{rate_condition}
    $\lim_{n\to \infty} r_{n} h_{n} \overline{F}^{x}(u_{n}^{x}) = 0$,
    \item For all $s,t > 0$,  \label{anti_clustering_condition}
    \begin{align*}
       \lim_{m\to \infty}
       \limsup_{n\to \infty} \sum_{j=m}^{r_n} 
       \frac{\E{K\left(\frac{x-X_0}{h_n}\right)1_{\left\{Y_0 > su_{n}^{x}\right\}}
       K\left(\frac{x-X_j}{h_n}\right)1_{\left\{Y_j > tu_{n}^{x}\right\}}}}
       {h_{n} \overline{F}^{x}(u_{n}^{x})} = 0,
    \end{align*}
    \item For the stationary time series $\left\{\left(X_{j}, Y_{j}\right), j \in  \mathbb{Z} \right\}$ with $\beta$-mixing coefficients $\curly{\beta_n}$, there exists an intermediate sequence $\curly{\ell_n}$ such that  \label{beta_mixing_condition} 
    \begin{align*}
        \lim_{n \to \infty} \frac{1}{\ell_n} =
        \lim_{n \to \infty} \frac{\ell_n}{r_n} = 
        \lim_{n \to \infty} \frac{r_n}{n} = 
        \lim_{n \to \infty} \frac{n}{r_n}\beta_{\ell_n} = 0,
        \end{align*}
    \item It holds that $  \sqrt{k_{n}h_{n}}\curly{g(x)- g_{n}(x)} \overset{\P }{\to} 0 $, see Lemma~\ref{lemma:g_n_emp_process_negligibility_anti_cl} or~\ref{lemma:g_n_emp_process_negligibility_alpha_mixing}, \label{Assu:gn_emp_proc_negl}
    \item There exists $s_{0} \in (0,1)$, such that
$\sqrt{k_{n}h_{n}}B_{n}^{x}(s_{0})\to 0$, \label{B_n_assumption}
    \item For all $A > 0$, \label{log_anti_clustering_condition}
    \begin{align*}
       \lim_{m\to \infty}
       \limsup_{n\to \infty} \sum_{j=m}^{r_n} 
         \frac{\E{K\left(\frac{x-X_{0}}{h_n}\right) K\left(\frac{x-X_{j}}{h_n}\right) 
         \log_{+}\left(\frac{Y_{0}}{u_{n}^{x}A} \right)
         \log_{+}\left(\frac{Y_{j}}{u_{n}^{x}A} \right)}}{h_{n}\overline{F}^{x}(u_{n}^{x})}
           = 0,
    \end{align*}
    \item  It holds that\label{integral_B_n_assumption}
    \begin{align*}
        \lim_{n\to \infty}
        \sqrt{k_{n}h_{n}}\int_{1}^{\infty} \frac{ \E{\widetilde{V}^{x}_{n}(t)}  - t^{-1/\gamma(x)}g(x)}{t} \mathrm{d}t = 0.
    \end{align*}
\end{enumerate}
    
\end{condition}

\begin{proposition} \label{Functional_CLT_Tail_empirical_process} 
    Assume~\eqref{def:condRV}, conditions~\ref{Assu:gamma_lip}-\ref{Assu:g_n_consistent} and~\ref{rate_condition}-\ref{B_n_assumption}.
    Then it holds that $ \widetilde{\mathbb{T}}^{x}_{n}$ converges weakly to a centred Gaussian process $\mathbb{T}^{x}$ in $\mathbb{D}\p{[s_0, \infty)}$ endowed with the $J_1$-topology, with covariance function
    \begin{align*} 
        \Gamma(s,t) = \frac{\p{s \vee t}^{-1/\gamma(x)}}{g(x)}\int  K^{2}(u)\mathrm{d}u,
    \end{align*}
    and the process ${\mathbb{T}}^{x}$ has an almost surely continuous version.
\end{proposition}

\begin{proposition}\label{prop:q_{n,k_{n}}^{x}_normality}
    Assume~\eqref{def:condRV}, conditions~\ref{Assu:gamma_lip}-\ref{Assu:g_n_consistent} and~\ref{rate_condition}-\ref{B_n_assumption}.
    Then
    \begin{align*}
         \sqrt{h_{n}k_{n}}\curly{\frac{q_{n,k_{n}}^{x}}{u_{n}^{x}} - 1} 
         \overset{d}{\to}
         \gamma(x)\mathbb{T}^{x}(1), 
    \end{align*}
    where 
    \begin{align*}
        \gamma(x)\mathbb{T}^{x}(1) \overset{d}{=}  \text{N}\p{0, \frac{\gamma(x)^{2}}{g(x)}\int  K^{2}(u)\mathrm{d}u}.
    \end{align*}
\end{proposition}
The sequence $\{u_{n}^{x}\}$ is unobserved and we may estimate it with $\{q_{n,k_{n}}^{x}\}$. Let
\begin{align*}
    &\widehat{\mathbb{T}}^{x}_{n}(s) =  
    \sqrt{k_{n}h_{n}}\curly{\widehat{T}^{x}_{n}(s) - T^{x}(s)},
\end{align*}
which we refer to as the tail empirical process with random levels.
\begin{theorem}\label{Thm:Functional_CLT_Tail_empirical_process_Random_level}
    Assume~\eqref{def:condRV}, conditions~\ref{Assu:gamma_lip}-\ref{Assu:g_n_consistent} and~\ref{rate_condition}-\ref{B_n_assumption}.
    Then
    \begin{align*}
        \widehat{\mathbb{T}}^{x}_{n} 
        \overset{w}{\Longrightarrow} \mathbb{T}^{x} - \mathbb{T}^{x}(1) \times T^{x},
    \end{align*}
    in $\mathbb{D}\p{[s_0, \infty)}$ endowed with the $J_{1}$ topology.
\end{theorem}

It turns out that Theorem~\ref{Thm:Functional_CLT_Tail_empirical_process_Random_level} is not required to establish convergence of the conditional Hill estimator. However, it is an interesting result in its own right, since the empirical tail process can be used to measure how closely the top order statistics behave like a strict Pareto tail. Generalisations to multiple dimensions are also possible through the same proof techniques. 

\begin{theorem}\label{Hill_normality_thm}
   Assume~\eqref{def:condRV}, that conditions~\ref{Assu:gamma_lip}-\ref{Assu:g_n_consistent} and~\ref{Normality_condition} hold.
   Then it holds that 
    \begin{align*}
        \sqrt{k_{n}h_{n}}\p{\widehat{\gamma}_{n}\p{x} - \gamma(x)} 
        \overset{d}{\to} \text{N}\p{0, \frac{\gamma(x)^{2}}{g(x)}\int  K^{2}(u)\mathrm{d}u}.
    \end{align*}
\end{theorem}
\begin{remark}
Examining closely the proof techniques of Theorems~\ref{Hill_Consistency_thm} and~\ref{Hill_normality_thm}, we see that our results can be extended to functionals of the tail empirical process of the form
$$\widehat\Phi_n(x)=\int\phi\p{s}\widehat{T}_{n}^{x}(\mathrm{d}s),$$
provided conditions involving the logarithmic function are replaced with suitably modified versions involving $\phi$.  For instance, the choice $\phi(s) = \log(s)^{2}$, yields an estimator of $2\gamma(x)^2$. In case the functional
$z\mapsto \int \phi \dd y$, or any other functional is defined on $\mathbb{D}\p{[A,B]}$ and is sup-norm continuous or Hadamard differentiable, consistency and a central limit theorem for the associated estimator follow directly from Theorem~\ref{Thm:Functional_CLT_Tail_empirical_process_Random_level}, together with the continuous mapping theorem or functional delta-method, respectively.
\end{remark}

\begin{remark}
    The bias of the conditional Hill estimator vanishes due to condition~\ref{integral_B_n_assumption}. If this condition is relaxed to 
     \begin{align*}
        \lim_{n\to \infty}
        \sqrt{k_{n}h_{n}}\int_{1}^{\infty} \frac{ \E{\widetilde{V}^{x}_{n}(t)}  - t^{-1/\gamma(x)}g(x)}{t} \mathrm{d}t \ = \  B \in \R,
    \end{align*}
    the asymptotic variance in Theorem~\ref{Hill_normality_thm} remains similar, while the bias becomes $B/g(x)$. 
    % In finite samples, this bias would simply be estimated by
    % \begin{align*}
    %     \frac{\sqrt{k_{n}h_{n}}}{g_n(x)}\int_{1}^{\infty} \frac{ \widetilde{V}^{x}_{n}(t)  - t^{-1/\widehat{\gamma}_{n}(x)}g_n(x)}{t} \mathrm{d}t=
    %     \frac{1}{\sqrt{k_{n}h_{n}}}\int_{1}^{\infty} \frac{
    % \sum_{j=1}^{n}K\left(\frac{x-X_j}{h_n}\right)1_{\curly{Y_j > tu_{n}^{x}}}  - t^{-1/\widehat{\gamma}_{n}(x)}g_n(x)}{tg_{n}(x)} \mathrm{d}t.
    % \end{align*}
\end{remark}
\section{Second order behaviour and specification of sequences}\label{section:verification}
In this section we investigate the fulfilment of the abstract second order conditions by considering certain classes of regularly-varying tails. Moreover, there are several deterministic sequences involved in the above results, which have to satisfy speed assumptions. We also verify that it is always possible to find such sequences for the proposed classes.

\subsection{Consistency}
In subsection~\ref{section:consistency} we presented two theorems that yield consistency of $\widetilde{T}^{x}_{n}$, the first one under $m$-dependence, and the second one assuming $\alpha$-mixing and additionally~\eqref{Assu:h_F_speed_relation}. In the latter case there are three limit assumptions related to the speed of deterministic sequences needed for consistency, which are given by
\begin{enumerate}
    \item $h_{n}\log(u_{n}^{x}) = h_{n}\log( F^{\leftarrow, x}\p{1-k_{n}/n}) \to 0$, (\ref{Assu:h_n_log_u_{n}^{x}}) 
    \item $n h_{n} \overline{F}^{x}(u_{n}^{x}) =k_nh_n\to \infty$, (\ref{Assu:n_h_F_limit})
    \item $ (\overline{F}^{x}(u_{n}^{x}) + h_{n}^{c_{x}\eta -2})/(n h_{n}^{c_{x}} (\overline{F}^{x}(u_{n}^{x}))^{2})
    =
     (k_{n} + nh_{n}^{c_{x}\eta -2}) / (h_{n}^{c_{x}} k_{n}^{2})
     \to 0$,~\eqref{Assu:h_F_speed_relation}.
\end{enumerate}
Notice that, in the case of $m$-dependence, only the first two conditions are required. In any case, the conditions are directly on the conditional distribution, so the results below do not make use of the time-series structure of the data, but only on the (first-order) regular variation property of the conditional tail.\\
In what follows, we use the notation $f_n  \asymp g_{n}$, when $f_n  \sim c g_{n}$, for some constant $c>0$.

\begin{theorem}\label{thm:sequence_cons_power}
    Assume~\eqref{def:condRV}.
     If $k_{n} \asymp n^{\delta}$ for $\delta \in ((\eta + 1)^{-1}, 1)$, $h_{n} \asymp \overline{F}^{x}(u_{n}^{x})^{e} $ for $e \in (0, \delta/(1-\delta)\wedge (\delta(\eta + 1) -1)/(2(1-\delta)))$, there exists a  $c_{x} \in (0,\infty)$ such that the above limit conditions are satisfied.
\end{theorem}
\begin{theorem}\label{thm:sequence_cons_m_dependence}
    Assume~\eqref{def:condRV}. 
    If $h_{n} \asymp \log(n)^{-e} $ for $e \in (4/3, 2)$ and $k_{n}  \asymp \log(n)^{2}$, then the first two above limit conditions are satisfied.
\end{theorem}

\subsection{Asymptotic normality}
Asymptotic normality of the tail empirical process $\widetilde{\mathbb{T}}_{n}^{x}$ and of the conditional Hill estimator $\widehat{\gamma}_{n}(x)$ rely on two crucial, yet abstract conditions~\ref{B_n_assumption} and~\ref{integral_B_n_assumption} respectively, related to the regular variation bias. In this subsection we propose quite general second-order-type conditions, under which we can provide examples of sequences that satisfy the speed restrictions required for asymptotic normality.

\begin{lemma}\label{lemma_second_order_conditions}
 Assume conditions~\ref{Assu:gamma_lip}-\ref{Assu:n_h_F_limit}, that $F^{x}$ is a continuous distribution and that there exist $c_{x},  \gamma(x) >0,$ and $\zeta_{x}\gamma(x) >  1$ such that
    \begin{align}\label{cond_second_order}
        \sup_{t\geq 1} t^{\zeta_{x}}\left\vert \overline{F}^{x}(t)-c_{x}t^{-1/\gamma(x)} \right\vert < \infty.
    \end{align}
    If there exist an intermediate sequence $\left\{k_{n} \right\}$ and a bandwidth sequence $\left\{h_{n} \right\}$ such that
    \begin{align}\label{eq:second_order_derived_sequences}
         \lim_{n \to \infty}\sqrt{k_{n}h_{n}^{3}}
         \log\left({n}/{k_{n}}\right)
          =
            \lim_{n \to \infty} \sqrt{h_{n}}
         k_{n}^{\zeta_{x}\gamma(x) - \frac{1}{2}}n^{1-\zeta_{x}\gamma(x)} = 0,
    \end{align}
    then conditions~\ref{B_n_assumption} and~\ref{integral_B_n_assumption} hold.
\end{lemma}
\begin{lemma}\label{lemma_log_second_order_conditions}
     Assume conditions~\ref{Assu:gamma_lip}-\ref{Assu:n_h_F_limit}, 
      that $F^{x}$ is a continuous distribution and that there exist $c_{x}, \gamma(x) >0,$ and $\zeta_{x}\gamma(x) >  1$ such that
\begin{align}\label{cond_second_order_log}
        \sup_{t\geq 1} t^{\zeta_{x}}\left\vert \overline{F}^{x}(t)-c_{x}\log(t)t^{-\gamma(x)} \right\vert < \infty.
    \end{align}
    If there exist an intermediate sequence $\left\{k_{n} \right\}$ and a bandwidth sequence $\left\{h_{n} \right\}$ such that
    \begin{align}\label{eq:second_order_log_derived_sequences}
         \lim_{n \to \infty}\sqrt{k_{n}h_{n}^{3}}
         \log\left({n}/{k_{n}}\right)
          =
          \lim_{n \to \infty}
        \frac{\sqrt{k_{n}h_{n}}}
       {\log\left({n}/{k_n}\right)}
       = 0,
    \end{align}
    then conditions~\ref{B_n_assumption} and~\ref{integral_B_n_assumption} hold.
\end{lemma}
\begin{lemma}\label{lemma::anti_clustering_sequence}
     Assume conditions~\ref{Assu:gamma_lip}-\ref{Assu:n_h_F_limit},~\ref{Assu:Expected_cross_indicator_limit},~\ref{rate_condition}, the time series $\left\{\left(X_{j}, Y_{j}\right), j \in  \mathbb{Z} \right\}$ is stationary and $\alpha$-mixing with $\alpha(j) = \mathcal{O}(j^{-\eta})$, for $\eta > 2$, and that for $A>0$ there exists a finite constant $\widetilde{M}^x$ not depending on $j$ such that
        \begin{align}\label{eq::Expected_cross_log_limit}
         \frac{\CE{ \log_{+}\left(\frac{Y_{0}}{u_{n}^{x}A} \right)
         \log_{+}\left(\frac{Y_{j}}{u_{n}^{x}A} \right)}{X_{0} = x + h_{n}w_{1}, X_{j} = x + h_{n}w_{2}}}{\overline{F}^{x}(u_{n}^{x})} \to \widetilde{M}^{x},
        \end{align}
         uniformly in $w_1, w_2 \in \left[-1,1\right]$. If there exist a bandwidth sequence $\left\{h_{n} \right\}$ and constants $c_{x}, c'_{x} > 0$ and $p>1$, such that
    \begin{align}\label{eq:sequence_anti_clustering_restrains}
          \lim_{n \to \infty}\left(h_{n}^{1-c_{x}} + h_{n}^{c_{x}(\eta-1) -1}\overline{F}^{x}(u_{n}^{x})^{-1}\right)
          =
         \lim_{n \to \infty}\left(
         h_n^{1-c'_x} + h_{n}^{-1/p + c'_{x}(\eta/p -1)}\overline{F}^{x}(u_{n}^{x})^{-1/p} \right)
       = 0,
    \end{align}
    then the anti-clustering conditions~\ref{anti_clustering_condition} and~\ref{log_anti_clustering_condition} hold.
\end{lemma}

In essence, Lemmas~\ref{lemma_second_order_conditions} and~\ref{lemma_log_second_order_conditions} transfer the asymptotic bias conditions into additional speed conditions on the $h_n$ and $k_n$ sequences. Standard second-order conditions in extremes are often formulated in terms of limits which imply~\eqref{cond_second_order}; in particular the broad Hall class~\citep{HallWelsh1985} falls into this framework. Second-order conditions akin to~\eqref{lemma_log_second_order_conditions} are not often encountered in the literature, and our result is inspired by~\cite[Eq. (9.6.3)]{Kulik2020}. The result of Lemma~\ref{lemma::anti_clustering_sequence} provides an alternative way of verifying the anti-clustering conditions. For convenience, we now restate the main speed conditions for asymptotic normality of the conditional Hill estimator to hold.

In addition to the assumptions needed for consistency, we require sequences $\left\{(k_{n}, h_{n}, r_{n}, l_{n}), n \in \N \right\}$ such that the following conditions are met:

\begin{enumerate}[itemsep=1.5mm]
    \item $r_{n}h_{n}\overline{F}^{x}(u_{n}^{x}) = r_{n}h_{n}k_{n}/n \to  0$, (\ref{rate_condition})
    \item $\lim_{n \to \infty} 1/\ell_n =
        \lim_{n \to \infty} \ell_n/r_n = 
        \lim_{n \to \infty} r_n/n = 
        \lim_{n \to \infty} (n/r_n)\beta_{\ell_n} = 0$, (\ref{beta_mixing_condition})
    \item $\sqrt{k_{n}h_{n}}\curly{g(x)- g_{n}(x)} \overset{\P }{\to} 0$, (\ref{Assu:gn_emp_proc_negl})  
    \item $\sqrt{k_{n}h_{n}}B_{n}^{x}(s_{0}) \to 0$,  (\ref{B_n_assumption})  
    \item $\sqrt{k_{n}h_{n}}\int_{1}^{\infty} (\mathbb{E}[\widetilde{V}^{x}_{n}(t)]  - t^{-1/\gamma(x)}g(x))/t \  \mathrm{d}t \to 0$, (\ref{integral_B_n_assumption}).
\end{enumerate}

\begin{theorem}\label{thm:existence_of_sequences}
    Assume~\eqref{def:condRV} and that $\left\{  (X_{j},Y_{j}), \ j\in \Z \right\}$ is a stationary, $\beta$-mixing time series, with $\alpha$-mixing coefficients $\alpha(j) = \mathcal{O}(j^{-\eta})$, for $\eta > 2$, $\beta$-mixing coefficients $\beta(j) = \mathcal{O}(j^{-\nu})$ for $\nu > 0$, and that $F^{x}$ is continuous.
    Assume conditions~\ref{Assu:gamma_lip},~\ref{Assu:L_lip},~\ref{Assu:Expected_cross_indicator_limit},~\eqref{eq::Expected_cross_log_limit}, and the second-order condition~\eqref{cond_second_order}.
    Then there exist sequences $\left\{(k_{n}, h_{n}, r_{n}, \ell_{n}), n \in \N \right\}$ such that the first two conditions for consistency, all the above conditions for normality and anti-clustering conditions~\ref{anti_clustering_condition} and~\ref{log_anti_clustering_condition} are satisfied.
\end{theorem}

\begin{theorem}\label{thm:existence_of_sequences_log}
    Assume~\eqref{def:condRV} and that $\left\{  (X_{j},Y_{j}), \ j\in \Z \right\}$ is a stationary, $\beta$-mixing time series, with $\alpha$-mixing coefficients $\alpha(j) = \mathcal{O}(j^{-\eta})$, for $\eta > 2$,  $\beta$-mixing coefficients $\beta(j) = \mathcal{O}(j^{-\nu})$ for $\nu > 0$, and that $F^{x}$ is continuous.
    Assume conditions~\ref{Assu:gamma_lip} and~\ref{Assu:L_lip} and the second-order condition~\eqref{cond_second_order_log}.
    Then there exist sequences $\left\{(k_{n}, h_{n}, r_{n}, \ell_{n}), n \in \N \right\}$ such that the first two conditions for consistency and all the above conditions for normality are satisfied.
\end{theorem}
Note that in both of the above theorems, only the first two conditions for consistency are required, by $\beta$-mixing. In the simulation study below, we present three non-trivial examples of processes where the above speed conditions and anti-clustering conditions are satisfied.

\begin{remark}
    In the proof of Theorem~\ref{thm:existence_of_sequences}, we let $k_{n} \asymp n^{\delta}$, where 
    \begin{align*}
         \delta \in \left(0, \ \frac{3}{2}\frac{\nu}{1+\nu} \wedge \frac{2(\zeta_{x} -1/\gamma(x))}{2\zeta_{x} - 1/\gamma(x)}\right).
    \end{align*}
    The restriction $\delta < 3\nu/(2(1+\nu))$ is a consequence of allowing the $\beta$-mixing coefficient to decay polynomially of any order. The more common and stronger assumption of exponential decay i.e. $\beta(j) = \mathcal{O}(\rho^{j})$ for $\rho \in (0,1)$ removes this restriction. On the other hand, the restriction $\delta < 2(\zeta_{x} - 1/\gamma(x))/(2\zeta_{x} - 1/\gamma(x))$ is a direct consequence of the required rates in Lemma~\ref{lemma_second_order_conditions} that arise from the second order assumption~\eqref{cond_second_order}.
    
    Also notice that the rate of convergence in Theorem~\ref{Hill_normality_thm} is $k_{n}h_{n}$. In the setting of Theorem~\ref{thm:existence_of_sequences}, this rate reduces to
    \begin{align*}
        k_{n}h_{n} \asymp n^{\delta - (1-\delta)e},
    \end{align*}
    where $\delta$ is in the above interval and $e$ is bounded from below by $\delta/(3(1-\delta))$. To optimize this rate, $\delta$ should be picked as large as possible and $e$ should be as small as possible. For $e$ close to its lower bound, $k_{n}h_{n}$ behaves like $n^{2\delta/3}$.  If we consider only the case of exponential decay of the $\beta$-mixing coefficients, the optimal rate is then 
    \begin{align*}
        k_{n}h_{n} = o\Big(n^{\frac{2}{3}\frac{2(\zeta_{x} -1/\gamma(x))}{2\zeta_{x} - 1/\gamma(x)}}\Big).
    \end{align*}
     It resembles the rate achieved in the unconditional case under a corresponding second order assumption, $k_n = o(n^{\rho})$ where $\rho = 2(\zeta_{x} - 1/\gamma(x))/(2\zeta_{x} - 1/\gamma(x))$ (see~\cite{Kulik2020}, problem 9.6), however the fraction $2/3$ in the exponent is caused by the bandwidth sequence slowing the speed of convergence.
\end{remark}
 The preprint~\cite{GoegebeurGuillou2024} obtains similar asymptotic rates for the conditional Hill estimator. We now compare their main assumptions.
Their kernel assumptions $(\mathcal{K})$ are similar to ours; they also cover the multidimensional covariate case. They throughout assume a second order condition $(\mathcal{D})$, which is seen to be a special case of equation~\eqref{cond_second_order}; they do not cover the second order condition of equation~\eqref{cond_second_order_log}.
Their assumption $(\mathcal{J})$ implies our anti-clustering condition~\ref{anti_clustering_condition}. Their strong condition $(\mathcal{X})$ on the joint behaviour of $(X_0, X_i)$ is here replaced by simply $g_{X_0, X_i}$ being bounded for every $i$. Their condition (A) is similar to our beta-mixing condition~\ref{beta_mixing_condition} but it depends on an involved probability sequence $\{\nu_n\}$. Their condition (B) implies our anti-clustering condition~\ref{log_anti_clustering_condition}. In their Theorem 3.2 and 3.3, the first bandwidth condition corresponds to condition~\ref{Assu:h_n_log_u_{n}^{x}}, the second to equation~\eqref{Condition:sqrt_n_h3_F}, the third and fourth to the first and second part of equation~\eqref{eq:second_order_derived_sequences}. 

\section{Finite-sample behaviour}\label{section:finite}
In this section we investigate the finite-sample behaviour of the conditional Hill estimator. We begin by discussing the selection of the bandwidth parameter for a fixed sample size. Then we perform some simulation studies where we examine how different intertemporal dependence structures affect the performance of the estimator.

\subsection{Bandwidth selection}
In this section we present methods for bandwidth selection, which is one of the main parameters above subject to tuning. The choice of bandwidth should, according to the theorems above, obey certain speed constraints, and the remark in the former section yields that it should depend on the threshold sequence $\{k_n\}$. Below we present four selection approaches, two of which are $k_n$-dependent and two of which are not. The proposed techniques are assessed in the subsequent subsection.
\begin{enumerate}
    \item  Sheather-Jones bandwidth fitted only on the covariate process (see~\cite{Sheather_Jones}). It's a plug-in method that provides a global bandwidth, independent of $k_n$. 
    \item  Sheather-Jones bandwidth fitted only on the top $k_n$ values of $\curly{X_j}$, sorted according to the order statistics of $\curly{Y_j}$.
    \item Various approaches concerning non-parametric regression (see~\citet{Daouia2010,Gannoun2002}
    ) have proposed a purely data-driven, cross-validation bandwidth selection method 
    \begin{align}\label{eq_h_hat}
        \hat{h} = \argmin_{h>0} \sum_{i=1}^{n}\sum_{j=1}^{n} \left(1_{\{Y_{i} > Y_{j}\}} - \overline{F}_{n, -i}^{X_{i}}(Y_j)\right)^{2},
    \end{align}
    where $\overline{F}_{n, -i}$ is the leave-one-out Nadaraya--Watson estimator of the conditional survival function given by
    \begin{align*}
        \overline{F}_{n, -i}^{x}(y)
        =
        \frac{
        \sum_{j=1, j\neq i}^{n}K\left(\frac{x-X_j}{h_n}\right)1_{\curly{Y_j > y}}}
        {\sum_{j=1,  j\neq i}^{n}K\left(\frac{x-X_j}{h_n}\right)}.
    \end{align*}
    Alternatively, as proposed by~\cite{Li_Racine}, a leave-one-out doubly-smoothed Nadaraya--Watson estimator of the conditional survival function, where also $\curly{Y_j}$ is smoothed with distribution function $G$, using bandwidth $b_n$ is given by
    \begin{align}\label{eq:F_bar_double_smoothed}
        \overline{F}_{n, -i}^{x}(y)
        =
        \frac{
        \sum_{j=1, j\neq i}^{n}K\left(\frac{x-X_j}{h_n}\right) G\left(\frac{y-Y_j}{b_n}\right)}
        {\sum_{j=1,  j\neq i}^{n}K\left(\frac{x-X_j}{h_n}\right)},
    \end{align}
    where $G(v) = \int_{(-\infty, v]} w(u) \dd u$, for symmetric, bounded and compactly supported density $w$. In the numeric studies below the kernels are taken to be Gaussian, which deviates from the compact support assumption; this deviation does not pose any practical issues. The asymptotic properties of the bandwidth selection through such leave-one-out procedure for the empirical conditional survival function have, to our knowledge, not been examined in the case of mixing data and are thus material for future work. Presently, we simply mention that this approach provides a global bandwidth independent of both $k_n$ and focuses not only on $x$, but on all covariate values at the same time. 
    \item Finally consider the data-independent case of picking $h_{n} = \sqrt{\log(k_{n})/n}$, which falls into the framework of conditions~\ref{Assu:h_n_log_u_{n}^{x}},~\ref{Assu:n_h_F_limit},~\ref{rate_condition}, and equations~\eqref{eq:second_order_derived_sequences},~\eqref{Condition:sqrt_n_h3_F} for certain $\{k_n\}$ and appropriately chosen $\{r_n\}$.
\end{enumerate}

\subsection{Simulation study}
In this section we examine the finite-sample behaviour of the conditional Hill Estimator $\widehat{\gamma}_{n}\p{x}$ through a simulation study. To this end, we compare the bias and the mean squared error in four settings, where in each setting we have two degrees of dependence in the data (low and high intertemporal dependence) and for a fixed $x$, arising from the true tail index function $\gamma (x) = 3x(x-1) + 1$. More concretely, the settings considered are as follows.
\begin{enumerate}[label=\alph*)]
    \item 
    Conditional Fr\'echet distribution: The covariate process $\curly{X_j}$ is simulated from an auto-regressive (AR) process of order $1$ (with coefficient $0.1$ and $0.9$ in the low and high dependence cases respectively) transformed to have uniform marginals, similarly for the uniform process $\curly{U_j}$ that is then transformed into $\curly{Y_j}$ via the inverse Fr\'echet distribution function using the chosen $\gamma (X_n)$ as tail index. Here we use the Sheather--Jones optimal bandwidth, fitted for each $k_n$ on the $k_n$ concomitant values of $\curly{X_j}$, i.e. selected to correspond to the top $k_n$ order statistics of $\curly{Y_j}$. To evaluate the sample size, we consider $n = 10^3, 10^{4}$.
    
    \item 
    Conditional Pareto distribution: Similar construction as the conditional Fr\'echet, using instead the Pareto distribution function. For a fixed sample size of $n=500$, we compare the (global) optimal Sheather-Jones bandwidth fitted on the entire covariate process, fitted on the top $k_n$ values of $\curly{X_j}$, sorted according to the order statistics of $\curly{Y_j}$, the method of equation~\eqref{eq_h_hat} using equation~\eqref{eq:F_bar_double_smoothed} as conditional tail estimate and the choice $h_{n} = \sqrt{\log(k_{n})/n}$.
    \item 
    Consider a centred Gaussian process $\textbf{W}$ with $\text{var}(W_{j}) = \sigma_{j}^{2}$ and Gaussian covariance  kernel, and a covariate process $\curly{X_j}$. Let $\left\{P^{z}_{i}\right\}$ be the points of a Poisson point process on $(0,\infty)$ having mean measure $\nu_{1/\gamma(z)}$, where $\R \ni z \mapsto 1/\gamma(z) \in (0, \infty)$ is Lipschitz, and let $\textbf{W}^{(i)}, i \geq 1$ be i.i.d. copies of $\textbf{W}$. 
    Then the process $(Y_{j})$ defined by
    \begin{align} \label{eq:MS_cond_construction}
        Y_{j} = \bigvee_{i=1}^{\infty} P^{X_{j}}_{i}e^{W_{j}^{(i)} -(1/\gamma(X_{j})) \sigma_{j}^{2}/2},
    \end{align}
    is marginally regularly varying with tail index $1/\gamma(x)$ on $(X_j = x)$ and its construction is shown in Figure~\ref{fig:MS_diagram}. We refer to this construction as a Conditional Subordinated Gaussian Max-Stable (CSGMS) process. In the simulation study we let $\curly{X_j}$ be i.i.d. standard uniform.
    To evaluate the sample size, we consider $n = 10^3, 10^{4}$, and length-scale $l = 0.5, 2$ for all $j$, for the low and high dependence cases respectively. 
    
    \item 
    Conditional Fr\'echet with long memory: Similar construction as the conditional Fr\'echet, where the auto-regressive processes are replaced by auto-regressive fractional integrated moving averages (ARFIMA) with AR and moving average (MA) orders of 1, with AR coefficient $0.5$, MA coefficient $0.2$, i.e.
    \begin{align*}
        \left( 1-\frac{1}{2}B \right)(1-B)^{d}X_{t} = \left( 1+\frac{1}{5}B \right)\varepsilon_{t},
    \end{align*}
    where $B$ is the backshift operator, and $\curly{\varepsilon_t}$ are iid standard Gaussian noise terms. The fractional coefficient $d$ is chosen, such that the decay of $\beta$-mixing coefficient is of polynomial order as opposed to exponential. To evaluate the sample size, we consider $n = 10^3, 10^{4}$.
\end{enumerate}
The conditional Fréchet and Pareto are known to satisfy the second-order condition of equation~\eqref{cond_second_order}. 
By the remarks on page 224 in~\cite{Mikosch_Wintenberger}, it holds that $V^{(i)}_{j} := e^{W_{j}^{(i)} -(1/\gamma(X_{j})) \sigma_{j}^{2}/2}$ in equation~\eqref{eq:MS_cond_construction} can be represented by $(f_{j}(U_{i}))_{j=1}^{\infty}$ for suitable non-negative measurable functions $f_j$ on $(0,1)$, and $\curly{U_i}$ are i.i.d. uniform variables on $(0,1)$. By eq. (5.5.7) therein, the CSGMS process above has conditional Fréchet marginals. In particular it also satisfies the condition of equation~\eqref{cond_second_order} and we note that all the three constructions satisfy equation~\eqref{def:condRV}. The mapping $z \mapsto 1/\gamma(z) = 1/(3z(z-1) +1)$ is clearly Lipschitz continuous and hence condition~\ref{Assu:gamma_lip} holds. Condition~\ref{Assu:Expected_cross_indicator_limit} and equation~\eqref{eq::Expected_cross_log_limit} are verified for the conditional Fréchet and Pareto in Lemma~\ref{Lemma:cond_par_fre_AC_requirements}, Appendix~\ref{app_AC}, and anti-clustering conditions for the CSGMS hold due to Lemma~\ref{lemma:cond_ind_impl_anti_clustering}, Appendix~\ref{app_AC}. In the first two cases, the underlying processes are simulated from $\text{AR}(1)$-processes with auto-regressive coefficients absolutely bounded by $1$, which implies $\beta$-mixing with mixing coefficients of exponential decay (see~\cite{Ibragimov1962}). When considering the unconditional version of the CSGMS-process, it is possible to show $\beta$-mixing decay of exponential order (see~\cite{Kulik2020}, Example 13.5.4). In equation~\eqref{eq:MS_cond_construction} we extend the construction by bundling with the covariate sequence, which itself is i.i.d and in particular of exponential $\beta$-mixing decay. Since $z \mapsto 1/\gamma(z)$ is Lipschitz, then also the bivariate sequence $\curly{X_j, Y_j}$ is of exponential $\beta$-mixing decay. Hence, since the  conditional slowly varying functions in equation~\eqref{def:condRV} are seen to satisfy condition~\ref{Assu:L_lip} for both Pareto and Fréchet marginals, we may invoke Theorem~\ref{thm:existence_of_sequences} to conclude that there exists sequences $\left\{(k_{n}, h_{n}, r_{n}, \ell_{n})\right\}$ such that all conditions required in Theorem~\ref{Hill_normality_thm} are met in the above simulation studies.

Note that in the case when the conditional Fr\'echet process is constructed from an $\text{ARFIMA}(1,d,1)$-process with $d = 0.1, 0.45$ in the low and high dependence respectively, both cases imply that the $\beta$-mixing coefficients decay polynomially, still falling in the framework of Theorem~\ref{thm:existence_of_sequences} and~\ref{thm:existence_of_sequences_log}. 
The results of the simulation study for $k_{n} \in \left\{n/100, \ldots, n/2\right\}$ are depicted in Figures~\ref{fig:Frechet_Pareto} and~\ref{fig:MS_Arfima}, where we plot the bias and Mean Squared Error (MSE) of the estimator, evaluated in $x = 0.6$, as a function of $k_{n}/n$, averaged over 500 simulations. 

For the Fr\'echet setting a), we observe that increasing the dependence structure in the underlying data, results in an similar finite-sample bias and MSE. In addition the bias exhibits unstable behaviour when $k_{n}/n$ approaches zero, for the smaller sample size. This highlights the double sub-sampling used in conditional extreme estimators.\\
For the Pareto setting b), we focus on comparing the bandwidth selection techniques, as this distribution has no second-order component (no bias term). For the both dependence structure, the optimal bandwidth on the upper $k_{n}$ concomitants of $X$, yields slightly lower bias when $k_{n}/n$ is away from zero. For $k_{n}/n$ lower than approximately $0.1$, the data-independent bandwidth $h_n =\sqrt{\log(k_{n})/n}$ results in the lowest mean squared error. We observe that neither of these two techniques are superior to another, and hence use the technique of lowest computational complexity, which is $h_n =\sqrt{\log(k_{n})/n}$ for the remaining simulations and real-data analysis.

\begin{figure}[hbt!]
\centering
\includegraphics[width=0.39\textwidth,trim={0.6cm 0cm 0cm 0cm},clip]{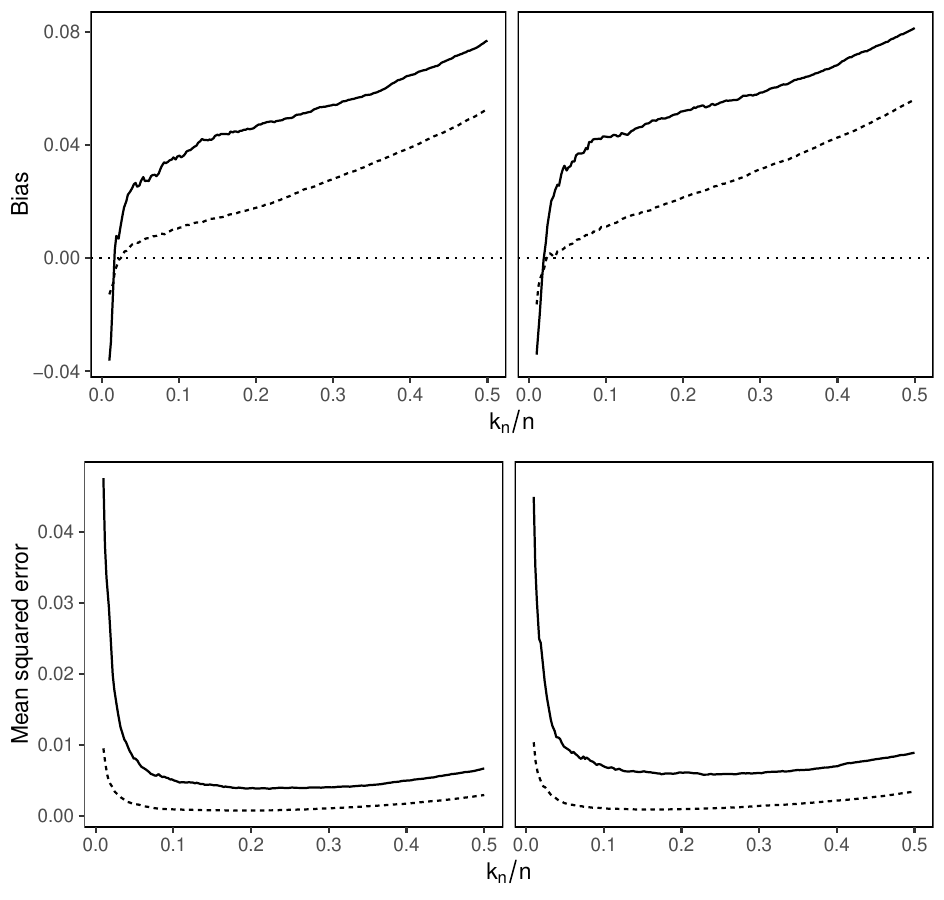}
\includegraphics[width=0.39\textwidth,trim={0.6cm 0cm 0cm 0cm},clip]{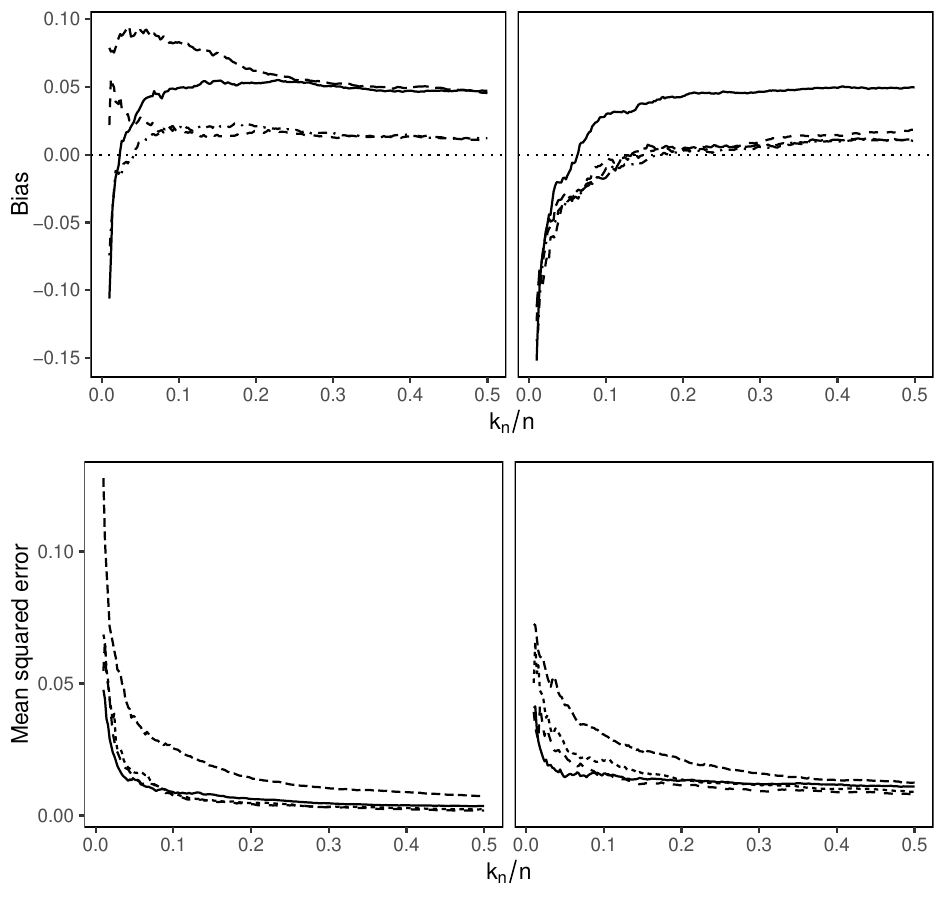}
\caption{
Top panels: Bias. Bottom panels: MSE.
We consider the conditional Fr\'echet model (low dependence (far left) and high dependence (center left)) $n = 10^3, 10^{4}$ in solid and dashed, respectively; number of simulations is $N = 500$.
We also consider the conditional Pareto model (low dependence (center right) and high dependence (far right)) with 
$n = 500, N = 200$, and optimal bandwidth of $\curly{X_j, Y_j}$ using least-squares cross validation (long dash), optimal bandwidth of $X$ (dot dash), optimal bandwidth of upper $k_n$ concomitants of $X$ (dash) and $h_n =\sqrt{\log(k_{n})/n}$ (solid).} 

\label{fig:Frechet_Pareto}
\end{figure}
For the CSGMS setting c), we observe a strong positive bias  for $k_n/n$ away from zero for high dependence and low sample size. Increasing the sample size dramatically reduces the bias and MSE for the high degree of dependence, to the point where the high dependence scenario catches up to the performance of the low dependence scenario.

Finally, for the long memory setting d), which also falls in the scope of our asymptotic results, the estimator performs rather well with a negative bias for high quantiles, and positive bias for lower quantiles. In both cases the MSE  stabilises as the sample size increases, and the negative bias at top quantiles diminishes. This suggests that the conditional Hill estimator can possibly be studied for time series with even stronger temporal dependence, albeit with possibly other techniques than the ones proposed in this paper.
\begin{figure}[hbt!]
\centering
\includegraphics[width=0.39\textwidth,trim={0.6cm 0cm 0cm 0cm},clip]{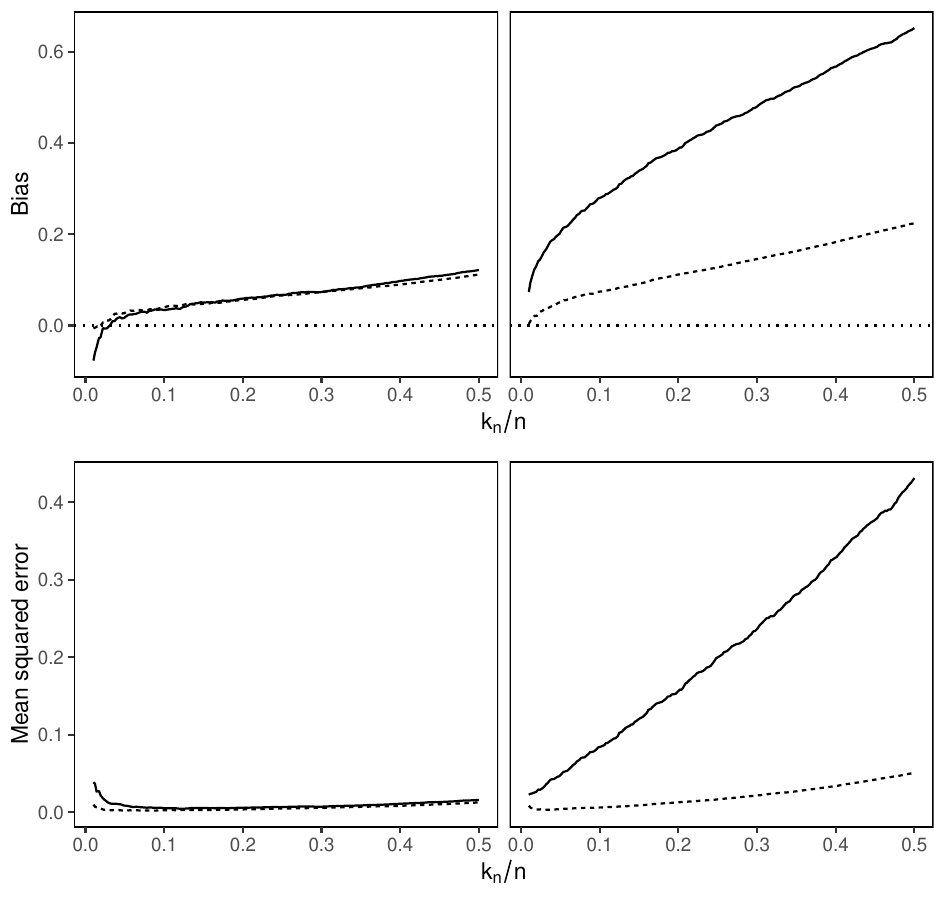}
\includegraphics[width=0.39\textwidth,trim={0.6cm 0cm 0cm 0cm},clip]{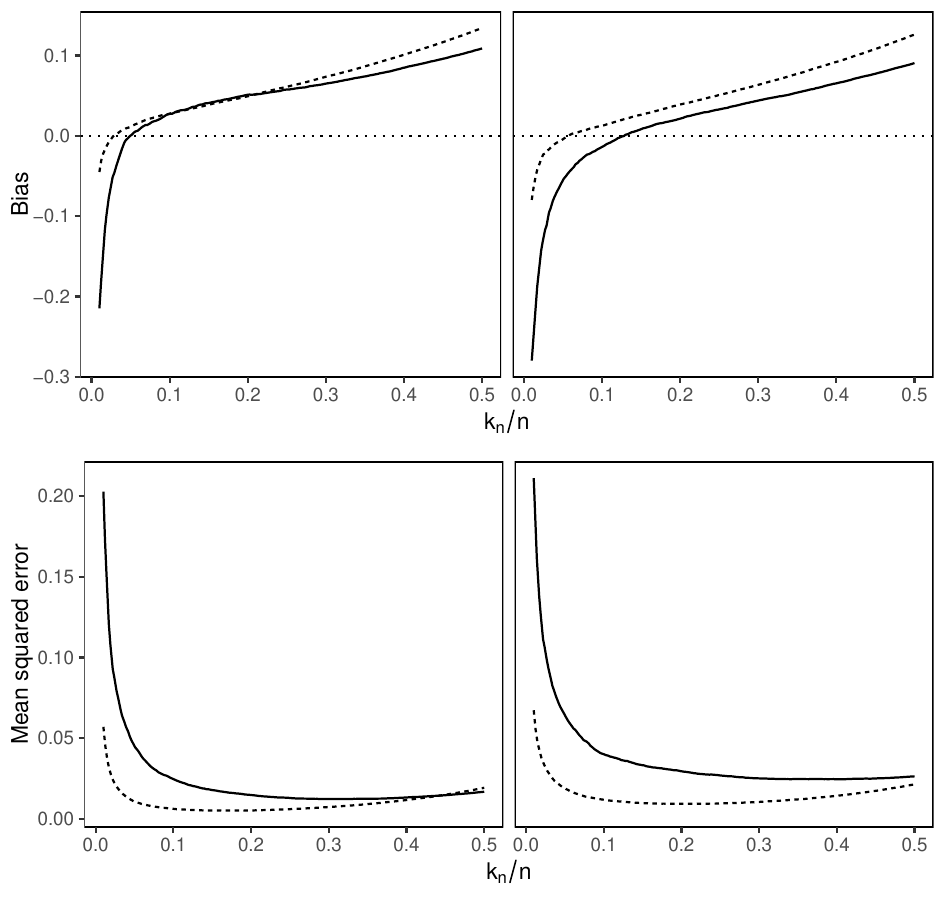}
\caption{
Top panels: Bias. Bottom panels: MSE. Sample sizes $n = 10^3, 10^{4}$ in solid and dashed, respectively; number of simulations is $N = 500$.
We consider the CSGMS model (low dependence (far left) and high dependence (center left)), and the conditional Fr\'echet model constructed from an underlying ARFIMA process (low dependence (center right) and high dependence (far right)).}
\label{fig:MS_Arfima}
\end{figure}

\section{Real data Analysis}\label{section:real}

\subsection{Description of dataset}
The West Texas Intermediate (WTI) is a grade of crude oil and is considered one of the main benchmarks in oil pricing. We are interested to examine whether the spot price of WTI is conditionally regularly varying given the Crude Oil Volatility Index (CBOE OVX) and, if this is the case, how the tail index is influenced by different levels of implied volatility.
The CBOE OVX is constructed by applying the VIX (Chicago Board Options Exchange's Volatility Index) methodology (see~\cite{whaley2000investor}) to options on the United States Oil Fund (USO), using a weighted average of mid-quote prices from calls and puts across a range of strike prices. This aggregated implied volatility is extrapolated to a 30‐day horizon to provide a forward-looking measure of the expectation of the market of price fluctuations and is expressed in percent.

Now, as USO is designed to track the WTI price using future contracts, the OVX provides a forward looking market expectation of the WTI price volatility. Thus, in the case of regularly varying marginal distributions of the WTI, we expect that the tail index varies as a function of OVX. While the OVX is mean-reverting and can be considered stationary (see~\cite{chen2018predictive}), oil prices are generally not stationary, and are consequently often transformed to log-returns (see~\citet{hamilton1983oil,sadorsky1999oil}) to achieve stationarity. We follow this approach. The corresponding time series are depicted in Figure~\ref{TS_WTI_CBOE}. Note how high activity periods such as the financial crisis (2008), oil price collapse due to US and OPEC oil war (2015), covid period (2020-2021), and the Russian and Ukrainian conflict (2022), are visually observable. Remark that the WTI price attained a negative value on the 20th of April 2020, however this observation has been removed in order to transform with logarithms.

\begin{figure}[hbt!] 
    \centering
    \includegraphics[width=0.75\textwidth]{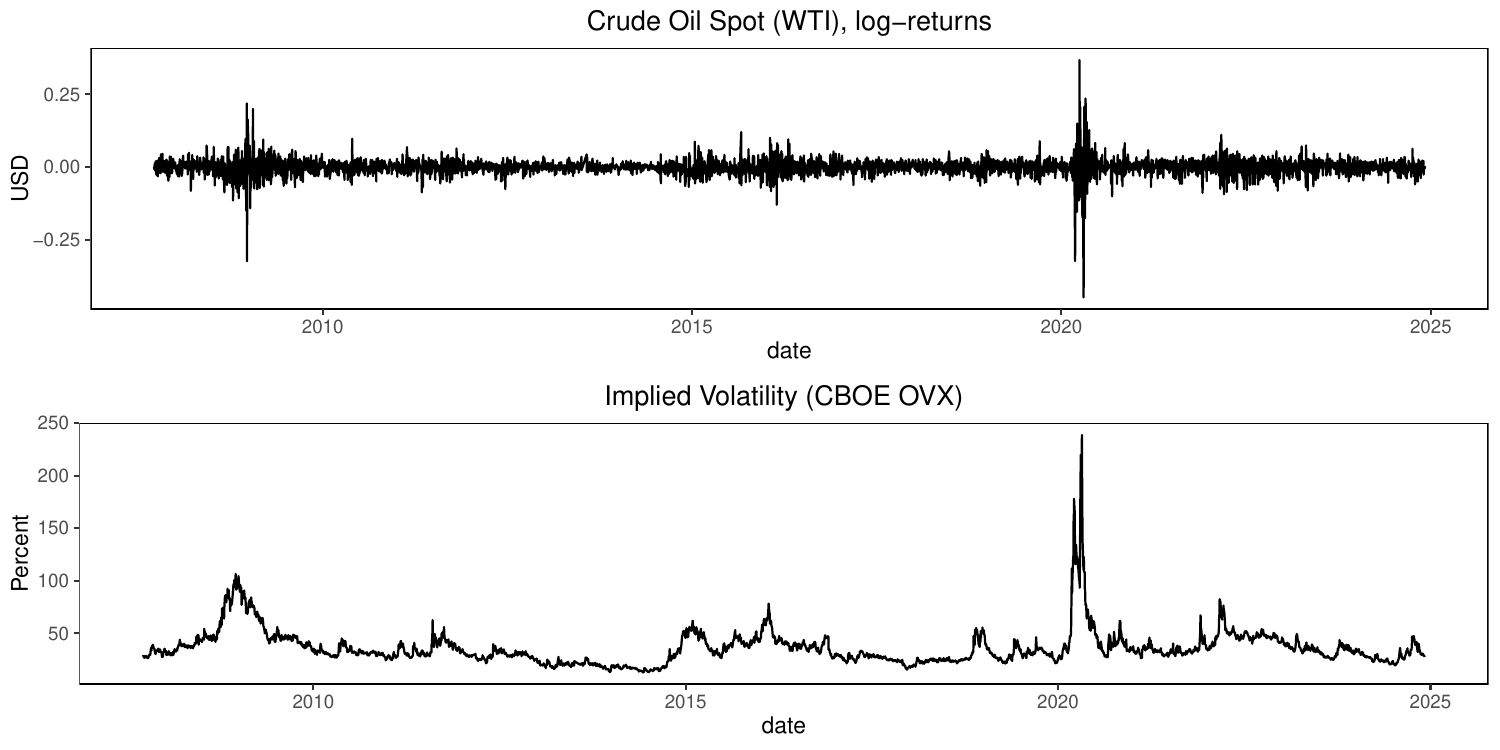} 
    \caption{Log-returns of the WTI spot price (top) and oil volatility index CBOE OVX (bottom) from 2007 to end of 2024.}
    \label{TS_WTI_CBOE}
\end{figure}

\subsection{Analysis of crude oil spot price}
As seen from Figure~\ref{TS_WTI_CBOE}, the WTI log-returns attain both positive and negative values. By construction, the conditional Hill estimator is applicable only for positive $\curly{Y_j}$. Splitting the bivariate time series $\curly{X_j, Y_j}$ according to where $\curly{Y_j}$ is positive and negative respectively, and switching the sign of the negative parts, results in two bivariate time series with positive $\curly{Y_j}$ components. The two time series are analysed separately and are seen to exhibit different tail behaviour, which allows for a more detailed analysis.

The covariate process OVX is, as in the simulation study, transformed to have uniform margins. Prior to employing our proposed conditional Hill estimator, we preliminarily check that the bivariate time series (positive or negative side) is in fact conditionally regularly varying. For this purpose, we subsample the WTI log-returns for different values of (the uniformly transformed) OVX.
This then results in the Pareto QQ-plots depicted in Figure~\ref{fig:ACF_plot}, where the largest $k_n=n/2$ observations are used. We clearly observe linearity in the QQ-plots, which suggests conditional regular variation. We can also already note that for both the positive and negative sides of the WTI log-returns, the slope is steepest for large implied volatility values.

\begin{figure}[hbt!] 
    \centering
    \includegraphics[width=.30\textwidth]{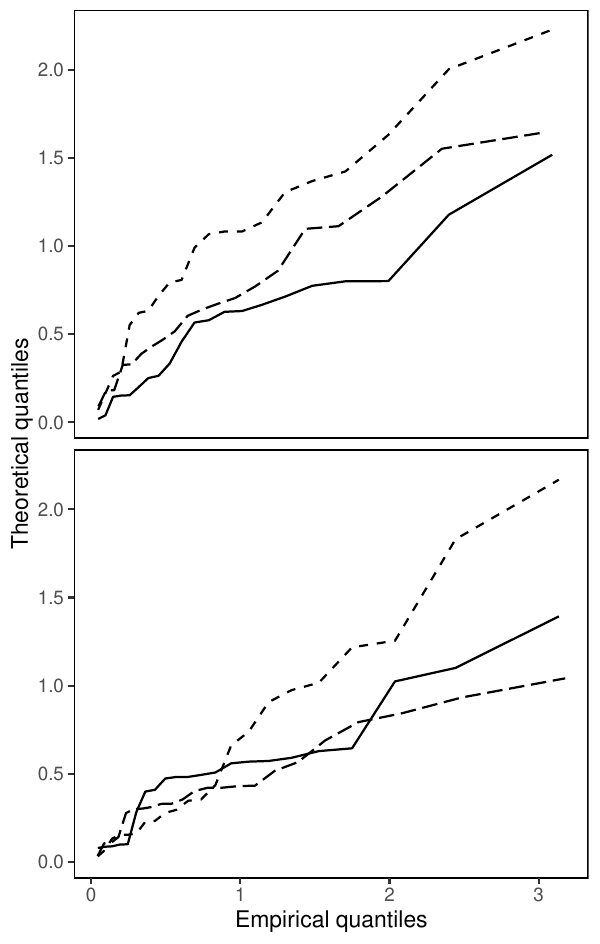} 
    \includegraphics[width=.5\textwidth]{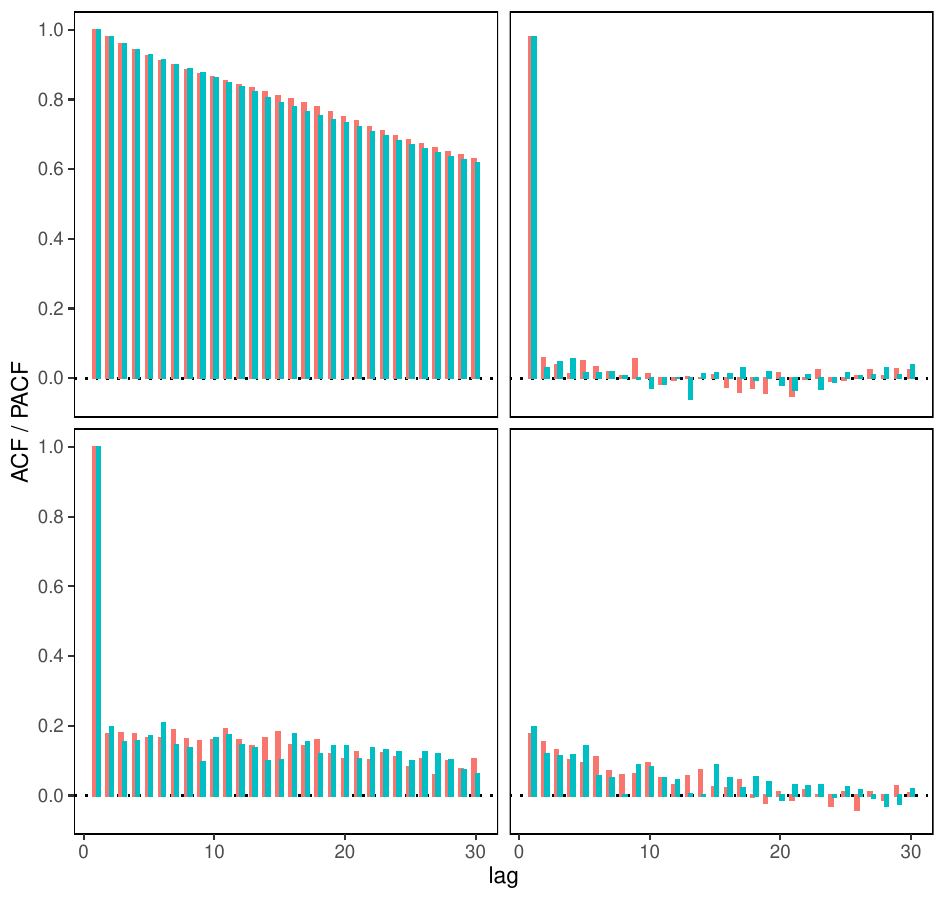} 
    \caption{Left panels: Pareto QQ-plots of the negative (top) and positive (bottom) WTI log-returns subsampled on uniformly transformed values of OVX in the intervals $\{[0.01, 0.03], [0.49, 0.51], [0.96, 0.98]\}$ in dashed, long-dashed and solid respectively.
    Four right panels: Auto-correlation plots (left) and partial auto-correlation plots (right) of the negative (red) and positive (blue) of CBOE OVX data (top) and WTI log-returns (bottom).}
    \label{fig:ACF_plot}
\end{figure}

\begin{figure}[hbt!]
    \centering
    \includegraphics[width=.65\textwidth]{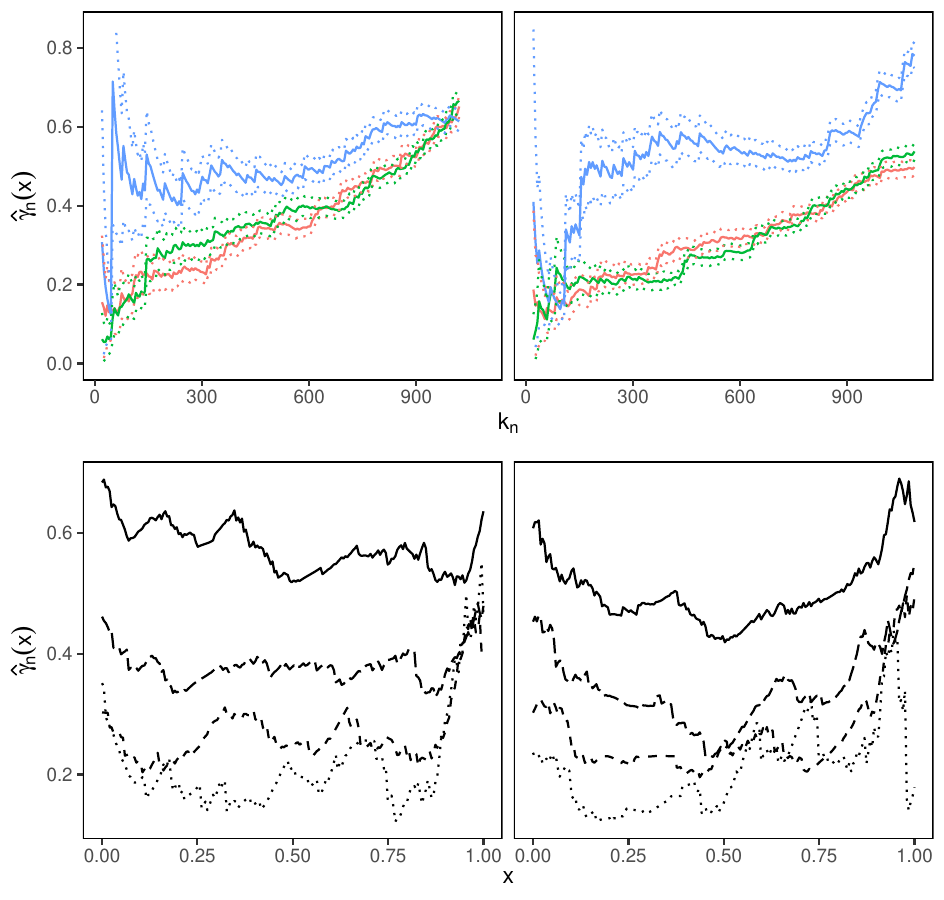} 
    \caption{Plots of $\widehat{\gamma}_{n}(x)$, using the negative (left) and positive (right) WTI log-returns as $(Y_{n})$ and the uniformly transformed OVX as $(X_{n})$ (covariate process). In the top panels we depict the conditional Hill estimator, as of function of $k_{n}$, i.e. the number of top order statistics included, for four chosen covariate values $x \in \{0.231, 0.402, 0.995\}$ (red, green, and blue, respectively). In the bottom panels $\widehat{\gamma}_{n}(x)$ is plotted as a function of $x$ for four fixed number of order statistics, $k_{n} \in \{80, 241, 552, 959\}$ (dotted, dashed, long-dashed and solid, respectively).}
    \label{fig:WTI_CBOE_Hill}
\end{figure}
In Figure~\ref{fig:WTI_CBOE_Hill} we consider the conditional Hill estimator of the WTI log-returns given the uniformly transformed OVX. We observe that the largest $\widehat{\gamma}_{n}(x)$ value is attained for both negative and positive sides of the WTI log-returns, when the implied volatility is large. This is in accordance with the common financial intuition that large implied volatility implies heavier tails (and thus fewer moments) for log-returns of markets on which it is based. Furthermore, when the implied volatility is large, the negative side of the WTI log-returns exhibits slightly heavier tails than the positive side. This has been documented for many assets (see~\citet{Christie_82}), where shock effects can create large upswings but even greater downswings. 
In addition, the confidence intervals that arise from Theorem~\ref{Hill_normality_thm} are depicted. For the both sides, the lower bound of the estimator corresponding to high volatility is separated from the upper bound of the of the estimators, corresponding to low and mid volatility, within the stable regions. This supports the claim the the the WTI log-returns change tail index according the implied volatility.
Since the conditional Hill estimator can be viewed as a function of the covariate value $x$ for a fixed sample fraction $k_{n}$ bivariately, we also plot these ``risk profiles", where slightly noisy right-skewed smiles can be seen, which indicates that the tail index is large when volatility is low and even larger when high.

\section{Conclusion and outlook}\label{section:conclusion}

We have shown that the conditional Hill estimator can be applied for a variety of time series models, in particular only $\alpha$-mixing being required to establish consistency. Numerical studies suggest good performance for both exponentially and polynomially decaying $\beta$-mixing coefficients, which is also supported by the theory. The estimator might also be asymptotically convergent for stronger dependence structures, which is subject of further investigation. 

Apart from considering stronger intertemporal dependence structures, our work also allows us to consider extensions in several directions. First, having an estimate of the entire tail at our disposal our results can be used to construct and statistically analyse conditional risk measures such as conditional expected shortfall or conditional tail moments. Secondly, a promising direction is to consider specific models where the same time series can be used as the covariate and response, but at different lags; Markov processes are particularly attractive in this connection. Finally,
removing top observations can sometimes lead to the construction of stable estimators that allow transparent threshold selection as in~\cite{bladt2020threshold}, which is yet to be explored both for univariate and conditional cases.

\newpage

%%%%%%%%%%%%%%%%%%%%%%%%%%%%%%%%%%%%%%%%%%%%%%
%% Support information, if any,             %%
%% should be provided in the                %%
%% Acknowledgements section.                %%
%%%%%%%%%%%%%%%%%%%%%%%%%%%%%%%%%%%%%%%%%%%%%%
%\begin{acks}[Acknowledgments]
%\end{acks}
%%%%%%%%%%%%%%%%%%%%%%%%%%%%%%%%%%%%%%%%%%%%%%
%% Funding information, if any,             %%
%% should be provided in the                %%
%% funding section.                         %%
%%%%%%%%%%%%%%%%%%%%%%%%%%%%%%%%%%%%%%%%%%%%%%
%\begin{funding}
\subsection*{Funding}
The authors were supported by the Carlsberg Foundation, grant CF23-1096.
%\end{funding}

\subsection*{Data availability statement}

The data were accessed through a subscription-based platform and are not publicly available.
\newpage
%%%%%%%%%%%%%%%%%%%%%%%%%%%%%%%%%%%%%%%%%%%%%%
%% Supplementary Material, including data   %%
%% sets and code, should be provided in     %%
%% {supplement} environment with title      %%
%% and short description. It cannot be      %%
%% available exclusively as external link.  %%
%% All Supplementary Material must be       %%
%% available to the reader on Project       %%
%% Euclid with the published article.       %%
%%%%%%%%%%%%%%%%%%%%%%%%%%%%%%%%%%%%%%%%%%%%%%
%\begin{supplement}
%\stitle{???}
%\sdescription{???.}
%\end{supplement}

%%  MR numbers will be added by VTeX.                      %%
%%                                                         %%
%%  Use~\citet{...} to cite references in text.             %%
%%                                                         %%
%%%%%%%%%%%%%%%%%%%%%%%%%%%%%%%%%%%%%%%%%%%%%%%%%%%%%%%%%%%%%

%% if your bibliography is in bibtex format, uncomment commands:

\bibliographystyle{imsart-nameyear} % Style BST file (imsart-number.bst or imsart-nameyear.bst)
\bibliography{bibliography}       % Bibliography file (usually '*.bib')

%% or include bibliography directly:
% \begin{thebibliography}{}
% \bibitem{b1}
% \end{thebibliography}

%%%%%%%%%%%%%%%%%%%%%%%%%%%%%%%%%%%%%%%%%%%%%%
%% Single Appendix:                         %%
%%%%%%%%%%%%%%%%%%%%%%%%%%%%%%%%%%%%%%%%%%%%%%
%\begin{appendix}
%\section*{???}%% if no title is needed, leave empty \section*{}.
%\end{appendix}
%%%%%%%%%%%%%%%%%%%%%%%%%%%%%%%%%%%%%%%%%%%%%%
%% Multiple Appendixes:                     %%
%%%%%%%%%%%%%%%%%%%%%%%%%%%%%%%%%%%%%%%%%%%%%%
\newpage
\begin{appendix}
\section{Proofs of main results from Section~\ref{section:main}}\label{appA}
\subsection{Proof of Theorem~\ref{Empirical_tail_dist_consistency_theorem_m_depence} (Tail consistency, \texorpdfstring{$m$-dependence}{m-dependence})}

We begin by proving some preliminary results related to the process $\widetilde{V}_{n}^{x}$.
\begin{lemma} \label{Limit_of_tail_process}
    Assume conditions~\ref{Assu:gamma_lip}-\ref{Assu:L_lip}. Then 
    \begin{align*}
        \E{\widetilde{V}^{x}_{n}(s)}  \to s^{-1/\gamma(x)}g(x), \ s > 0.
    \end{align*}
\end{lemma}
\begin{proof}
    If $g(x) > 0$ it holds that
    \begin{align*}
\E{\widetilde{V}^{x}_{n}(s)} =
         g(x)\frac{\overline{F}^{x}(su_{n}^{x})}
         {\overline{F}^{x}(u_{n}^{x})}
         \int 
          (su_{n}^{x})^{-(1/\gamma(x-uh_{n}) -1/\gamma(x))}
         \frac{L^{x-uh_{n}}(s u_{n}^{x})}
         {L^{x}(su_{n}^{x})} 
        \frac{g(x - h_{n}u)}{g(x)}   K\p{u}
         \mathrm{d}u. 
    \end{align*}
    The expression before the integral converges to the desired limit, and hence it remains to show that the integral converges to $1$. We consider each of the three $n$-dependent factors in the integral
    \begin{align*}
        &(su_{n}^{x})^{-(1/\gamma(x-uh_{n}) -1/\gamma(x))} =
        \exp\left(-(1/\gamma(x-uh_{n}) -1/\gamma(x))\log(su_{n}^{x})\right)\\
        &\leq 
        \exp\left(c_{\gamma}h_n\lvert u\rvert \log(su_{n}^{x})\right) 
        =
         \exp\left(c_{\gamma}\lvert u\rvert h_n\log(u_{n}^{x})\right) 
          \exp\left(c_{\gamma}h_n\lvert u\rvert \log(s)\right) \to 1,
    \end{align*}
    and similarly 
    \begin{align*}
        (su_{n}^{x})^{-(1/\gamma(x-uh_{n}) -1/\gamma(x))} \geq
         \exp\left(-c_{\gamma}\lvert u\rvert h_n\log(u_{n}^{x})\right) 
          \exp\left(-c_{\gamma}h_n\lvert u\rvert \log(s)\right) \to 1,
    \end{align*}
    uniformly in $u \in [-1,1]$, using condition~\ref{Assu:gamma_lip} and $h_n\log(u_{n}^{x}) \to 0$.
    Now, condition~\ref{Assu:L_lip} implies
    \begin{align*}
        \exp\big(-\log(su_{n}^{x}) h_n c_{L}\left\vert u \right\vert\big)\le \frac{L^{x-uh_{n}}(s u_{n}^{x})}
         {L^{x}(su_{n}^{x})} 
         \leq \exp\big(\log(su_{n}^{x}) h_n c_{L}\left\vert u \right\vert\big)
    \end{align*}
    and both sides converge to $1$, uniformly in $u \in [-1,1]$. 
    The density $g$ is assumed twice continuously differentiable which implies Lipschitz continuity, so that 
    \begin{align*}
        1 - \frac{ c_{g}\lvert u \rvert h_n}{g(x)}\le \frac{g(x - h_{n}u)}{g(x)}
        \leq 
        1 + \frac{ c_{g}\lvert u \rvert h_n}{g(x)}
    \end{align*}
      and both sides converge to $1$, uniformly in $u \in [-1,1]$.\\
    By the dominated convergence theorem get $\mathbb{E}[\widetilde{V}^{x}_{n}(s)] \to g(x)s^{-1/\gamma(x)}$. If $g(x) = 0$, similar arguments, using Lipschitz continuity of $g$, yield the result.
\end{proof}

\begin{lemma} \label{Expected_indicator_kernel_lemma}
    Assume conditions~\ref{Assu:gamma_lip}-\ref{Assu:L_lip}. For any $t,q > 0$ it holds that
    \begin{align*}
        \frac{\E{1_{\curly{Y > u_{n}^{x}t}}K^{q}\p{\frac{x-X}{h_n}}}}{h_{n}\overline{F}^{x}(u_{n}^{x})}
        \to t^{-1/\gamma(x)}g(x)\int K^{q}\p{u}\mathrm{d}u.
    \end{align*}
\end{lemma}
\begin{proof}
    Calculations similar to those of Lemma~\ref{Limit_of_tail_process} yield the result.
\end{proof}

\begin{proof}[Proof of Theorem~\ref{Empirical_tail_dist_consistency_theorem_m_depence}]
    By assumption $g_{n}(x) \overset{\P}{\to} g(x)$, and hence, using the relation \newline $\widetilde{T}^{x}_{n}(s) = \widetilde{V}^{x}_{n}(s)/g_{n}(x)$, it remains to show that $\widetilde{V}^{x}_{n}(s)  \overset{\P}{\to} g(x)s^{-1/\gamma(x)}$.\\
    Let $m \in \N$ be given
    and split  $\widetilde{V}^{x}_{n}(\cdot)$ into a main component, which is a sum of independent terms, and a remainder term
    \begin{align*}
        \widetilde{V}^{x}_{n}(s) = 
        \frac{1}{m}\sum_{i = 1}^{m}\widetilde{V}^{x, i}_{n}(s) + R_{n}(s), 
    \end{align*}
    where 
    \begin{align*}
        \widetilde{V}^{x, i}_{n}(s) &= 
        \frac{m}{nh_{n} \overline{F}^{x}(u_{n}^{x})}
        \sum_{j=1}^{\floor{\frac{n}{m}}}K\left(\frac{x-X_{(j-1)m + i}}{h_n}\right)1_{\left\{Y_{(j-1)m + i} > su_{n}^{x}\right\}}, \\
        R_{n}(s) &= 
        \frac{1}{nh_{n} \overline{F}^{x}(u_{n}^{x})}
    \sum_{j=\floor{\frac{n}{m}}m}^{n}K\left(\frac{x-X_j}{h_n}\right)1_{\left\{Y_j > su_{n}^{x}\right\}}.
    \end{align*}
    Stationarity and Lemma~\ref{Expected_indicator_kernel_lemma} yields that
    \begin{align*}
        \E{\widetilde{V}^{x, i}_{n}(s)} = 
        \frac{m}{n}\floor{\frac{n}{m}}
        \frac{\E{1_{\curly{Y > u_{n}^{x}s}}K\p{\frac{x-X}{h_n}}}}{h_{n} \overline{F}^{x}(u_{n}^{x})} \to 
        s^{-1/\gamma(x)}g(x),
    \end{align*}
    while $m$-dependence implies
    \begin{align*}
         &\Var \p{\widetilde{V}^{x, i}_{n}(s)} =
          \frac{\frac{m}{n}\floor{\frac{n}{m}} m}{nh_{n} \overline{F}^{x}(u_{n}^{x})}
          \Bigg[\frac{\E{1_{\curly{Y > u_{n}^{x}s}}K^{2}\p{\frac{x-X}{h_n}}}}{h_{n}\overline{F}^{x}(u_{n}^{x})} -
           \frac{\E{1_{\curly{Y > u_{n}^{x}s}}K\p{\frac{x-X}{h_n}}}^2}{h_{n}\overline{F}^{x}(u_{n}^{x})}\Bigg],
    \end{align*}
    which vanishes by Lemma~\ref{Expected_indicator_kernel_lemma}. Negligibility of the remainder term follows from
    \begin{align*}
    &\E{R_{n}(s)} = \left(n- \left\lfloor\frac{n}{m}\right\rfloor m\right)\frac{1}{n} \frac{\E{K\left(\frac{x-X}{h_n}\right)1_{\left\{Y > su_{n}^{x}\right\}}}}{h_{n} \overline{F}^{x}(u_{n}^{x})} \to 0, 
    \\
    &\E{R_{n}(s)^2} 
    = \frac{1}{\left(nh_{n} \overline{F}^{x}(u_{n}^{x})\right)^2} 
    \biggl( \sum_{j=\left\lfloor\frac{n}{m}\right\rfloor m}^{n} \E{K^2\left(\frac{x-X_j}{h_n}\right)1_{\left\{Y_j > su_{n}^{x}\right\}}} 
    \\
    &\qquad \qquad\quad \ \ + 2\sum_{j=\left\lfloor\frac{n}{m}\right\rfloor m}^{n}\sum_{i=j}^{n} \E{1_{\left\{Y_i > u_{n}^{x}s\right\}}K\left(\frac{x-X_i}{h_n}\right)1_{\left\{Y_j > u_{n}^{x}s\right\}}K\left(\frac{x-X_j}{h_n}\right)} \biggr) 
    \\
    &\leq \frac{\left(n-\left\lfloor\frac{n}{m}\right\rfloor m\right)}{\left(nh_{n} \overline{F}^{x}(u_{n}^{x})\right)^2} \biggl(  \E{K^2\left(\frac{x-X}{h_n}\right)1_{\left\{Y > su_{n}^{x}\right\}}}  
    + 2\left(n-\left\lfloor\frac{n}{m}\right\rfloor m\right) \sqrt{\E{1_{\left\{Y > u_{n}^{x}s\right\}}K^2\left(\frac{x-X}{h_n}\right)}^2} \ \biggr) 
    \\
    &= \frac{\left(n-\left\lfloor\frac{n}{m}\right\rfloor m\right)}{n^2 h_{n} \overline{F}^{x}(u_{n}^{x})} \frac{\E{1_{\left\{Y > u_{n}^{x}t\right\}}K^{2}\left(\frac{x-X}{h_n}\right)}}{h_{n}\overline{F}^{x}(u_{n}^{x})} \left(1 + 2\left(n-\left\lfloor\frac{n}{m}\right\rfloor m\right)\right) \to 0,
\end{align*}
    where we used the Cauchy--Schwarz inequality. Then Dini's theorem extends the convergence uniformly on compacts bounded away from zero. 
    This finishes the proof.
\end{proof}

\subsection{Proof of Theorem~\ref{Empirical_tail_dist_consistency_theorem} (Tail consistency, \texorpdfstring{$\alpha$-mixing}{alpha-mixing})}
\begin{proof}[Proof of Theorem~\ref{Empirical_tail_dist_consistency_theorem}]
By assumption $g_{n}(x) \overset{\P}{\to} g(x)$ and Lemma~\ref{Limit_of_tail_process} yields asymptotic unbiasedness of $\widetilde{V}_{n}^{x}$, and hence consistency of $\widetilde{V}_{n}^{x}$ follows if the variance is negligible as $n\to \infty$. We have
\begin{align*}
    \text{var}(\widetilde{V}_{n}^{x}(s))
    &=
    \frac{1}{(n h_{n} \overline{F}^{x}(u_{n}^{x}))^{2}}
    \Bigg(
    \sum_{j=1}^{n}  \E{K^{2}\left(\frac{x-X_{j}}{h_n}\right)1_{\left\{Y_{j} > su_{n}^{x}\right\}}}  - n^{2}\E{K\left(\frac{x-X_{0}}{h_n}\right)1_{\left\{Y_{0} > su_{n}^{x}\right\}}}^{2}  
    \\ &\qquad\qquad\qquad 
    + 2n  \sum_{j=1}^{n} \left(1-\frac{j}{n}\right) \E{K\left(\frac{x-X_{0}}{h_n}\right)1_{\left\{Y_{0} > su_{n}^{x}\right\}}K\left(\frac{x-X_{j}}{h_n}\right)1_{\left\{Y_{j} > su_{n}^{x}\right\}}}
    \Bigg) \\
    &\leq
    \frac{1}{(n h_{n} \overline{F}^{x}(u_{n}^{x}))^{2}}
    \Bigg(
    \sum_{j=1}^{n}  \E{K^{2}\left(\frac{x-X_{j}}{h_n}\right)1_{\left\{Y_{j} > su_{n}^{x}\right\}}}\\
    &\quad 
    + 2n  \sum_{j=1}^{n} \text{cov}\left(K\left(\frac{x-X_{0}}{h_n}\right)1_{\left\{Y_{0} > su_{n}^{x}\right\}}, K\left(\frac{x-X_{j}}{h_n}\right)1_{\left\{Y_{j} > su_{n}^{x}\right\}} \right)
    \Bigg),
\end{align*}
where the first term vanishes  due to Lemma~\ref{Expected_indicator_kernel_lemma} and condition~\ref{Assu:n_h_F_limit}.
The leading term of the covariance behaves as
\begin{align*}
        &\frac{1}{h_{n} \overline{F}^{x}(u_{n}^{x})}\E{K\left(\frac{x-X_{0}}{h_n}\right)1_{\left\{Y_{0} > su_{n}^{x}\right\}}K\left(\frac{x-X_{j}}{h_n}\right)1_{\left\{Y_{j} > su_{n}^{x}\right\}}} \\
        &= h_{n}
        \int_{[-1,1]^2}  \biggl\{
         \frac{\CE{1_{\left\{Y_0 > su_{n}^{x}\right\}}
       1_{\left\{Y_j > su_{n}^{x}\right\}}}{X_0 = x + w_{1}h_n, X_j = x + w_{2}h_n}}
         {\overline{F}^{x}(u_{n}^{x})} \\
         & \qquad \qquad \qquad \qquad \qquad \qquad 
         \times g(x + w_{1}h_{n},x + w_{2}h_{n})K\left(w_{1}\right)K\left(w_{2}\right) \biggr\}
         \mathrm{d}(w_1, w_2) = \mathcal{O}(h_n),
\end{align*}
where the integral is bounded,  for $n$ large enough, due to condition~\ref{Assu:Expected_cross_indicator_limit} and smoothness of $g_{X_0, X_j}$. Now, Billingsley's inequality with the $\alpha$-mixing assumption yields that
\begin{align*}
    &\left|\text{cov}\left(K\left(\frac{x-X_{0}}{h_n}\right)1_{\left\{Y_{0} > su_{n}^{x}\right\}}, K\left(\frac{x-X_{j}}{h_n}\right)1_{\left\{Y_{j} > su_{n}^{x}\right\}} \right)\right| \\
    & \leq 4\alpha(j) 
    \norm{K\left(\frac{x-X_{0}}{h_n}\right)1_{\left\{Y_{0} > su_{n}^{x}\right\}}}_{\infty} 
    \norm{ K\left(\frac{x-X_{j}}{h_n}\right)1_{\left\{Y_{j} > su_{n}^{x}\right\}}}_{\infty} 
    =
    \mathcal{O}(\alpha(j)),
\end{align*}
using that the kernel is bounded.
Let $d_{n} = \lceil h_{n}^{-c_{x}} \rceil$ for some $c_{x} > 0$. Combing the above considerations yields that
\begin{align*}
    &\frac{2n}{(n h_{n} \overline{F}^{x}(u_{n}^{x}))^{2}}
    \left|\sum_{j=1}^{n} \text{cov}\left(K\left(\frac{x-X_{0}}{h_n}\right)1_{\left\{Y_{0} > su_{n}^{x}\right\}}, K\left(\frac{x-X_{j}}{h_n}\right)1_{\left\{Y_{j} > su_{n}^{x}\right\}} \right)\right| \\
    & \leq
    \frac{\text{cst}}{n (h_{n} \overline{F}^{x}(u_{n}^{x}))^{2}}
    \left(
    \sum_{j=1}^{d_n -1} h_{n}^{2} \overline{F}^{x}(u_{n}^{x}) + \sum_{j = d_n}^{\infty} j^{-\eta}
    \right) \leq
    \frac{\text{cst}}{n (h_{n} \overline{F}^{x}(u_{n}^{x}))^{2}}
    \left(
    d_n h_{n}^{2} \overline{F}^{x}(u_{n}^{x}) + d_{n}^{-\eta +1}
    \right)\\
    &=
    \text{cst}\left(\frac{d_{n}}{n\overline{F}^{x}(u_{n}^{x})} + \frac{d_{n}^{-\eta + 1}}{n (h_{n} \overline{F}^{x}(u_{n}^{x}))^{2}} \right)\sim
    \text{cst}\left(\frac{1}{n h_{n}^{c_{x}} \overline{F}^{x}(u_{n}^{x})} + \frac{h_{n}^{c_{x}(\eta - 1) -2}}{n (\overline{F}^{x}(u_{n}^{x}))^{2}} \right) \to 0.
\end{align*}
Once again, Dini's theorem extends the convergence uniformly on compacts bounded away from zero.
\end{proof}

\subsection{Proof of Theorem~\ref{Hill_Consistency_thm} (Hill consistency)}

\begin{proof}[Proof of Proposition~\ref{Empirical_quantile_consistency_prop}]
    Let $\varepsilon > 0$ be given and assume that $q_{n,k_{n}}^{x}/u_{n}^{x} > 1+\varepsilon$ for all $n \in \N$. Since $\overline{F}^{x}_{n}$ is decreasing we have 
    \begin{align*}
            \widetilde{T}^{x}_{n}\p{1+\varepsilon} = 
             \frac{ \overline{F}^{x}_{n}(\p{1+\varepsilon}u_{n}^{x})}
    { \overline{F}^{x}(u_{n}^{x})} \geq 
    \frac{ \overline{F}^{x}_{n}(q_{n,k_{n}}^{x})}
    { \overline{F}^{x}(u_{n}^{x})} 
    =
    \frac{ \overline{F}^{x}_{n}( F^{\leftarrow, x}_{n}\p{1-k_{n}/n})}
    { \overline{F}^{x}( F^{\leftarrow, x}\p{1-k_{n}/n})} 
    =
    1.
    \end{align*}
    In particular it holds that
    \begin{align*}
        \P\p{\frac{q_{n,k_{n}}^{x}}{u_{n}^{x}} > 1 + \varepsilon} 
        \leq
        \P\p{\widetilde{T}^{x}_{n}\p{1+\varepsilon}> 1}. 
    \end{align*}
    Since~\eqref{Empirical_tail_dist_consistency_equation} holds, then $\widetilde{T}^{x}_{n}\p{1+\varepsilon} \overset{\P}{\to} \p{1+\varepsilon}^{-1/\gamma(x)}$ and thus
    \begin{align*}
        \limsup_{n\to \infty} \P\p{\frac{q_{n,k_{n}}^{x}}{u_{n}^{x}} > 1 + \varepsilon} 
        \leq
        \lim_{n\to \infty}  \P\p{\widetilde{T}^{x}_{n}\p{1+\varepsilon}> 1} = 0.
    \end{align*}
    Using similar arguments we obtain that $\limsup_{n\to \infty} \P({q_{n,k_{n}}^{x}}/{u_{n}^{x}} < 1 - \varepsilon) = 0$.
\end{proof}
\begin{proof}[Proof of Theorem~\ref{Hill_Consistency_thm}]
    We consider the following decomposition
    \begin{align*}
     \widehat{\gamma}_{n}\p{x}
        =
        \int_{q_{n,k_{n}}^{x}/u_{n}^{x}}^{1}
        \frac{\widetilde{T}_{n}^{x}\p{s}}{s}\mathrm{d}s
        +
        \int_{1}^{\infty}
        \frac{\widetilde{T}_{n}^{x}\p{s}}{s}\mathrm{d}s.
    \end{align*}
    Proposition~\ref{Empirical_quantile_consistency_prop} yields that $q_{n,k_{n}}^{x}/u_{n}^{x} \overset{\P}{\to} 1$ and thus the first integral is $o_{\P}\p{1}$.\\
     The convergence of equation~\eqref{Empirical_tail_dist_consistency_equation} is uniform on compacts of $[1,\infty)$ and hence for all $t>1$ we have that
    \begin{align*}  
        \int_{1}^{t}
        \frac{\widetilde{T}_{n}^{x}\p{s}}{s}\mathrm{d}s \overset{\P}{\to}
        \int_{1}^{t}
        \frac{T^{x}\p{s}}{s}\mathrm{d}s, \ \  \text{as} \ n \to \infty.
    \end{align*}
    Since $\lim_{t \to \infty} \int_{1}^{t}
        (T^{x}(s)/s)\mathrm{d}s = \gamma (x)$, consistency of $\widehat{\gamma}_{n}$ follows if for every $\eta > 0$
    \begin{align*}
        \lim_{t \to \infty}\limsup_{n\to \infty}\P\p{  \int_{t}^{\infty}
        \frac{\widetilde{T}_{n}^{x}\p{s}}{s}\mathrm{d}s > \eta} = 0,
    \end{align*}
    Now, since
        \begin{align*}
            \lim_{t \to \infty} \int_{t}^{\infty}
        \frac{\widetilde{T}_{n}^{x}\p{s}}{s}\mathrm{d}s  
        =
         \frac{g(x)}{g_{n}(x)}\lim_{t \to \infty} \int_{t}^{\infty}
        \frac{\widetilde{V}_{n}^{x}\p{s}}{g(x)s}\mathrm{d}s,
        \end{align*}
        and $g(x)/g_{n}(x) \overset{\P}{\to} 1$, it suffices to show that for every $\eta >0$,
        \begin{align*}
        \lim_{t \to \infty}\limsup_{n\to \infty}\P\p{  \int_{t}^{\infty}
        \frac{\widetilde{V}_{n}^{x}\p{s}}{g(x)s}\mathrm{d}s > \eta} = 0.
    \end{align*}
    Fix $\eta > 0$. We get
     \begin{align*}
          &\P\p{  \int_{t}^{\infty}
        \frac{\widetilde{V}_{n}^{x}\p{s}}{g(x)s}\mathrm{d}s > \eta}
        \leq
        \frac{1}{\eta}
        \int_{t}^{\infty}
        \frac{\E{\widetilde{V}_{n}^{x}\p{s}}}{g(x)s}\mathrm{d}s \\
         &=
        \frac{1}{\eta}
        \int_{t}^{\infty}
        \frac{\overline{F}^{x}(su_{n}^{x})}
         {s\overline{F}^{x}(u_{n}^{x})}
         \int 
          (su_{n}^{x})^{-(1/\gamma(x-uh_{n}) -1/\gamma(x))}
         \frac{L^{x-uh_{n}}(s u_{n}^{x})}
         {L^{x}(su_{n}^{x})} 
        \frac{g(x - h_{n}u)}{g(x)}   K\p{u}
         \mathrm{d}u \mathrm{d}s
         \\
         &\leq 
        \frac{1}{\eta}
        \int_{t}^{\infty}
        \frac{\overline{F}^{x}(su_{n}^{x})}
         {s\overline{F}^{x}(u_{n}^{x})}
             \sup_{u \in [-1,1]}\left( 
          (su_{n}^{x})^{-(1/\gamma(x-uh_{n}) -1/\gamma(x))}
         \frac{L^{x-uh_{n}}(s u_{n}^{x})}
         {L^{x}(su_{n}^{x})} 
        \frac{g(x - h_{n}u)}{g(x)} \right)   \mathrm{d}s
        \\
         &\leq 
        \frac{1}{\eta}
        \int_{t}^{\infty}
        \frac{\overline{F}^{x}(su_{n}^{x})}
         {s\overline{F}^{x}(u_{n}^{x})}
             \sup_{u \in [-1,1]}\left(
        \frac{g(x - h_{n}u)}{g(x)} \right) 
         s^{ \left(c_{\gamma}+ c_{L}\right) h_n }
           (u_{n}^{x})^{\left(c_{\gamma}+ c_{L}\right) h_n }
        \mathrm{d}s
        \\
         &\leq 
        \frac{ \text{cst}}{\eta}
        \sup_{u \in [-1,1]}\left(
        \frac{g(x - h_{n}u)}{g(x)} \right)  
        (u_{n}^{x})^{\left(c_{\gamma}+ c_{L}\right) h_n }
        \int_{t}^{\infty}
         s^{-1/\gamma(x)+\varepsilon -1 + \left(c_{\gamma}+ c_{L}\right) h_n }
        \mathrm{d}s,
         \end{align*}
         for $\varepsilon \in (0, 1/\gamma(x))$, using Potter's bound. Let $\varepsilon' \in (0, 1/\gamma(x) - \varepsilon)$ be given. It is possible to find $n'$ such that $\left(c_{\gamma}+ c_{L}\right) h_n \leq \varepsilon'$ for all $n\geq n'$. Hence for $n\geq n'$ we have that the integrand $s \mapsto s^{-1/\gamma(x)+\varepsilon -1 + \left(c_{\gamma}+ c_{L}\right) h_n }$ is integrable and consequently
         \begin{align*}
            &\limsup_{n\to \infty}\P\p{  \int_{t}^{\infty}
            \frac{\widetilde{V}_{n}^{x}\p{s}}{g(x)s}\mathrm{d}s > \eta}
            \\&\leq
            \limsup_{n\to \infty}
            \left\{
            \frac{ \text{cst}}{\eta}
        \sup_{u \in [-1,1]}\left(
        \frac{g(x - h_{n}u)}{g(x)} \right)  
        (u_{n}^{x})^{\left(c_{\gamma}+ c_{L}\right) h_n }
        \int_{t}^{\infty}
         s^{-1/\gamma(x)+\varepsilon -1 + \left(c_{\gamma}+ c_{L}\right) h_n }\mathrm{d}s\right\} \\
         &=
         \mathcal{O}\p{t^{-1/\gamma(x) + \varepsilon + \varepsilon'}}, 
         \end{align*}
         which vanishes as $t \to \infty$.
\end{proof}

\subsection{Proof of Theorem~\ref{Thm:Functional_CLT_Tail_empirical_process_Random_level} (Functional CLT)}
Given the original sample $\curly{\p{X_1, Y_1},\ldots,\p{X_n, Y_n}}$, we consider a pseudo-sample 
\begin{align*}
    \curly{\p{X^{\dag}_1, Y^{\dag}_1},\ldots,\p{X^{\dag}_n, Y^{\dag}_n}},
\end{align*}
where the blocks $$\curly{\p{X^{\dag}_{(i-1)r_{n} +1}, Y^{\dag}_{(i-1)r_{n} +1}},\ldots,\p{X^{\dag}_{ir_{n}}, Y^{\dag}_{ir_{n}}}}, \ i = 1, \ldots , m_n, $$ are mutually independent and each have the same distribution as the first original sample block $\{(X_1, Y_1),\ldots,(X_{r_n}, Y_{r_n})\}$.\\
Define for every $s > 0$
\begin{align*}
    &\widetilde{V}^{\dag,x}_{n}(s) = 
    \frac{1}{nh_{n} \overline{F}^{x}(u_{n}^{x})}
    \sum_{i=1}^{m_n}
    \sum_{j=1}^{r_n}K\left(\frac{x-X^{\dag}_{(i-1)r_{n} +j}}{h_n}\right)1_{\left\{Y^{\dag}_{(i-1)r_{n} +j} > su_{n}^{x}\right\}}, \\
      & \widetilde{\mathbb{V}}^{\dag,x}_{n}(s) =  \sqrt{k_{n}h_{n}}\curly{\widetilde{V}^{\dag,x}_{n}(s) - \E{\widetilde{V}^{\dag,x}_{n}(s) }}.
\end{align*}
We establish weak convergence for $\widetilde{\mathbb{V}}^{\dag,x}_{n}$ and relate it to the original $\widetilde{\mathbb{V}}^{x}_{n}$ and in turn $\widetilde{\mathbb{T}}^{x}_{n}$.\\
The first step is to establish finite dimensional convergence for $\widetilde{\mathbb{V}}^{\dag,x}_{n}$.
\begin{lemma}\label{Lemma_covariance_of_V}
    Assume conditions~\ref{Assu:gamma_lip}-\ref{Assu:n_h_F_limit}, and~\ref{rate_condition}-\ref{anti_clustering_condition}. Then for all $s,t > 0$
    \begin{align}\label{Covariance_of_V}
        \lim_{n \to \infty} \Cov{\p{\widetilde{\mathbb{V}}^{\dag,x}_{n}(s), \widetilde{\mathbb{V}}^{\dag,x}_{n}(t)}} = \p{s \vee t}^{-1/\gamma(x)}g(x)\int K^{2}(u)\mathrm{d}u. 
    \end{align}
\end{lemma}
\begin{proof}
Independence between the blocks of the pseudo-sample and stationarity yield
    \begin{align*}
        &\Cov{\p{\widetilde{\mathbb{V}}^{\dag,x}_{n}(s), \widetilde{\mathbb{V}}^{\dag,x}_{n}(t)}} 
          = 
        \frac{m_{n}r_{n}}{nh_{n}\overline{F}^{x}(u_{n}^{x})}
        \biggl[  \E{ 1_{\curly{Y_{0} > \p{s \vee t}u_{n}^{x}}} K^{2}\left(\frac{x-X_{0}}{h_n}\right)}
        \\
        & \quad \quad - r_{n} \E{K\left(\frac{x-X_0}{h_n}\right)1_{\curly{Y_0 > su_{n}^{x}}}}
        \E{K\left(\frac{x-X_0}{h_n}\right)1_{\curly{Y_0 > tu_{n}^{x}}}} \\
        & \quad \quad + \sum_{j =1}^{r_n}\p{1-\frac{j}{r_n}}
        \biggl\{\E{K\left(\frac{x-X_0}{h_n}\right)1_{\left\{Y_0 > su_{n}^{x}\right\}}
       K\left(\frac{x-X_j}{h_n}\right)1_{\left\{Y_j > tu_{n}^{x}\right\}}} \\
       & \qquad \qquad \qquad \qquad \qquad +
       \E{K\left(\frac{x-X_0}{h_n}\right)1_{\left\{Y_0 > tu_{n}^{x}\right\}}
       K\left(\frac{x-X_j}{h_n}\right)1_{\left\{Y_j > su_{n}^{x}\right\}}}
       \biggr\}
        \biggr].
    \end{align*}
    Since $m_n = \floor{n/r_n}$ we have that $\lim_{n \to \infty} m_{n}r_{n}/n = 1$. The limit of the first term follows from Lemma~\ref{Expected_indicator_kernel_lemma}
    \begin{align*}
        \lim_{n \to \infty} 
        \frac{m_{n}r_{n}}{n}
        \frac{\E{ 1_{\curly{Y_{0} > \p{s \vee t}u_{n}^{x}}} K^{2}\left(\frac{x-X_{0}}{h_n}\right)}}{h_{n}\overline{F}^{x}(u_{n}^{x})} = 
        \p{s \vee t}^{-1/\gamma(x)}g(x)\int K^{2}\p{u}\mathrm{d}u,
    \end{align*}
    which is the asymptotic covariance. It remains to show that the second and third terms disappear asymptotically.
    The second term vanishes by the assumption $ \lim_{n \to \infty} {r_{n}h_{n}\overline{F}^{x}(u_{n}^{x})} = 0$.
    Since $m_{n}r_{n}/n = \mathcal{O}(1)$, the leading term in the third term is 
     asymptotically proportional to
     \begin{align*}
         \frac{1}{h_n\overline{F}^{x}(u_{n}^{x})} \sum_{j=1}^{r_n}
         \E{K\left(\frac{x-X_0}{h_n}\right)1_{\left\{Y_0 > tu_{n}^{x}\right\}}
       K\left(\frac{x-X_j}{h_n}\right)1_{\left\{Y_j > su_{n}^{x}\right\}}}.
     \end{align*}
To examine the limit, fix $m \in \N$, and consider
\begin{align*}
         &\limsup_{n \to \infty}
         \frac{1}{h_n \overline{F}^{x}(u_{n}^{x})} 
         \sum_{j=1}^{m-1}
         \E{K\left(\frac{x-X_0}{h_n}\right)1_{\left\{Y_0 > tu_{n}^{x}\right\}}
        K\left(\frac{x-X_j}{h_n}\right)1_{\left\{Y_j > su_{n}^{x}\right\}}} \\
        &= \limsup_{n \to \infty} 
         \sum_{j=1}^{m-1} h_{n}
        \int_{[-1,1]^2}  \biggl\{
         \frac{\CE{1_{\left\{Y_0 > tu_{n}^{x}\right\}}
       1_{\left\{Y_j > su_{n}^{x}\right\}}}{X_0 = x + w_{1}h_n, X_j = x + w_{2}h_n}}
         {\overline{F}^{x}(u_{n}^{x})} \\
         & \qquad \qquad \qquad \qquad \qquad \qquad 
         \times g(x + w_{1}h_{n},x + w_{2}h_{n})K\left(w_{1}\right)K\left(w_{2}\right) \biggr\}
         \mathrm{d}(w_1, w_2)=0.
     \end{align*}
     Fix $\varepsilon > 0$. By condition~\ref{anti_clustering_condition}, it is possible to pick $m_0$ large enough so that for all $m \geq m_0$
     \begin{align*}
        \limsup_{n \to \infty}
         \frac{1}{h_{n}\overline{F}^{x}(u_{n}^{x})} 
         \sum_{j=m}^{r_n}
         \E{K\left(\frac{x-X_0}{h_n}\right)1_{\left\{Y_0 > tu_{n}^{x}\right\}}
        K\left(\frac{x-X_j}{h_n}\right)1_{\left\{Y_j > su_{n}^{x}\right\}}}
        \leq \varepsilon.
     \end{align*}
     This finishes the proof.
\end{proof}

\begin{theorem} \label{Fidis_theorem}
    Assume conditions~\ref{Assu:gamma_lip}-\ref{Assu:n_h_F_limit},~\ref{rate_condition}-\ref{anti_clustering_condition}. Then the finite dimensional distributions of $ \widetilde{\mathbb{V}}^{\dag,x}_{n}$ converge to those of a Gaussian process $ \mathbb{V}^{x}$ with covariance function defined in~\eqref{Covariance_of_V}.
\end{theorem}
\begin{proof}
    For $i = 1,\ldots, m_{n}$, let 
    \begin{align*}
        S_{n,i}^{x}(s) = \sum_{j=1}^{r_n}K\left(\frac{x-X_{(i-1)r_{n} +j}}{h_n}\right)1_{\curly{Y_{(i-1)r_{n} +j} > su_{n}^{x}}}, 
        \quad 
        \overline{S}_{n,i}^{x}(s) = S_{n,i}^{x}(s) - \E{S_{n,1}^{x}(s)},
        \\
        S_{n,i}^{\dag, x}(s) = \sum_{j=1}^{r_n}K\left(\frac{x-X^{\dag}_{(i-1)r_{n} +j}}{h_n}\right)1_{\curly{Y^{\dag}_{(i-1)r_{n} +j} > su_{n}^{x}}}, 
        \quad 
        \overline{S}_{n,i}^{\dag,x}(s) = S_{n,i}^{\dag,x}(s) - \E{S_{n,1}^{x}(s)},
    \end{align*}
    for $s>0$. Then $\widetilde{\mathbb{V}}^{\dag,x}_{n}$ can be written as 
    \begin{align*}
        \widetilde{\mathbb{V}}^{\dag,x}_{n}(s) =  
        \frac{1}{\sqrt{k_{n}h_{n}}}
        \sum_{i=1}^{m_n}
        \overline{S}_{n,i}^{\dag,x}(s),
    \end{align*}
    where the summands are i.i.d. In the calculations below we show that the conditions of the Lindeberg--Lévy CLT are fulfilled. By Lemma~\ref{Lemma_covariance_of_V} 
    it only remains to show the negligibility condition. Since $\overline{S}_{n,1}^{x}(s)$ and $\overline{S}_{n,1}^{\dag, x}(s)$ have the same distribution, we may establish it for the former.
    Note that  $m_n/(nh_{n}\overline{F}^{x}(u_{n}^{x})) = \mathcal{O}((r_{n}h_{n}\overline{F}^{x}(u_{n}^{x}))^{-1})$. Thus the negligibility condition thus follows if the expectation below is $o(r_{n}h_{n}\overline{F}^{x}(u_{n}^{x}))$:
    \begin{align*}
        &\E{\p{\overline{S}_{n,1}^{x}(s)}^{2}
        1_{\curly{\lvert \overline{S}_{n,1}^{x}(s) \rvert > \varepsilon \sqrt{k_{n}h_{n}}}}}=
        \\
        &
        \E{
        \left(
        \sum_{j=1}^{r_n}K\left(\frac{x-X_{j}}{h_n}\right)1_{\curly{Y_{j} > su_{n}^{x}}}
        \right)^{2}
        1_{\curly{\lvert \overline{S}_{n,1}^{x}(s) \rvert > \varepsilon \sqrt{k_{n}h_{n}}}}} 
        \\
        & \quad + 
        \p{r_{n}\E{K\left(\frac{x-X_{0}}{h_n}\right)1_{\curly{Y_{0} > su_{n}^{x}}}}}^{2}
        \P{\p{\left\lvert \overline{S}_{n,1}^{x}(s) \right\rvert > \varepsilon \sqrt{k_{n}h_{n}}}}
        \\
        & \quad - 2r_{n}\E{K\left(\frac{x-X_{0}}{h_n}\right)1_{\curly{Y_{0} > su_{n}^{x}}}}
        \sum_{j=1}^{r_n}\E{K\left(\frac{x-X_{j}}{h_n}\right)1_{\curly{Y_{j} > su_{n}^{x}}}
        1_{\curly{\lvert \overline{S}_{n,1}^{x}(s) \rvert > \varepsilon \sqrt{k_{n}h_{n}}}}}.
    \end{align*}
    Consider the second term. Lemma~\ref{Expected_indicator_kernel_lemma} yields that it is bounded by
    \begin{align*}
          r_{n}^{2} \ \E{K\left(\frac{x-X_{0}}{h_n}\right)1_{\curly{Y_{0} > su_{n}^{x}}}}^{2} = 
        r_{n}^{2} \mathcal{O}\p{\p{h_{n}\overline{F}^{x}(u_{n}^{x})}^{2}}
        = 
        o\p{r_{n}h_{n}\overline{F}^{x}(u_{n}^{x})}.
    \end{align*}
    The third term follows similarly. 
    We turn the first term and decompose it according to an expansion of the square, yielding diagonal terms $I_1$ and off-diagonal terms $2I_2$. 
    We get
    \begin{align*}
        I_1 &\leq  \frac{1}{\varepsilon \sqrt{k_{n}h_{n}}}
        \sum_{j=1}^{r_n} \E{K\left(\frac{x-X_{j}}{h_n}\right)^{2}1_{\curly{Y_{j} > su_{n}^{x}}} \left\lvert \overline{S}_{n,1}^{x}(s) \right\rvert }
        \\
        & \leq 
        \frac{1}{\varepsilon \sqrt{k_{n}h_{n}}}
        \sum_{j=1}^{r_n} \E{K\left(\frac{x-X_{j}}{h_n}\right)^{2}1_{\curly{Y_{j} > su_{n}^{x}}} \p{S_{n,1}^{x}(s) + \E{S_{n,1}^{x}(s)}} }
    \end{align*}
    The latter of the two terms is of order
    \begin{align*}        &\frac{\mathcal{O}\p{r_{n}h_{n}\overline{F}^{x}(u_{n}^{x})}}{\varepsilon \sqrt{k_{n}h_{n}}}
        \sum_{j=1}^{r_n} \E{K\left(\frac{x-X_{j}}{h_n}\right)^{2}1_{\curly{Y_{j} > su_{n}^{x}}}} = \frac{\mathcal{O}\p{\p{r_{n}h_{n}\overline{F}^{x}(u_{n}^{x})}^{2}}}{\varepsilon \sqrt{k_{n}h_{n}}} = o\p{r_{n}h_{n}\overline{F}^{x}(u_{n}^{x})},
    \end{align*}
    using condition~\ref{rate_condition} and Lemma~\ref{Expected_indicator_kernel_lemma}.
    The former of the two terms can be rewritten as
    \begin{align*}
        &\frac{ r_{n}}{\varepsilon \sqrt{k_{n}h_{n}}}
        \Biggl(\E{K\left(\frac{x-X_{0}}{h_n}\right)^{3} 1_{\curly{Y_{0} > su_{n}^{x}}}}\\
        &\qquad \qquad + \sum_{j=1}^{r_n}\p{1-\frac{j}{r_n}} \biggr\{
        \E{K\left(\frac{x-X_{j}}{h_n}\right)^{2} K\left(\frac{x-X_{0}}{h_n}\right)
        1_{\curly{Y_{j} > su_{n}^{x}}}  1_{\curly{Y_{0} > su_{n}^{x}}}} \\
        & \qquad \qquad   +
        \E{K\left(\frac{x-X_{0}}{h_n}\right)^{2} K\left(\frac{x-X_{j}}{h_n}\right)
        1_{\curly{Y_{0} > su_{n}^{x}}}  1_{\curly{Y_{j} > su_{n}^{x}}}}
        \biggr\}
        \Biggr),
    \end{align*}
    where the first term is order $o(r_{n}h_{n}\overline{F}^{x}(u_{n}^{x}))$
    as above using condition~\ref{rate_condition} and Lemma~\ref{Expected_indicator_kernel_lemma}. The leading term in the sum is dealt with 
    as in the proof of Lemma~\ref{Lemma_covariance_of_V}, from which we infer that the expression is $o(r_{n}h_{n}\overline{F}^{x}(u_{n}^{x}))$ and consequently $I_1 =  o(r_{n}h_{n}\overline{F}^{x}(u_{n}^{x}))$.
    
    Consider $I_2$. Given $\delta > 0$, condition~\ref{anti_clustering_condition} yields that we can choose a positive integer $m$ such that for sufficiently large $n$ it holds that
    \begin{align*}
         \sum_{j=m}^{r_n}
          \frac{\E{K\left(\frac{x-X_{j}}{h_n}\right) K\left(\frac{x-X_{0}}{h_n}\right)
        1_{\curly{Y_{j} > su_{n}^{x}}}  1_{\curly{Y_{0} > su_{n}^{x}}}}}{h_{n}\overline{F}^{x}(u_{n}^{x})}
        \leq \delta.
    \end{align*}
    Split the inner sum of $I_2$ according to the chosen $m$.
     
    This yields two terms. By stationarity the first term can be rewritten as 
    \begin{align*}
         &\sum_{j=1}^{m}\p{m-j}
        \E{K\left(\frac{x-X_{0}}{h_n}\right)K\left(\frac{x-X_{j}}{h_n}\right)1_{\curly{Y_{0} > su_{n}^{x}}}1_{\curly{Y_{j} > su_{n}^{x}}}
        1_{\curly{\lvert \overline{S}_{n,1}^{x}(s) \rvert > \varepsilon \sqrt{k_{n}h_{n}}}}} \\
        &\leq
        m\sum_{j=1}^{m}
        \E{K\left(\frac{x-X_{0}}{h_n}\right)K\left(\frac{x-X_{j}}{h_n}\right)1_{\curly{Y_{0} > su_{n}^{x}}}1_{\curly{Y_{j} > su_{n}^{x}}}} 
        \\ 
        &\leq 
        \frac{m}{r_n} r_{n}h_{n}\overline{F}^{x}(u_{n}^{x})
        \sum_{j=1}^{r_n}
        \frac{\E{K\left(\frac{x-X_{0}}{h_n}\right)K\left(\frac{x-X_{j}}{h_n}\right)1_{\curly{Y_{0} > su_{n}^{x}}}1_{\curly{Y_{j} > su_{n}^{x}}}}}{h_{n}\overline{F}^{x}(u_{n}^{x})} \\
        &  = \frac{\mathcal{O}\p{{r_{n}h_{n}\overline{F}^{x}(u_{n}^{x})}}}{r_n} = 
        o\p{r_{n}h_{n}\overline{F}^{x}(u_{n}^{x})}.
    \end{align*}
    The second term of $I_2$ satisfies
    \begin{align*}
        &\sum_{i=1}^{r_n} \sum_{j=i + m}^{r_n} \E{K\left(\frac{x-X_{i}}{h_n}\right)K\left(\frac{x-X_{j}}{h_n}\right)1_{\curly{Y_{i} > su_{n}^{x}}}1_{\curly{Y_{j} > su_{n}^{x}}}
        1_{\curly{\lvert \overline{S}_{n,1}^{x}(s) \rvert > \varepsilon \sqrt{k_{n}h_{n}}}}} 
         \leq \delta r_{n}h_{n}\overline{F}^{x}(u_{n}^{x}).
    \end{align*}
    Gathering the above yields that
    \begin{align*}
        \limsup_{n \to \infty} \frac{I_2}{r_{n}h_{n}\overline{F}^{x}(u_{n}^{x})} \leq \delta,
    \end{align*}
    and since $\delta$ can be made arbitrarily small we conclude that $I_2 = o(r_{n}h_{n}\overline{F}^{x}(u_{n}^{x}))$ which concludes the proof.
\end{proof}
Now we show that terms constructed from the incomplete block $\{(X_{m_{n}r_{n} + 1}, Y_{m_{n}r_{n} + 1}),\ldots,\p{X_{n}, Y_{n}}\}$ are asymptotically negligible.
\begin{lemma} \label{Incomplete_block_lemma}
    Assume conditions~\ref{Assu:gamma_lip}-\ref{Assu:n_h_F_limit} and~\ref{rate_condition}. Then
    \begin{align*}
         \sup_{s \in [a,b]}
         \left\lvert
         \sum_{j=m_{n}r_{n} + 1}^{n}
         \frac{K\left(\frac{x-X_{j}}{h_n}\right)1_{\curly{Y_{j} > su_{n}^{x}}}
         - \E{K\left(\frac{x-X_{j}}{h_n}\right)1_{\curly{Y_{j} > su_{n}^{x}}}}}{\sqrt{nh_{n} \overline{F}^{x}(u_{n}^{x})}}
         \right\rvert
         \overset{\P}{\to} 0.
    \end{align*}
\end{lemma}
\begin{proof}
    Since $s \mapsto K\left((x-X_{j})/h_{n}\right)1_{\curly{Y_{j} > su_{n}^{x}}}$ is monotonically decreasing, and $(n - m_{n}r_{n}) \leq r_{n}$, i.e. the last incomplete block consists of fewer terms than a complete block, we get 
    \begin{align*}
        &\sup_{s \in [a,b]}
         \left\lvert
         \sum_{j=m_{n}r_{n} + 1}^{n}
         \frac{K\left(\frac{x-X_{j}}{h_n}\right)1_{\curly{Y_{j} > su_{n}^{x}}}
         - \E{K\left(\frac{x-X_{j}}{h_n}\right)1_{\curly{Y_{j} > su_{n}^{x}}}}}{\sqrt{nh_{n} \overline{F}^{x}(u_{n}^{x})}}
         \right\rvert
         \\ 
         & \leq
         \frac{1}{\sqrt{nh_{n} \overline{F}^{x}(u_{n}^{x})}}
         \Bigg\{
         \sum_{j=m_{n}r_{n} + 1}^{n}
          K\left(\frac{x-X_{j}}{h_n}\right)1_{\curly{Y_{j} > au_{n}^{x}}}   +
         r_{n} \E{K\left(\frac{x-X_{j}}{h_n}\right)1_{\curly{Y_{j} > au_{n}^{x}}}} \Bigg\}.
   \end{align*}
    The sum has variance
    \begin{align*}
         &\frac{1}{nh_{n} \overline{F}^{x}(u_{n}^{x})}
         \Var\p{\sum_{j= 1}^{(n-m_{n}r_{n})}
          K\left(\frac{x-X_{j}}{h_n}\right)1_{\curly{Y_{j} > au_{n}^{x}}} } 
          \\
         &\leq
         \frac{1}{nh_{n} \overline{F}^{x}(u_{n}^{x})}
         \Var\p{\sum_{j= 1}^{r_{n}}
          K\left(\frac{x-X_{j}}{h_n}\right)1_{\curly{Y_{j} > au_{n}^{x}}}} 
          = \frac{1}{m_n}  \Var\p{\widetilde{\mathbb{V}}^{\dag,x}_{n}(a)} \to 0,
    \end{align*}   
    where we used that the joint process $\left\{\left(X_{j}, Y_{j}\right)\right\}$ is stationary, Lemma~\ref{Expected_indicator_kernel_lemma} and  
    Lemma~\ref{Lemma_covariance_of_V} . Further,
    by condition~\ref{rate_condition},
    \begin{align*}
        &\frac{r_{n}}{\sqrt{nh_{n} \overline{F}^{x}(u_{n}^{x})}} \E{K\left(\frac{x-X_{j}}{h_n}\right)1_{\curly{Y_{j} > au_{n}^{x}}}}  =
        \frac{r_{n}}{\sqrt{nh_{n} \overline{F}^{x}(u_{n}^{x})}} \mathcal{O}\p{h_{n} \overline{F}^{x}(u_{n}^{x})}=
        \frac{\mathcal{O}\p{r_{n} h_{n} \overline{F}^{x}(u_{n}^{x})}}{\sqrt{nh_{n} \overline{F}^{x}(u_{n}^{x})}}
        \to 0.
    \end{align*}
    These two limits now immediately imply the desired result.
\end{proof}

\begin{lemma}
      Assume conditions~\ref{Assu:gamma_lip}-\ref{Assu:n_h_F_limit} and~\ref{rate_condition}-\ref{beta_mixing_condition}. Then the finite dimensional limit distributions of $ \widetilde{\mathbb{V}}^{\dag,x}_{n}$ and $ \widetilde{\mathbb{V}}^{x}_{n}$ coincide.
\end{lemma}
\begin{proof}
    By Lemma~\ref{Incomplete_block_lemma} the incomplete block is asymptotically negligible. Thus we apply Lemma~\ref{coinciding_fidis_lemma} with $h_{n,r_{n}}\p{\p{X_1, Y_1},\ldots,\p{X_{r_{n}}, Y_{r_{n}}}}$ and $f_{n}\p{\p{X_1, Y_1},\ldots,\p{X_{r_{n}}, Y_{r_{n}}}}$ both equal to
    \begin{align*}
        \sum_{j=1}^{r_n}
        \frac{K\left(\frac{x-X_{j}}{h_n}\right)1_{\left\{Y_{j} > su_{n}^{x}\right\}}}{\sqrt{nh_{n} \overline{F}^{x}(u_{n}^{x})}}.   
    \end{align*} 
    The result follows if $\lim_{n \to \infty} m_{n}\beta_{\ell_{n}} = 0$ and $\lim_{n \to \infty} m_{n} \E{f_{n}^{2}\p{\p{X_1, Y_1},\ldots,\p{X_{\ell_{n}}, Y_{\ell_{n}}}}} = 0.$
    The former condition holds by  condition~\ref{beta_mixing_condition},
    while the latter condition follows from Lemma~\ref{Lemma_covariance_of_V} since 
    \begin{align*}
       m_{n} \E{\p{ 
        \sum_{j=1}^{\ell_n}
        \frac{K\left(\frac{x-X_{j}}{h_n}\right)1_{\left\{Y_{j} > su_{n}^{x}\right\}}}
        {\sqrt{nh_{n} \overline{F}^{x}(u_{n}^{x})}}
        }^{2}} 
        \leq
        \frac{l_{n}}{r_{n}} \Var\p{\widetilde{\mathbb{V}}^{\dag,x}_{n}(s)} 
        \to 0. 
    \end{align*}
\end{proof}
This concludes the finite-dimensional convergence and we now proceed to proving a functional central limit theorem.

Define the modulus of continuity over an arbitrary compact interval $C$ as
\begin{align*}
    w(f, \delta, C) = 
    \sup_{\substack{s,t \in C \\ \lvert t-s \rvert < \delta}}
    \left\lvert f(t)-f(s) \right\rvert,
\end{align*}
for $f: C \to \R$.
\begin{theorem} \label{theorem_fct_clt_V}
     Assume conditions~\ref{Assu:gamma_lip}-\ref{Assu:n_h_F_limit},~\ref{rate_condition}-\ref{beta_mixing_condition}. Then $\widetilde{\mathbb{V}}^{x}_{n}$ converges weakly to the centred Gaussian process ${\mathbb{V}}^{x}$ in $\mathbb{D}((0,\infty))$ endowed with the $J_1$-topology, and the process $\mathbb{V}^{x}$ admits an almost surely continuous version.
\end{theorem}
\begin{proof}
    Theorem~\ref{Fidis_theorem} yields that the finite dimensional distributions of $\widetilde{\mathbb{V}}^{x}_{n}$ converges to that of the proposed Gaussian process. It thus remains to prove asymptotic equicontinuity.
    According to Theorem~\ref{equicontinuity_theorem} the statement above follows if $\widetilde{\mathbb{V}}^{x}_{n}$ is asymptotically equicontinuous w.r.t. the uniform norm on each interval $[a,b]$ with $0 < a < b$, i.e if for all $\varepsilon > 0$
    \begin{align*}
        \lim_{\delta \to 0}
        \limsup_{n \to \infty}
        \P\p{w\p{\widetilde{\mathbb{V}}^{x}_{n}, \delta, [a,b]} > \varepsilon} = 0.
    \end{align*}
    To show this limit, we relate the asymptotic equicontinuity of $\widetilde{\mathbb{V}}^{x}_{n}$ to the pseudo-sample version $\widetilde{\mathbb{V}}^{\dag, x}_{n}$ as follows. Define the processes 
    \begin{align*}
        &\widetilde{\mathbb{V}}^{x}_{n,1} =
        \sum_{\substack{1 \leq i \leq m_{n} \\ i \ \text{odd}}}
        \frac{ \overline{S}_{n,i}^{x}}{\sqrt{nh_{n} \overline{F}^{x}(u_{n}^{x})}}, 
        \quad \ \
        \widetilde{\mathbb{V}}^{x}_{n,2} =
        \sum_{\substack{1 \leq i \leq m_{n} \\ i \ \text{even}}}
        \frac{ \overline{S}_{n,i}^{x}}{\sqrt{nh_{n} \overline{F}^{x}(u_{n}^{x})}} 
        \\
        &\widetilde{\mathbb{V}}^{\dag, x}_{n,1} =
        \sum_{\substack{1 \leq i \leq m_{n} \\ i \ \text{odd}}}
        \frac{ \overline{S}_{n,i}^{\dag,x}}{\sqrt{nh_{n} \overline{F}^{x}(u_{n}^{x})}},
        \quad
        \widetilde{\mathbb{V}}^{\dag, x}_{n,2} =
        \sum_{\substack{1 \leq i \leq m_{n} \\ i \ \text{even}}}
        \frac{ \overline{S}_{n,i}^{\dag,x}}{\sqrt{nh_{n} \overline{F}^{x}(u_{n}^{x})}}
    \end{align*}
    These processes are sums of blocks separated by a distance of the block length $r_{n}$ and thus Lemma~\ref{blocking_lemma} yields that 
    \begin{align*}
        \mathrm{d}_{\mathrm{TV}}
        \p{\mathcal{L}\p{\widetilde{\mathbb{V}}^{\dag, x}_{n,i}},
           \mathcal{L}\p{\widetilde{\mathbb{V}}^{x}_{n,i}}}
        \leq m_{n}\beta_{r_{n}}, \ \ i=1,2,
    \end{align*}
    where $ \mathrm{d}_{\mathrm{TV}}$ denotes the total variation distance.
    The processes defined above allows for the following decompositions
    \begin{align*}
        &\widetilde{\mathbb{V}}^{x}_{n} = 
        \widetilde{\mathbb{V}}^{x}_{n,1} +
        \widetilde{\mathbb{V}}^{x}_{n,2} +
        \widetilde{\mathbb{V}}^{x, \text{incomplete}}_{n}, \quad\widetilde{\mathbb{V}}^{\dag, x}_{n} = 
        \widetilde{\mathbb{V}}^{\dag, x}_{n,1} +
        \widetilde{\mathbb{V}}^{\dag, x}_{n,2} ,
    \end{align*}
    where \begin{align*}
        \widetilde{\mathbb{V}}^{x, \text{incomplete}}_{n}(s) := \sum_{j=m_{n}r_{n} + 1}^{n}
         \frac{K\left(\frac{x-X_{j}}{h_n}\right)1_{\curly{Y_{j} > su_{n}^{x}}}
         - \E{K\left(\frac{x-X_{j}}{h_n}\right)1_{\curly{Y_{j} > su_{n}^{x}}}}}{\sqrt{nh_{n} \overline{F}^{x}(u_{n}^{x})}}.
    \end{align*}
    Let $\varepsilon > 0$ be given and notice that $w\p{f + g, \delta, [a,b]} \leq w\p{f, \delta, [a,b]} + w\p{ g, \delta, [a,b]}$ for any $f,g : [a,b] \to \R$. The total variation distance yields that
    \begin{align*}
        &\P\p{w\p{\widetilde{\mathbb{V}}^{x}_{n}, \delta, [a,b]} > \varepsilon} 
        \\
        &\leq 
        \P\p{w\p{\widetilde{\mathbb{V}}^{x}_{n,1}, \delta, [a,b]} > \frac{\varepsilon}{3}}  
        + \P\p{w\p{\widetilde{\mathbb{V}}^{x}_{n,2}, \delta, [a,b]} > \frac{\varepsilon}{3}}  
        + \P\p{w\p{\widetilde{\mathbb{V}}^{x, \text{incomplete}}_{n}, \delta, [a,b]} > \frac{\varepsilon}{3}} 
        \\
        & \leq 
        \P\p{w\p{\widetilde{\mathbb{V}}^{\dag, x}_{n,1}, \delta, [a,b]} > \frac{\varepsilon}{3}}  
        + \P\p{w\p{\widetilde{\mathbb{V}}^{\dag, x}_{n,2}, \delta, [a,b]} > \frac{\varepsilon}{3}}  
         + \P\p{w\p{\widetilde{\mathbb{V}}^{x, \text{incomplete}}_{n}, \delta, [a,b]} > \frac{\varepsilon}{3}}\\
         &\quad + 2m_{n}\beta_{r_n}.
    \end{align*}
     Now, Lemma~\ref{Incomplete_block_lemma} implies asymptotic equicontinuity of $\widetilde{\mathbb{V}}^{x, \text{incomplete}}_{n}$ w.r.t. the uniform norm on compact positive intervals and hence 
     \begin{align*}
        \lim_{\delta \to 0}
        \limsup_{n \to \infty}
        \P\p{w\p{\widetilde{\mathbb{V}}^{x, \text{incomplete}}_{n}, \delta, [a,b]} > \frac{\varepsilon}{3}} = 0.
    \end{align*}
    Condition~\ref{beta_mixing_condition} yields
    $m_{n}\beta_{l_{n}} \to 0$ and $l_{n}/r_{n} \to 0$ which implies $m_{n}\beta_{r_n} \to 0$. Thus asymptotic equicontinuity of $\widetilde{\mathbb{V}}^{x}_{n}$ follows if we show asymptotic equicontinuity of $\widetilde{\mathbb{V}}^{\dag, x}_{n,1}$ and $\widetilde{\mathbb{V}}^{\dag, x}_{n,1}$ or equivalently $\widetilde{\mathbb{V}}^{\dag, x}_{n}$ since these processes are sums of the same i.i.d. summands.
    
    The summands of $\widetilde{\mathbb{V}}^{\dag, x}_{n}$ are i.i.d processes indexed on $[a,b]$ with càdlàg sample paths, the latter implying separability of each summand. Consequently we can apply Theorem~\ref{bracketing_clt} to show asymptotic equicontinuity. The Lindeberg condition was already shown in Theorem~\ref{Fidis_theorem}.
    With $v_{n} = nh_{n} \overline{F}^{x}(u_{n}^{x})$ define the random metric $\mathrm{d}_{n}$ on $[a,b]$ by
    \begin{align*}
        \mathrm{d}_{n}^{2}(s,t) =
        \frac{1}{v_n}\sum_{i=1}^{m_n} 
        \curly{\overline{S}_{n, i}^{\dag,x}(s) - \overline{S}_{n, i}^{\dag,x}(t)}^{2}.
    \end{align*}
    Monotonicity of $s \mapsto \overline{S}_{n, i}^{\dag,x}(s)$ yields
    \begin{align*}
        &\E{\mathrm{d}_{n}^{2}(s,t)}
        =
        \frac{m_n}{ nh_{n} \overline{F}^{x}(u_{n}^{x})}
        \E{\curly{\overline{S}_{n, i}^{\dag,x}(s) - \overline{S}_{n, i}^{\dag,x}(t)}^{2}} \\
        &\leq 
        \frac{m_n}{ nh_{n} \overline{F}^{x}(u_{n}^{x})} \p{
        \E{\overline{S}_{n, i}^{\dag,x}(s)^{2}} - \E{\overline{S}_{n, i}^{\dag,x}(t)^{2}}} 
        \to 
        \p{s^{-1/\gamma(x)} -  t^{-1/\gamma(x)}}g(x)\int K^{2}\p{u}\mathrm{d}u,
    \end{align*}
    where Lemma~\ref{Lemma_covariance_of_V} was invoked. Since the limit function is continuous, Dini's Theorem implies that the convergence holds uniformly on compact intervals bounded away from zero. Thus for every sequence $\curly{\delta_{n}}$ decreasing to zero it holds that
    \begin{align*}
        &\lim_{n \to \infty} 
        \sup_{\substack{s,t \in [a,b] \\ \lvert s-t \rvert < \delta_{n}}}
        \E{\mathrm{d}_{n}^{2}(s,t)} \leq
        \lim_{n \to \infty} 
        \sup_{\substack{s,t \in [a,b] \\ \lvert s-t \rvert < \delta_{n}}}
        \curly{\frac{m_n}{ nh_{n} \overline{F}^{x}(u_{n}^{x})} \p{
        \E{\overline{S}_{n, i}^{\dag,x}(s)^{2}} - \E{\overline{S}_{n, i}^{\dag,x}(t)^{2}}} }
        = 0.
    \end{align*}
    In order to prove the random entropy integral condition, define the map $\mathcal{E}_{s}$ on $\ell_{0}(\R)$ by
    \begin{align*}
        \mathcal{E}_{s}\p{\boldsymbol{y}, \boldsymbol{x}} 
        =
        \sum_{j \in \Z}1_{\curly{y_{j}>s}}K(x_j), \quad \p{\boldsymbol{y}, \boldsymbol{x}} = \p{\p{y_j,x_j}, j \in \Z},
    \end{align*}
    and let $\mathcal{G}$ be the class of functions on $\ell_{0}(\R)$ defined by $\mathcal{G}= \curly{ \mathcal{E}_{s}, a \leq s \leq b}$. In addition define a sequence of processes $\curly{\Z_n}$ indexed by $\mathcal{G}$ by
    \begin{align*}
        \Z_n\p{\mathcal{E}_{s}}
        &= v_{n}^{-1}
        \sum_{i=1}^{m_n} \mathcal{E}_{s}
        \p{(u_{n}^{x})^{-1}\boldsymbol{Y}^{\dag}_{(i-1)r_{n} +1, ir_n}, h_{n}^{-1}\p{\boldsymbol{x} -\boldsymbol{X}^{\dag}_{(i-1)r_{n} +1, ir_n}}}
        =
         \frac{ 1}{nh_{n} \overline{F}^{x}(u_{n}^{x})} 
         \sum_{i=1}^{m_n}
         S_{n,i}^{\dag,x}(s),
    \end{align*}
    using the notation $\boldsymbol{X}^{\dag}_{(i-1)r_{n} +1, ir_n} = (X^{\dag}_{(i-1)r_{n} +1},\ldots X^{\dag}_{ir_n})$ for all $i=1,\ldots,m_n$, and similarly for $\boldsymbol{Y}^{\dag}$.\\
    If we define a random metric $\widetilde{d}_n$ on $\mathcal{G}$ by
    \begin{align*}
        \widetilde{d}_{n}^{2}\p{\mathcal{E}_{s}, \mathcal{E}_{t}}
        &=
        v_{n}^{-1}
        \sum_{i=1}^{m_n} \Biggl\{ \mathcal{E}_{s}
        \p{(u_{n}^{x})^{-1}\boldsymbol{Y}^{\dag}_{(i-1)r_{n} +1, ir_n}, h_{n}^{-1}\p{\boldsymbol{x} -\boldsymbol{X}^{\dag}_{(i-1)r_{n} +1, ir_n}}} \\
        & \qquad \qquad \quad - \mathcal{E}_{t}
        \p{(u_{n}^{x})^{-1}\boldsymbol{Y}^{\dag}_{(i-1)r_{n} +1, ir_n}, h_{n}^{-1}\p{\boldsymbol{x} -\boldsymbol{X}^{\dag}_{(i-1)r_{n} +1, ir_n}}}
        \Biggr\}^{2},
    \end{align*}
    then $ \widetilde{d}_{n}\p{\mathcal{E}_{s}, \mathcal{E}_{t}} = \mathrm{d}_{n}(s,t)$, and thus the random entropy condition can be proved for the class $\mathcal{G}$ with respect to the random metric $\widetilde{d}_n$. This reformulation allows us to invoke Lemma~\ref{uniform_entropy_lemma}. Consider the envelope
    \begin{align*}
        \boldsymbol{G}\p{\boldsymbol{y}, \boldsymbol{x}}
        = \sup_{s \in [a,b]}
        \left\lvert  \mathcal{E}_{s}\p{\boldsymbol{y}, \boldsymbol{x}}  \right\rvert 
        =
        \sup_{s \in [a,b]}
        \left\lvert   \sum_{j \in \Z}1_{\curly{y_{j}>s}}K(x_j)
        \right\rvert 
        =
        \sum_{j \in \Z}1_{\curly{y_{j}>a}}K(x_j).
    \end{align*}
    Lemma~\ref{Covariance_of_V} implies that 
    \begin{align*}
        &\limsup_{n \to \infty} \frac{m_n}{nh_{n} \overline{F}^{x}(u_{n}^{x})}
        \E{\boldsymbol{G}^{2} \p{(u_{n}^{x})^{-1}\boldsymbol{Y}^{\dag}_{1, r_n}, h_{n}^{-1}\p{\boldsymbol{x} -\boldsymbol{X}^{\dag}_{1, r_n}}}} 
        \\
        &=
        \limsup_{n \to \infty} \frac{m_n}{nh_{n} \overline{F}^{x}(u_{n}^{x})}
         \E{\p{\sum_{j=1}^{r_n}
        K\left(\frac{x-X_{j}}{h_n}\right)1_{\left\{Y_{j} > au_{n}^{x}\right\}}}^{2}}
         \\
        &=
        \limsup_{n \to \infty} \Var\p{\widetilde{\mathbb{V}}^{\dag, x}_{n}(a)}
        = a^{-1/\gamma(x)}g(x) \int K^{2}(u)\mathrm{d}u < \infty.
    \end{align*}
    We note that $\mathcal{G}$ is linearly ordered. This indeed holds due to monotonicity:
    \begin{align*}
        s < t \iff \mathcal{E}_{s}\p{\boldsymbol{y}, \boldsymbol{x}} \geq
        \mathcal{E}_{t}\p{\boldsymbol{y}, \boldsymbol{x}}.
    \end{align*}
    Consequently, Lemma~\ref{linearly_ordered_class_lemma} yields that the uniform entropy condition in Lemma~\ref{uniform_entropy_lemma} is fulfilled, and hence the random entropy condition in Theorem~\ref{bracketing_clt} holds, and thus the asymptotic equicontinuity of Theorem~\ref{equicontinuity_theorem} holds. This concludes the proof.
\end{proof}

\begin{proof}[Proof of Proposition~\ref{Functional_CLT_Tail_empirical_process}]
Consider the following decomposition
\begin{align*}
    \widetilde{\mathbb{T}}^{x}_{n}(s) 
    &= \sqrt{k_{n}h_{n}}
    \curly{g(x)- g_{n}(x)}
    \frac{\widetilde{T}^{x}_{n}(s)}{g(x)} + 
    \frac{1}{g(x)}\widetilde{\mathbb{V}}^{x}_{n}(s)+ 
    \sqrt{k_{n}h_{n}}\frac{1}{g(x)}
    \curly{ \E{\widetilde{V}^{x}_{n}(s)} - g(x)s^{-1/\gamma(x)}}.
\end{align*}
    The absolute value of the last term vanishes, since it is bounded by $\sqrt{k_{n}h_{n}}B_{n}^{x}(s_{0}) \to 0$,
    so Lemma~\ref{lemma:g_n_emp_process_negligibility_alpha_mixing} and Theorem~\ref{theorem_fct_clt_V} imply that
    \begin{align*}
    \widetilde{\mathbb{T}}^{x}_{n}(s)
    = \sqrt{k_{n}h_{n}}
    \curly{g(x)- g_{n}(x)}
    \frac{\widetilde{T}^{x}_{n}(s)}{g(x)} + 
    \frac{1}{g(x)}\widetilde{\mathbb{V}}^{x}_{n}(s) 
     + 
    o(1)
     \overset{w}{\Longrightarrow}\frac{1}{g(x)}\mathbb{V}^{x}(s). 
\end{align*}
\end{proof}
Recall the definitions $u_{n}^{x} = F^{\leftarrow, x}\p{1-k_{n}/n}$ and $q_{n,k_{n}}^{x} = F^{\leftarrow, x}_{n}\p{1-k_{n}/n}$. Then $n\overline{F}^{x}(u_{n}^{x}) = k_{n}$. Define the conditional tail empirical quantile function $\widetilde{Q}_{n}^{x}$ on $\p{0, k_{n}/n}$ by 
\begin{align*}
    \widetilde{Q}_{n}^{x}(t) = 
    \widetilde{T}_{n}^{\leftarrow,x}\p{\frac{\floor{k_{n}t}}{k_{n}}} =
    \frac{q_{\floor{k_{n}t},n}}{u_{n}^{x}},
\end{align*}
and let $Q^{x}(t) = t^{-\gamma(x)}, t > 0$.
\begin{theorem} \label{th_quantile_normality}
    Assume~\eqref{def:condRV}, conditions~\ref{Assu:gamma_lip}-\ref{Assu:g_n_consistent} and~\ref{rate_condition}-\ref{B_n_assumption}.
    Then
    \begin{align*}
        \sqrt{h_{n}k_{n}}\curly{\widetilde{T}_{n}^{x} - T^{x},   \widetilde{Q}_{n}^{x} - Q^{x}} 
        \overset{w}{\Longrightarrow} \p{\mathbb{T}^{x}, -\p{Q^{x}}'\cdot \mathbb{T}^{x}\circ Q^{x}},
    \end{align*}
    in $\mathbb{D}\p{[s_0, \infty)} \times \mathbb{D}((0,s_{0}^{-\gamma(x)}])$ endowed with the product $J_{1}$ topology.\\ Consequently
    \begin{align*}
         \sqrt{h_{n}k_{n}}\curly{\frac{q_{n,k_{n}}^{x}}{u_{n}^{x}} - 1} 
         \overset{d}{\to}
         \gamma(x)\mathbb{T}^{x}(1), 
    \end{align*}
    jointly with the convergence of $\sqrt{h_{n}k_{n}}\curly{\widetilde{T}_{n}^{x} - T^{x}}$ to $\mathbb{T}^{x}$.
\end{theorem}
\begin{proof}
    Note that $[s_0, \infty) \ni s \mapsto T^{x}(s)$ is surjective on $(0,s_{0}^{-\gamma(x)}]$, continuous and decreasing with negative derivative. As the processes $\widetilde{T}_{n}^{x}$ have non-increasing sample paths for each $n \geq 1$, Vervaat's Lemma~\ref{Vervaat} yields that 
    \begin{align*}
        \sqrt{h_{n}k_{n}}\curly{\widetilde{T}_{n}^{x} - T^{x},  \widetilde{T}_{n}^{\leftarrow,x}- T^{\leftarrow, x}} 
        \overset{w}{\Longrightarrow} 
        \p{\mathbb{T}^{x}, -\p{T^{\leftarrow, x}}'\cdot \mathbb{T}^{x}\circ T^{\leftarrow, x}},
    \end{align*}
    in $\mathbb{D}\p{[s_0, \infty)} \times \mathbb{D}((0,s_{0}^{-\gamma(x)}])$ endowed with the product $J_{1}$ topology. This concludes the proof.
\end{proof}
\begin{proof}[Proof of Proposition~\ref{prop:q_{n,k_{n}}^{x}_normality}]
    The limit follows directly from Theorem~\ref{th_quantile_normality}. The asymptotic variance follows as
    \begin{align*}
        \text{var}\left(\gamma(x)\mathbb{T}^{x}(1)\right)
        =
        \gamma(x)^{2} \text{cov}(\mathbb{T}^{x}(1),\mathbb{T}^{x}(1)) 
        =
        \frac{\gamma(x)^{2}}{g(x)}\int K^{2}(u) \dd u.
    \end{align*}
\end{proof}
\begin{proof}[Proof of Theorem~\ref{Thm:Functional_CLT_Tail_empirical_process_Random_level}]
    Decompose $ \widehat{\mathbb{T}}^{x}_{n}$ into the following components
    \begin{align*}
         \widehat{\mathbb{T}}^{x}_{n} =  \widetilde{\mathbb{T}}^{x}_{n}\p{s\frac{q_{n,k_{n}}^{x}}{u_{n}^{x}}} + 
         \sqrt{k_{n}h_{n}}\curly{T^{x}\p{s\frac{q_{n,k_{n}}^{x}}{u_{n}^{x}}} - T^{x}(s)}.
    \end{align*}
    The continuous mapping theorem and Theorem~\ref{th_quantile_normality} yields 
     that the first term above converges to $\mathbb{T}^{x}$.  As for the second term, Theorem~\ref{th_quantile_normality} and the delta-method yields that
     \begin{align*}
         \sqrt{k_{n}h_{n}}\curly{T^{x}\p{s\frac{q_{n,k_{n}}^{x}}{u_{n}^{x}}} - T^{x}(s)} &= 
         s^{-{1/\gamma(x)}}\sqrt{k_{n}h_{n}}\curly{T^{x}\p{\frac{q_{n,k_{n}}^{x}}{u_{n}^{x}}} - T^{x}(1)}
         \\
         &\overset{d}{\to}
          s^{-{1/\gamma(x)}} \p{T^{x}}'(1) \cdot \gamma(x)\mathbb{T}^{x}(1) = -T^{x}(s)\mathbb{T}^{x}(1), 
     \end{align*}
     jointly with $\widetilde{\mathbb{T}}^{x}_{n}\p{sq_{n,k_{n}}^{x}/u_{n}^{x}}$. The argument can be transferred unchanged to pathwise weak convergence.
\end{proof}

\subsection{Proof of Theorem~\ref{Hill_normality_thm} (Hill CLT)}

\begin{proof}[Proof of Theorem~\ref{Hill_normality_thm}]
    We consider the decomposition
    \begin{align*}
        \sqrt{k_{n}h_{n}}\p{\widehat{\gamma}_{n}(x) - \gamma(x)}
        &= \sqrt{k_{n}h_{n}}\int_{q_{n,k_{n}}^{x}/u_{n}^{x}}^{1}
        \frac{\widetilde{T}_{n}^{x}\p{s}}{s}\mathrm{d}s
        +
        \int_{1}^{\infty}
        \frac{\widetilde{\mathbb{T}}_{n}^{x}\p{s}}{s}\mathrm{d}s. 
    \end{align*}
    We examine the first term.
    Since Theorem~\ref{th_quantile_normality} yields $q_{n,k_{n}}^{x}/u_{n}^{x} \overset{\P}{\to} 1$, we may assume without loss of generality that for given $\varepsilon \in (0,1)$, there exists $N \in \N$ such that for $n\geq N$ it holds that $q_{n,k_{n}}^{x}/u_{n}^{x} \in (1-\varepsilon, 1+ \varepsilon)$. \\
    Thus for  $n\geq N$ we have that
    \begin{align*}
        \sqrt{k_{n}h_{n}} \left\vert\int_{q_{n,k_{n}}^{x}/u_{n}^{x}}^{1}
        \frac{\widetilde{T}_{n}^{x}\p{s} -T^{x}\p{s}}{s}\mathrm{d}s \right\vert
        \leq 
        \frac{\sqrt{k_{n}h_{n}}}{1-\varepsilon}
        \left\vert\frac{q_{n,k_{n}}^{x}}{u_{n}^{x}} -1 \right\vert
        \sup_{t \in (1-\varepsilon, 1+ \varepsilon)}
        \left\vert \widetilde{T}_{n}^{x}\p{t} -T^{x}\p{t} \right\vert.
    \end{align*}
    Moreover, Theorem~\ref{th_quantile_normality} and the continuous mapping theorem yield that
    \begin{align*}
        \sqrt{k_{n}h_{n}}
        \left\vert\frac{q_{n,k_{n}}^{x}}{u_{n}^{x}} -1 \right\vert
        \overset{d}{\to}
        \gamma(x) 
        \left\vert \mathbb{T}^{x}(1)\right\vert,
    \end{align*}
    while Theorem~\ref{Empirical_tail_dist_consistency_theorem_m_depence} and Dini's theorem yield  $  \sup_{t \in (1-\varepsilon, 1+ \varepsilon)}
        \left\vert \widetilde{T}_{n}^{x}\p{t} -T^{x}\p{t} \right\vert \overset{\P}{\to} 0$ and thus
    \begin{align*}
        \sqrt{k_{n}h_{n}}\int_{q_{n,k_{n}}^{x}/u_{n}^{x}}^{1}
        \frac{\widetilde{T}_{n}^{x}\p{s}}{s}\mathrm{d}s 
        =
        \sqrt{k_{n}h_{n}}\int_{q_{n,k_{n}}^{x}/u_{n}^{x}}^{1}
        \frac{T^{x}\p{s}}{s}\mathrm{d}s + o_{\P}(1).
    \end{align*}
    Applying the delta-method with $x \mapsto \int_{x}^{1}  \frac{T^{x}\p{s}}{s}\mathrm{d}s$, on the sequence $(q_{n,k_{n}}^{x}/u_{n}^{x})$ yields that
    \begin{align*}
        \sqrt{k_{n}h_{n}}\int_{q_{n,k_{n}}^{x}/u_{n}^{x}}^{1}
        \frac{T^{x}\p{s}}{s}\mathrm{d}s 
        \overset{d}{\to}
        -\gamma(x) \mathbb{T}^{x}(1).
    \end{align*}
    This provides the asymptotics of the first term. We proceed to the second term
    \begin{align*}
        \int_{1}^{\infty}
        \frac{\widetilde{\mathbb{T}}_{n}^{x}\p{s}}{s}\mathrm{d}s  
        &= 
        \sqrt{k_{n}h_{n}}
         \frac{1}{g(x)}
        \curly{g(x)- g_{n}(x)}
         \int_{1}^{\infty}
        \frac{\widetilde{T}^{x}_{n}(s)}{s}\mathrm{d}s  + \frac{1}{g(x)}\int_{1}^{\infty}
        \frac{\widetilde{\mathbb{V}}_{n}^{x}\p{s}}{s}\mathrm{d}s \\
        & \qquad
        +
        \sqrt{k_{n}h_{n}}\frac{1}{g(x)}\int_{1}^{\infty}
        \frac{\E{\widetilde{V}^{x}_{n}(s)} - g(x)s^{-1/\gamma(x)}}{s}\mathrm{d}s.
    \end{align*}
    The last term is negligible due to condition~\ref{integral_B_n_assumption}. Theorem~\ref{th_quantile_normality} yields that equation~\eqref{Empirical_tail_dist_consistency_equation} holds, and consequently $\int_{1}^{\infty}
        \widetilde{T}^{x}_{n}(s)/s \ \mathrm{d}s \overset{\P}{\to} \gamma(x)$ as in the proof of Theorem~\ref{Hill_Consistency_thm}. Combined with condition~\ref{Assu:gn_emp_proc_negl} we have that the first term also vanishes.
    Finally we consider $g(x)^{-1}\int_{1}^{\infty} \widetilde{\mathbb{V}}_{n}^{x}\p{s}/s \ \mathrm{d}s$. 
    It holds that
    \begin{align*}
        &\int_{1}^{A}
        \frac{\widetilde{\mathbb{V}}_{n}^{x}\p{s}}{s}\mathrm{d}s
        \underset{n \to \infty}{\overset{d}{\to}} 
        \int_{1}^{A}
        \frac{\widetilde{\mathbb{V}}^{x}\p{s}}{s}\mathrm{d}s,
        \qquad
        \int_{A}^{\infty}
        \frac{\widetilde{\mathbb{V}}^{x}\p{s}}{s}\mathrm{d}s
         \underset{A \to \infty}{\overset{d}{\to}}  0.
    \end{align*}
   Introduce
    \begin{align*}
        \mathbb{J}_{n,A} = 
        \int_{A}^{\infty}
        \frac{\widetilde{\mathbb{V}}_{n}^{x}\p{s}}{s}\mathrm{d}s 
        , \quad 
        \mathbb{J}^{\dag}_{n,A} = 
        \int_{A}^{\infty}
        \frac{\widetilde{\mathbb{V}}_{n}^{\dag, x}\p{s}}{s}\mathrm{d}s. 
    \end{align*}
    Normality is established if we can show that, for every $\varepsilon > 0$,
    \begin{align*}
        \lim_{A \to \infty}\limsup_{n\to \infty}
        \P\p{\left\vert  \mathbb{J}_{n,A}\right\vert > \varepsilon} = 0
    \end{align*}
    By repeating the arguments in the proof of Theorem~\ref{theorem_fct_clt_V}, it holds that
    \begin{align*}
        \P\p{\left\vert  \mathbb{J}_{n,A}\right\vert > \varepsilon} 
        \leq \P\p{\left\vert  \mathbb{J}^{\dag}_{n,A}\right\vert > \varepsilon} 
        +
        m_{n}\beta_{r_{n}},
    \end{align*}
    and thus it suffices to show that
     \begin{align*}
        \lim_{A \to \infty}\limsup_{n\to \infty}
        \P\p{\left\vert  \mathbb{J}^{\dag}_{n,A}\right\vert > \varepsilon} = 0.
    \end{align*}
    As $ \mathbb{J}^{\dag}_{n,A}$ is centred, we show negligibility of the variance.
    By independence of blocks we have that
    \begin{align*}
        &\text{var}\p{ \mathbb{J}^{\dag}_{n,A}}
        = 
       \frac{m_n}{nh_{n}\overline{F}^{x}(u_{n}^{x})}
         \text{var}\p{ \sum_{j=1}^{r_n}K\left(\frac{x-X_{j}}{h_n}\right) 
         \log_{+}\left(\frac{Y_{j}}{u_{n}^{x}A} \right)} \\
         &\leq
         \frac{2m_{n}r_{n}}{nh_{n}\overline{F}^{x}(u_{n}^{x})}
         \sum_{j=0}^{r_n}
         \E{K\left(\frac{x-X_{0}}{h_n}\right) K\left(\frac{x-X_{j}}{h_n}\right) 
         \log_{+}\left(\frac{Y_{0}}{u_{n}^{x}A} \right)
         \log_{+}\left(\frac{Y_{j}}{u_{n}^{x}A} \right)}
    \end{align*}
    Now, for fixed $m \in \N$, stationarity yields
    \begin{align*}
        &
         \frac{m_{n}r_{n}}{nh_{n}\overline{F}^{x}(u_{n}^{x})}
        \sum_{j = 0}^{m-1} \E{K\left(\frac{x-X_{0}}{h_n}\right) K\left(\frac{x-X_{j}}{h_n}\right) 
         \log_{+}\left(\frac{Y_{0}}{u_{n}^{x}A} \right)
         \log_{+}\left(\frac{Y_{j}}{u_{n}^{x}A} \right)} 
         \\
         & \leq 
         \frac{m_{n}r_{n}}{nh_{n}\overline{F}^{x}(u_{n}^{x})}
         m
         \E{K^{2}\left(\frac{x-X_{0}}{h_n}\right) \log_{+}\left(\frac{Y_{0}}{u_{n}^{x}A} \right)^{2}}.
    \end{align*}
    Then, recalling $m_{n}r_{n} = \mathcal{O}(n)$, and that the kernel $K$ has support $[-1,1]$, we have that $1_{\{(z-x)/h_{n} \leq 1\}} =  1_{\{z \leq x + h_{n}\}}$ and hence we have the further bound
    \begin{align*}
         &\frac{m}{h_{n}\overline{F}^{x}(u_{n}^{x})}\int  
         K^{2}\left(\frac{x-z}{h_{n}}\right) 
         \E{\log_{+}\left(\frac{Y_{0}}{u_{n}^{x}A} \right)^{2} \Big\vert X_0 = z} g(z) \dd z \\
         & \leq 
         m
         \sup_{u \in [-h_{n} + x, x + h_{n}]}\left(\frac{1}{\overline{F}^{u}(u_{n}^{x})}\E{\log_{+}\left(\frac{Y_{0}}{u_{n}^{x}A} \right)^{2} \Big\vert X_0 = u}\right) \times \frac{1}{h_n} \int  
         K^{2}\left(\frac{x-z}{h_{n}}\right) g(z) \frac{\overline{F}^{z}(u_{n}^{x})}{\overline{F}^{x}(u_{n}^{x})} \dd z,
    \end{align*}
    where it holds that
    \begin{align*}
        &\frac{1}{h_n} \int  
         K^{2}\left(\frac{x-z}{h_{n}}\right) g(z) \frac{\overline{F}^{z}(u_{n}^{x})}{\overline{F}^{x}(u_{n}^{x})} \dd z
         \leq 
          \frac{1}{h_n} \int  
         K^{2}\left(\frac{x-z}{h_{n}}\right) g(z) (u_{n}^{x})^{-(c_{\gamma} + c_{L})\lvert x - z \rvert} \dd z \\
         & \leq
         (u_{n}^{x})^{(c_{\gamma} + c_{L})h_{n}}
         \int  
         K^{2}\left(w\right) g(x - wh_{n}) \dd w = \mathcal{O}(1),
    \end{align*}
    and 
    \begin{align*}
        &\lim_{n \to \infty} \sup_{u \in [-h_{n} + x, x + h_{n}]}\left(\frac{1}{\overline{F}^{u}(u_{n}^{x})}\E{\log_{+}\left(\frac{Y_{0}}{u_{n}^{x}A} \right)^{2} \Big\vert X_0 = u}\right)
        \\
        &\leq
        \lim_{n \to \infty}\left(\frac{1}{\overline{F}^{x}(u_{n}^{x})}\E{\log_{+}\left(\frac{Y_{0}}{u_{n}^{x}A} \right)^{2} \Big\vert X_0 = x}\right) \\
         & \quad +
        \lim_{n \to \infty} \sup_{u \in [-h_{n} + x, x + h_{n}]}
        \Bigg\vert
        \frac{1}{\overline{F}^{u}(u_{n}^{x})}\E{\log_{+}\left(\frac{Y_{0}}{u_{n}^{x}A} \right)^{2} \Big\vert X_0 = u}  -
        \frac{1}{\overline{F}^{x}(u_{n}^{x})}\E{\log_{+}\left(\frac{Y_{0}}{u_{n}^{x}A} \right)^{2} \Big\vert X_0 = x}
        \Bigg\vert
        \\
        &=
        \int_{A}^{\infty} \log^{2}(t) 1/\gamma(x)t^{-1/\gamma(x)-1} \dd t,
     \end{align*}
     where we have used Potter's bound to invoke the dominated convergence theorem for the first limit, and the term in the difference can be bounded using Potter's bound and applying twice the Lipschitz conditions on $L^x$ and $1/\gamma(x)$ as was done in Theorem~\ref{Hill_Consistency_thm}.
    The remaining upper sum is is asymptotically negligible by condition~\ref{log_anti_clustering_condition}, i.e. 
     \begin{align*}
       \lim_{m\to \infty}
       \limsup_{n\to \infty} \sum_{j=m}^{r_n} 
         \frac{\E{K\left(\frac{x-X_{0}}{h_n}\right) K\left(\frac{x-X_{j}}{h_n}\right) 
         \log_{+}\left(\frac{Y_{0}}{u_{n}^{x}A} \right)
         \log_{+}\left(\frac{Y_{j}}{u_{n}^{x}A} \right)}}{h_{n}\overline{F}^{x}(u_{n}^{x})}
           = 0.
    \end{align*}
    Thus, taking limit as $A\to\infty$ establishes the required asymptotic normality.
    By construction, $\mathbb{T}^{x}(\cdot) = \mathbb{V}^{x}(\cdot)/g(x)$, and thus we have the representation
    \begin{align*}
         \sqrt{k_{n}h_{n}}\p{\widehat{\gamma}_{n}\p{x} - \gamma(x)}
         \overset{d}{\to} \frac{1}{g(x)}\int_{1}^{\infty}
        \frac{\mathbb{V}^{x}\p{s}}{s}\mathrm{d}s -\frac{\gamma(x)}{g(x)} \mathbb{V}^{x}(1).
    \end{align*}
    It remains to establish the variance of the limit distribution. We provide the straightforward calculations.
    Change of variable yields that the limit can be rewritten as
    \begin{align*}
        \gamma(x) \int_{0}^{1}
        \frac{\mathbb{T}^{x}\p{t^{-\gamma(x)}} -t\mathbb{T}^{x}\p{1}}{t}\mathrm{d}t.
    \end{align*}
    Let $\mathbb{W}^{x}\p{t} = \mathbb{T}^{x}\p{t^{-\gamma(x)}}$ and note that $\mathbb{W}^{x}$ has the following covariance structure
    \begin{align*}
         \text{cov}\left(\mathbb{W}^{x}\p{s}, \mathbb{W}^{x}\p{t}\right) 
        &= \p{s \wedge t} \frac{1}{g(x)} \int K^{2}(u) \dd u =: c^{*}(s,t).
    \end{align*}
    The limit process is a linear functional of centred Gaussian processes, and hence centred itself. The variance of the limit variable is thus given by
    \begin{align*}
        &\gamma(x)^{2} \E{\int_{0}^{1}\frac{\mathbb{T}^{x}\p{s^{-\gamma(x)}} -s\mathbb{T}^{x}\p{1}}{s}\mathrm{d}s 
        \int_{0}^{1}\frac{\mathbb{T}^{x}\p{t^{-\gamma(x)}} -t\mathbb{T}^{x}\p{1}}{t}\mathrm{d}t}\\
        & =\gamma(x)^{2} \int_{0}^{1}\int_{0}^{1}\frac{c^{*}(s,t) -sc^{*}(t,1) -tc^{*}(s,1) + stc^{*}(1,1)}{st}\mathrm{d}s\mathrm{d}t
        \\
        & =\gamma(x)^{2} \left\{c^{*}(1,1) + \int_{0}^{1}\int_{0}^{1}\frac{c^{*}(s,t)}{st}\mathrm{d}s\mathrm{d}t 
        -2\int_{0}^{1}\frac{c^{*}(s,1)}{s}\mathrm{d}s\right\}. 
    \end{align*}
    We have that $c^{*}(s,t) = c^{*}(t,s) = tc^{*}(s/t,1)$, so that the latter two terms cancel 
    and we get variance of the limiting distribution given by 
    \begin{align*}
        \gamma(x)^{2}c^{*}(1,1) = \gamma(x)^{2}\frac{1}{g(x)} \int K^{2}(u) \dd u.
    \end{align*}
\end{proof}

\section{Proofs of results from Section~\ref{section:verification}} 
\begin{proof}[Proof of Theorem~\ref{thm:sequence_cons_power}]
Due to the fact that $\delta \in ((\eta + 1)^{-1}, 1)$, we can let 
    \begin{align*}
        e \in \left(0,\frac{\delta}{1-\delta}\wedge \frac{\delta(\eta + 1) -1}{2(1-\delta)}\right),
    \end{align*}
    since the interval is non-empty.\\
    Notice that by regular variation it holds that 
    \begin{align*}
         \log(u_{n}^{x})
         =\gamma(x)\log\left(\ell^{x}(u_{n}^{x})\right)
         - 
         \gamma(x)\log\left(k_{n}/n\right). 
    \end{align*}
    Consequently, for any strictly positive $e$ it holds that
    \begin{align*}
        h_{n}\log(u_{n}^{x})
        &=
        h_{n}\left(\gamma(x)\log\left(\ell^{x}(u_{n}^{x})\right)
         - 
         \gamma(x)\log\left(k_{n}/n\right)\right) \\
         &=
        (\overline{F}^{x}(u_{n}^{x}))^{e}\left(\gamma(x)\log\left(\ell^{x}(u_{n}^{x})\right)
         - 
         \gamma(x)\log\left(n^{-(1-\delta)}\right)\right) \\
         &=
         \gamma(x) (u_{n}^{x})^{-e/\gamma(x)} \ell^{x}(u_{n}^{x})^{e}\log\left(\ell^{x}(u_{n}^{x})\right)
         +
         \mathcal{O}\left(n^{-(1-\delta)e} \log(n)\right) \\
         &\sim
         \gamma(x) (u_{n}^{x})^{-e/\gamma(x)} \ell^{x}(u_{n}^{x})^{e}\log\left(\ell^{x}(u_{n}^{x})\right) \\
         &\leq
         \gamma(x) (u_{n}^{x})^{-e/\gamma(x)} \ell^{x}(u_{n}^{x})^{1+e}
         =
         \gamma(x) \left((u_{n}^{x})^{-(e/1+e)/\gamma(x)} \ell^{x}(u_{n}^{x})\right)^{1+e}
         \to 0.
    \end{align*}
    The upper bound on $e$ ensures
    \begin{align*}
        &n h_{n} \overline{F}^{x}(u_{n}^{x}) 
        \asymp n \left(\overline{F}^{x}(u_{n}^{x})\right)^{1+e}
        \asymp n\left(\frac{k_{n}}{n}\right)^{1+e}= n^{-(1-\delta)(1+e) + 1}
        = n^{-(1-\delta)e +\delta}\to \infty.
    \end{align*}
    As for the third condition we have
    \begin{align*}
        \frac{\overline{F}^{x}(u_{n}^{x}) + h_{n}^{c_{x}\eta -2}}{n h_{n}^{c_{x}} (\overline{F}^{x}(u_{n}^{x}))^{2}}
        =
        \frac{1}{n h_{n}^{c_{x}} \overline{F}^{x}(u_{n}^{x})} + \frac{h_{n}^{c_{x}(\eta - 1) -2}}{n (\overline{F}^{x}(u_{n}^{x}))^{2}}.
    \end{align*}
    Let $c_{x}$ be a positive constant such that
    \begin{align*}
        c_{x} \in \left(\frac{2e(1-\delta) +1 -2\delta}{e(1-\delta)(\eta-1)} \vee 0, \frac{\delta}{e(1-\delta)}\right),
    \end{align*}
    and note that the restrictions on $\delta$ and $e$ ensures that this interval is non-empty.
  The upper bound on $c_{x}$ ensures that the first of the above two terms vanishes asymptotically since
    \begin{align*}
         \left( n h_{n}^{c_{x}} \overline{F}^{x}(u_{n}^{x})\right)^{-1}
         &\asymp
          \left( n (\overline{F}^{x}(u_{n}^{x}))^{c_{x}e + 1}\right)^{-1}=
          \left( n \left(\frac{k_n}{n}\right)^{c_{x}e + 1}\right)^{-1} 
          \asymp
          \left( n^{1- (1-\delta)(c_{x}e + 1)}\right)^{-1} \to 0,
    \end{align*}
    while the lower bound ensures
    \begin{align*}
        \frac{h_{n}^{c_{x}(\eta - 1) -2}}{n (\overline{F}^{x}(u_{n}^{x}))^{2}}    
        &=
        \frac{1}{nh_{n}\overline{F}^{x}(u_{n}^{x})}\frac{h_{n}^{c_{x}(\eta- 1) -1}}{\overline{F}^{x}(u_{n}^{x})} \\
        &\asymp
        n^{(1-\delta)e -\delta}
        \overline{F}^{x}(u_{n}^{x})^{e(c_{x}(\eta-1)-1)-1} 
        \asymp 
        n^{(1-\delta)e -\delta- (1-\delta)(e(c_{x}(\eta-1)-1)-1)}\to 0
        ,
    \end{align*}
    as $n\to \infty$.
\end{proof}
\begin{proof}[Proof of Theorem~\ref{thm:sequence_cons_m_dependence}]
    Note that $n h_{n} \overline{F}^{x}(u_{n}^{x}) 
        = k_{n}h_{n} 
        \asymp \log(n)^{2-e}\to \infty$.
Now, the relation
    \begin{align*}
        \log\left(n/k_{n}\right)
        =
        -\log\left( \overline{F}^{x}(u_{n}^{x})\right)
        =
        (1/\gamma(x))\log(u_{n}^{x}) - \log(\ell^{x}(u_{n}^{x})),
    \end{align*}
    implies that 
    \begin{align*}
        \frac{ \gamma(x)\log(n/k_{n})}{\log(u_{n}^{x})}
        =
        1- \frac{\gamma(x)\log(\ell^{x}(u_{n}^{x}))}{\log(u_{n}^{x})} \to 1,
    \end{align*}
    using~\citet{Bingham1987} Proposition 1.3.6 (i). Hence $ \log\left(n/k_n\right) \sim (1/\gamma(x))\log(u_{n}^{x}) $ and consequently
    \begin{align*}
       h_{n}\log(u_{n}^{x})
       \sim
        \gamma(x)h_{n}\log(n/k_{n})
        \leq
        \gamma(x)h_{n}\log(n)
        \asymp
        \gamma(x)\log(n)^{1-e} \to 0. 
    \end{align*}
\end{proof}
\begin{proof}[Proof of Lemma~\ref{lemma_second_order_conditions}]
    By definition of $u_{n}^{x}$ and properties of generalised inverses, it holds that $n\overline{F}^{x}(u_{n}^{x}) = k_{n}$. The second order assumption of equation~\eqref{cond_second_order} implies that $u_{n}^{x} = \mathcal{O}((n/k_{n})^{\gamma(x)})$.\\
    Then recall that 
    \begin{align*}
         B_{n}^{x}(s_{0}) = \sup_{t \geq s_{0}}\left\vert   \E{\widetilde{V}^{x}_{n}(t)}  - t^{-1/\gamma(x)}g(x)  \right\vert.
    \end{align*}
    The above difference can be split as follows
    \begin{align*}
        &\E{\widetilde{V}^{x}_{n}(t)}  - t^{-1/\gamma(x)}g(x)
        =
        \int  K\p{u}
        \left( \frac{\overline{F}^{x-uh_{n}}(t u_{n}^{x})}
         {\overline{F}^{x}(u_{n}^{x})}g(x - h_{n}u)  
         -
          t^{-1/\gamma(x)}g(x)\right)
         \mathrm{d}u \\
         &=
         \int  K\p{u} t^{-1/\gamma(x)}
        \left(g(x - h_{n}u)  
         -
         g(x)\right)
         \mathrm{d}u   +
         \int  K\p{u}
        \left( \frac{\overline{F}^{x}(t u_{n}^{x})}
         {\overline{F}^{x}(u_{n}^{x})}
         -
          t^{-1/\gamma(x)}\right)g(x - h_{n}u)  
         \mathrm{d}u \\
         & \quad +
         \int  K\p{u}
        \left( \frac{\overline{F}^{x-uh_{n}}(t u_{n}^{x})}
         {\overline{F}^{x}(u_{n}^{x})}
         -
         \frac{\overline{F}^{x}(t u_{n}^{x})}
         {\overline{F}^{x}(u_{n}^{x})} \right)
         g(x - h_{n}u)  
         \mathrm{d}u.
    \end{align*}
    The above expression is analysed term by term. The density $g$ is Lipschitz by assumption and hence the first of the above three terms is of order
    \begin{align*}
        \int  K\p{u} t^{-1/\gamma(x)}
        \left(g(x - h_{n}u)  
         -
         g(x)\right)
         \mathrm{d}u
         \leq
         t^{-1/\gamma(x)} h_{n} c_{g}  \int  K\p{u} \left\vert u \right\vert  \mathrm{d}u = \mathcal{O}( t^{-1/\gamma(x)} h_{n}).
    \end{align*}
    When examining the third term, we rely on calculations from the proof of Lemma~\ref{Limit_of_tail_process} and utilize the inequality
    $\left\vert e^{z} -1 \right\vert \leq \left\vert z \right\vert + (1+ e^{\left\vert z \right\vert }) z^{2}/2$. This yields
    \begin{align*}
        &\left( \frac{\overline{F}^{x-uh_{n}}(t u_{n}^{x})}
         {\overline{F}^{x}(u_{n}^{x})}
         -
         \frac{\overline{F}^{x}(t u_{n}^{x})}
         {\overline{F}^{x}(u_{n}^{x})} \right)
         \leq
         \frac{\overline{F}^{x}(t u_{n}^{x})}
         {\overline{F}^{x}(u_{n}^{x})}
        \left\vert \frac{\overline{F}^{x-uh_{n}}(t u_{n}^{x})}
         {\overline{F}^{x}(tu_{n}^{x})}
         -
         1 \right\vert
         \\
         &\leq 
        \frac{\overline{F}^{x}(t u_{n}^{x})}
         {\overline{F}^{x}(u_{n}^{x})}
         \left(
         \left\vert \log\left(  \frac{\overline{F}^{x-uh_{n}}(t u_{n}^{x})}
         {\overline{F}^{x}(tu_{n}^{x})} \right) \right\vert
         +
         \frac{1}{2} \log\left(  \frac{\overline{F}^{x-uh_{n}}(t u_{n}^{x})}
         {\overline{F}^{x}(tu_{n}^{x})} \right)^{2}\left(1 + \frac{\overline{F}^{x-uh_{n}}(t u_{n}^{x})}
         {\overline{F}^{x}(tu_{n}^{x})}\right)
         \right)
         \\
         & \leq \text{cst} \
         \frac{\overline{F}^{x}(t u_{n}^{x})}
         {\overline{F}^{x}(u_{n}^{x})} \log(t u_{n}^{x}) \left\vert u \right\vert h_{n}
         \left( 1 + \frac{1}{2} \log(t u_{n}^{x}) \left\vert u \right\vert h_{n}\left( 1 + (t u_{n}^{x})^{(c_{\gamma} + c_{L})\left\vert u \right\vert h_{n}}\right)
         \right)
         \\
         & \leq \text{cst} \
         t^{-1/\gamma(x) + \varepsilon}
         \log(t u_{n}^{x}) \left\vert u \right\vert h_{n}
         \left( 1 + \frac{1}{2} \log(t u_{n}^{x}) \left\vert u \right\vert h_{n}\left( 1 + (t u_{n}^{x})^{(c_{\gamma} + c_{L})\left\vert u \right\vert h_{n}}\right)
         \right),
    \end{align*}
    for $\varepsilon \in (0, 1/\gamma(x))$. The first term is the leading term and of order 
    \begin{align*}
        \mathcal{O}\left(t^{-1/\gamma(x) + \varepsilon}
         \log(t u_{n}^{x})h_{n}\right)
         =
         \mathcal{O}\left(t^{-1/\gamma(x) + \varepsilon}
         \log\left(t \frac{n}{k_{n}}\right)h_{n}\right).
    \end{align*}
         To determine the order at which the second and final term converge to zero the second order condition is needed. Indeed
     \begin{align*}
          &\left \vert \frac{\overline{F}^{x}(t u_{n}^{x})}
         {\overline{F}^{x}(u_{n}^{x})}
         -
          t^{-1/\gamma(x)}\right \vert=
          \left \vert  \frac{n \overline{F}^{x}(t u_{n}^{x})}
         {k_n}
         -
          t^{-1/\gamma(x)}\right \vert\\
          & \leq \frac{n}{k_n} \left( \left \vert \overline{F}^{x}(t u_{n}^{x})
         -
         c_{x}(u_{n}^{x}t)^{-1/\gamma(x)}
          \right \vert
          + t^{-1/\gamma(x)}
          \left \vert 
         c_{x}(u_{n}^{x})^{-1/\gamma(x)}
         -
         \overline{F}^{x}(u_{n}^{x})
          \right \vert
          \right)
          \\
          & \leq 
          \text{cst}
          \frac{n}{(u_{n}^{x})^{\zeta_{x}}k_n}\left(t^{-\zeta_{x}} + t^{-1/\gamma(x)}\right)
          \leq 
          \text{cst}
          \frac{n}{(u_{n}^{x})^{\zeta_{x}}k_n}t^{-1/\gamma(x)}
          =
          \mathcal{O}\left( 
         k_{n}^{\zeta_{x}\gamma(x) - 1}n^{1-\zeta_{x}\gamma(x)}
         t^{-1/\gamma(x)}
          \right).
     \end{align*}
     The above three O-terms are uniformly bounded in $t$ for $t\geq s_{0} \in (0,1)$. Thus, we get
     \begin{align*}
         \sqrt{k_{n}h_{n}}B_{n}^{x}(s_{0})&\leq
          \mathcal{O}( \sqrt{k_{n}h_{n}} h_{n})
          +
          \mathcal{O}\left( \sqrt{k_{n}h_{n}}
         \log\left(n/k_{n}\right)h_{n}\right)
          +
           \mathcal{O}\left( 
            \sqrt{k_{n}h_{n}}
         k_{n}^{\zeta_{x}\gamma(x) - 1}n^{1-\zeta_{x}\gamma(x)}
          \right) \\
          &= 
          \mathcal{O}\left( \sqrt{k_{n}h_{n}^{3}}
         \log\left(n/k_{n}\right)\right)
          +
           \mathcal{O}\left( 
            \sqrt{h_{n}}
         k_{n}^{\zeta_{x}\gamma(x) - \frac{1}{2}}n^{1-\zeta_{x}\gamma(x)}
          \right) \to 0.
     \end{align*}
     This concludes condition~\ref{B_n_assumption}. As for condition~\ref{integral_B_n_assumption} we have
    \begin{align*}
        &\int_{1}^{\infty} \frac{ \E{\widetilde{V}^{x}_{n}(t)}  - t^{-1/\gamma(x)}g(x)}{t} \mathrm{d}t \\
        &\leq 
        \mathcal{O}(h_n)\int_{1}^{\infty} t^{-1/\gamma(x)-1} \mathrm{d}t 
        +
        \mathcal{O}\left( 
         k_{n}^{\zeta_{x}\gamma(x) - 1}n^{1-\zeta_{x}\gamma(x)}
          \right)
        \int_{1}^{\infty} t^{-1/\gamma(x)-1} \mathrm{d}t \\
        & \quad +
        \mathcal{O}(h_n)\int_{1}^{\infty} 
        t^{-1/\gamma(x) + \varepsilon -1 }
         \log(t u_{n}^{x}) 
         \left( 1 + \frac{1}{2} \log(t u_{n}^{x}) h_{n}\left( 1 + (t u_{n}^{x})^{(c_{\gamma} + c_{L})\left\vert u \right\vert h_{n}}\right)
         \right)
        \mathrm{d}t,
    \end{align*}
    for $\varepsilon \in (0, 1/\gamma(x))$.
    The first two integrals are obviously finite, and thus we examine the last integral, which itself consists of three sub-integrals. The first of these subintegrals is finite, and of order $\mathcal{O}(h_n\log(u_{n}^{x}))$, likewise for the second which is of order $\mathcal{O}(h_{n}^{2}\log(u_{n}^{x})^{2})$. Finally, let $\varepsilon' \in (0, 1/\gamma(x) - \varepsilon)$ be given. It is possible to find $n'$ such that $\left(c_{\gamma}+ c_{L}\right) h_n \leq \varepsilon'$ for all $n\geq n'$. For such $n$'s the third integral is also finite and of order $\mathcal{O}(h_n\log(u_{n}^{x}))$, since $(u_{n}^{x})^{h_{n}} = \mathcal{O}(1)$. \\
     Consequently
      \begin{align*}
         \sqrt{k_{n}h_{n}}\int_{1}^{\infty} \frac{ \E{\widetilde{V}^{x}_{n}(t)}  - t^{-1/\gamma(x)}g(x)}{t} \mathrm{d}t \leq
          \mathcal{O}\left( \sqrt{k_{n}h_{n}^{3}}
         \log\left(n/k_{n}\right)\right)
          +
           \mathcal{O}\left( 
            \sqrt{h_{n}}
         k_{n}^{\zeta_{x}\gamma(x) - \frac{1}{2}}n^{1-\zeta_{x}\gamma(x)}
          \right),
     \end{align*}
     which vanishes as $n \to \infty$.
\end{proof}
\begin{proof}[Proof of Lemma~\ref{lemma_log_second_order_conditions}]
    As above, it holds that $\overline{F}^{x}(t u_{n}^{x}) /\overline{F}^{x}(u_{n}^{x}) = \overline{F}^{x}(tu_{n}^{x})n/k_{n}$. Consequently
    \begin{align*}
         &\left \vert \frac{\overline{F}^{x}(t u_{n}^{x})}
         {\overline{F}^{x}(u_{n}^{x})}
         -
          t^{-1/\gamma(x)}\right \vert \\
          &= \frac{n}{k_{n}} 
          \left \vert 
          \overline{F}^{x}(t u_{n}^{x}) 
          -c_{x}\log(u_{n}^{x}t)(u_{n}^{x}t)^{-1/\gamma(x)} 
          +c_{x}\log(u_{n}^{x}t)(u_{n}^{x}t)^{-1/\gamma(x)} 
          - \overline{F}^{x}(u_{n}^{x})t^{-1/\gamma(x)}
          \right \vert \\
          & \leq
          \frac{n}{k_{n}} \Big(
          \left \vert 
          \overline{F}^{x}(t u_{n}^{x}) 
          -c_{x}\log(u_{n}^{x}t)(u_{n}^{x}t)^{-1/\gamma(x)} \right \vert
          \\
          & \qquad 
          +
         t^{-1/\gamma(x)}\left \vert 
          c_{x}\log(u_{n}^{x})(u_{n}^{x})^{-1/\gamma(x)} 
          - \overline{F}^{x}(u_{n}^{x})
          \right \vert 
          +
          c_{x}\log(t)(u_{n}^{x}t)^{-1/\gamma(x)} \Big) \\
          & \leq
          \text{cst}
          \frac{n}{k_{n}} \left(
          (tu_{n}^{x})^{-\zeta_{x}}
           + 
         t^{-1/\gamma(x)}(u_{n}^{x})^{-\zeta_{x}}
          +
          c_{x}\log(t)(u_{n}^{x}t)^{-1/\gamma(x)} \right) \\
          &\leq
          \text{cst} 
          \frac{n}{k_{n}}
          t^{-1/\gamma(x) } (u_{n}^{x})^{-1/\gamma(x) }\left(
         1+ \log(t)\right).
    \end{align*}
    Equation~\eqref{cond_second_order_log} yields that
    \begin{align*}
        (u_{n}^{x})^{-1/\gamma(x)}\log(u_{n}^{x})
        = \frac{k_n}{n} + \mathcal{O}\left((u_{n}^{x})^{-\zeta_{x}}\right),
    \end{align*}
    and consequently
    \begin{align*}
         &\sqrt{k_{n}h_{n}}\left \vert \frac{\overline{F}^{x}(t u_{n}^{x})}
         {\overline{F}^{x}(u_{n}^{x})}
         -
          t^{-1/\gamma(x)}\right \vert
          \leq
          \text{cst} \
           t^{-1/\gamma(x) } \left(
         1+ \log(t)\right)\sqrt{k_{n}h_{n}}
           \frac{n}{k_{n}}
          \frac{(u_{n}^{x})^{-1/\gamma(x) } \log(u_{n}^{x})}{\log(u_{n}^{x})}
          \\
          &=
          \text{cst} \
           t^{-1/\gamma(x) } \left(
         1+ \log(t)\right)\sqrt{k_{n}h_{n}}
           \frac{n}{k_{n}}
          \frac{\frac{k_n}{n} + \mathcal{O}\left((u_{n}^{x})^{-\zeta_{x}}\right)}{\log(u_{n}^{x})}
          \\
          &=
          \text{cst} \
           t^{-1/\gamma(x) } \left(
         1+ \log(t)\right)
         \left(\frac{\sqrt{k_{n}h_{n}}}{\log(u_{n}^{x})} + \frac{\mathcal{O}\left(n\sqrt{h_{n}}(u_{n}^{x})^{-\zeta_{x}}\right)}{\log(u_{n}^{x})\sqrt{k_{n}}} \right).
    \end{align*}
    The mapping $t \mapsto t^{-1/\gamma(x) - 1}\left(1+ \log(t)\right) $ is integrable on $(1,\infty)$ and hence it remains to examine the deterministic sequence
    \begin{align*}
        \left(\frac{\sqrt{k_{n}h_{n}}}{\log(u_{n}^{x})} + \frac{\mathcal{O}\left(n\sqrt{h_{n}}(u_{n}^{x})^{-\zeta_{x}}\right)}{\log(u_{n}^{x})\sqrt{k_{n}}} \right)
        =
        \frac{\sqrt{k_{n}h_{n}}}{\log(u_{n}^{x})} 
        \left( 1 + \frac{n}{k_{n} (u_{n}^{x})^{\zeta_{x}} } \right).
    \end{align*}
    Notice first that, since $\zeta_{x}\gamma(x) > 1$,
    \begin{align*}
        \frac{n}{k_{n} (u_{n}^{x})^{\zeta_{x}}} 
        =
        \frac{1}{ \overline{F}^{x}(u_{n}^{x}) (u_{n}^{x})^{\zeta_{x}}} \to 0,
    \end{align*}
    As argued in the proof of Theorem~\ref{thm:sequence_cons_m_dependence} we have $ \log\left(n/k_n\right) \sim (1/\gamma(x))\log(u_{n}^{x})$ and thus
    \begin{align*}
        \frac{\sqrt{k_{n}h_{n}}}{\log(u_{n}^{x})} 
        \left( 1 + \frac{n}{k_{n} (u_{n}^{x})^{\zeta_{x}} } \right)
        \sim
        \frac{\sqrt{k_{n}h_{n}}}{\gamma(x)\log(n/k_{n})} 
        \left( 1 + \frac{n}{k_{n} (u_{n}^{x})^{\zeta_{x}} } \right) \to 0.
    \end{align*}
    To verify that condition~\ref{B_n_assumption} holds, we can repeat some of the arguments from Lemma~\ref{lemma_second_order_conditions}:
    \begin{align*}
         \sqrt{k_{n}h_{n}}B_{n}^{x}(s_{0}) &\leq
          \mathcal{O}( \sqrt{k_{n}h_{n}} h_{n})
          +
          \mathcal{O}\left( \sqrt{k_{n}h_{n}}
         \log\left(u_{n}^{x}\right)h_{n}\right)
          +
           \mathcal{O}\left( \frac{\sqrt{k_{n}h_{n}}}
       {\log\left(n/k_{n}\right)}\right)  \\
       &=
        \mathcal{O}\left( \sqrt{k_{n}h_{n}^{3}}
         \log\left(n/k_{n}\right)\right)
          +
          \mathcal{O}\left( \frac{\sqrt{k_{n}h_{n}}}
       {\log\left(n/k_{n}\right)}\right) \to 0.
    \end{align*}
In a similar way, arguments akin to those of Lemma~\ref{lemma_second_order_conditions}  also yield that condition~\ref{integral_B_n_assumption} holds.
\end{proof}

\begin{proof}[Proof of Lemma~\ref{lemma::anti_clustering_sequence}]
    Consider the expression of condition~\ref{anti_clustering_condition}:
    \begin{align*}
       &\sum_{j=m}^{r_n}\frac{\E{K\left(\frac{x-X_0}{h_n}\right)1_{\left\{Y_0 > su_{n}^{x}\right\}}
       K\left(\frac{x-X_j}{h_n}\right)1_{\left\{Y_j > tu_{n}^{x}\right\}}}}
       {h_{n} \overline{F}^{x}(u_{n}^{x})} \\
       & \quad=
       \sum_{j=m}^{r_n}\frac{\text{cov}\left(K\left(\frac{x-X_0}{h_n}\right)1_{\left\{Y_0 > su_{n}^{x}\right\}}
       K\left(\frac{x-X_j}{h_n}\right)1_{\left\{Y_j > tu_{n}^{x}\right\}}\right)}
       {h_{n} \overline{F}^{x}(u_{n}^{x})} \\
       & \qquad +
       (r_n-m)h_{n} \overline{F}^{x}(u_{n}^{x})\frac{\E{K\left(\frac{x-X_0}{h_n}\right)1_{\left\{Y_0 > su_{n}^{x}\right\}}}}
       {h_{n} \overline{F}^{x}(u_{n}^{x})} 
       \frac{\E{K\left(\frac{x-X_0}{h_n}\right)1_{\left\{Y_0 > tu_{n}^{x}\right\}}}}
       {h_{n} \overline{F}^{x}(u_{n}^{x})}, 
    \end{align*}
    where the latter term is negligible due to condition~\ref{rate_condition}.
    Calculations similar to those of the proof of Theorem~\ref{Empirical_tail_dist_consistency_theorem} yield
    \begin{align*}
        &\sum_{j=m}^{r_n}\frac{\text{cov}\left(K\left(\frac{x-X_0}{h_n}\right)1_{\left\{Y_0 > su_{n}^{x}\right\}}
       K\left(\frac{x-X_j}{h_n}\right)1_{\left\{Y_j > tu_{n}^{x}\right\}}\right)}
       {h_{n} \overline{F}^{x}(u_{n}^{x})}
       \leq \frac{\text{cst}}{h_{n} \overline{F}^{x}(u_{n}^{x})} \left(h_{n}^{2-c_{x}} \overline{F}^{x}(u_{n}^{x}) + h_{n}^{c_{x}(\eta-1)} \right) \\
       &\quad = \text{cst} (h_{n}^{1-c_{x}} + h_{n}^{c_{x}(\eta-1) -1}\overline{F}^{x}(u_{n}^{x})^{-1}) \to 0,
    \end{align*}
    as $n \to \infty$, for some constant $c_{x} > 0$, by assumption. This yields condition~\ref{anti_clustering_condition}.  \\
    We proceed to establish the logarithmic anti-clustering condition~\ref{log_anti_clustering_condition}, focusing, as above, on the covariance rather than the cross-term.
    Let $q>2$ and $p>1$ such that $2/q = 1- 1/p$. Rio's covariance inequality and calculations similar to those in the proof of Theorem~\ref{Hill_normality_thm} yield
    \begin{align*}
        &\left\vert \text{cov}\left( K\left(\frac{x-X_{0}}{h_n}\right)
         \log_{+}\left(\frac{Y_{0}}{u_{n}^{x}A} \right),
         K\left(\frac{x-X_{j}}{h_n}\right) 
         \log_{+}\left(\frac{Y_{j}}{u_{n}^{x}A} \right) \right) \right\vert \\
         &\leq 
         2p (2\alpha(j))^{1/p}\E{ K^{q}\left(\frac{x-X_{0}}{h_n}\right)
         \log_{+}^{q}\left(\frac{Y_{0}}{u_{n}^{x}A} \right)}^{2/q}
         =
          \mathcal{O}(\alpha(j)^{1/p} h_{n}^{2/q} \overline{F}^{x}(u_{n}^{x})^{2/q}).
    \end{align*}
  In addition it holds that
  \begin{align*}
        &\E{K\left(\frac{x-X_{0}}{h_n}\right) K\left(\frac{x-X_{j}}{h_n}\right) 
         \log_{+}\left(\frac{Y_{0}}{u_{n}^{x}A} \right)
         \log_{+}\left(\frac{Y_{j}}{u_{n}^{x}A} \right)} \\
         &= 
          h_{n}^{2}\overline{F}^{x}(u_{n}^{x})
        \int_{[-1,1]^2}  \biggl\{
         \frac{\CE{ \log_{+}\left(\frac{Y_{0}}{u_{n}^{x}A} \right)
         \log_{+}\left(\frac{Y_{j}}{u_{n}^{x}A} \right)}{X_{0} = x + h_{n}w_{1}, X_{j} = x + h_{n}w_{2}}}
         {\overline{F}^{x}(u_{n}^{x})} \\
         & \qquad \qquad \qquad \qquad \qquad \qquad 
         \times g(x + w_{1}h_{n},x + w_{2}h_{n})K\left(w_{1}\right)K\left(w_{2}\right) \biggr\}
         \mathrm{d}(w_1, w_2) = \mathcal{O}\left( h_{n}^{2}\overline{F}^{x}(u_{n}^{x})\right),
  \end{align*}
  due to~\eqref{eq::Expected_cross_log_limit}. The above two inequalities give 
  \begin{align*}
      &\frac{1}{h_{n} \overline{F}^{x}(u_{n}^{x})}
    \left|\sum_{j=m}^{r_n} \text{cov}\left( K\left(\frac{x-X_{0}}{h_n}\right)
         \log_{+}\left(\frac{Y_{0}}{u_{n}^{x}A} \right),
         K\left(\frac{x-X_{j}}{h_n}\right) 
         \log_{+}\left(\frac{Y_{j}}{u_{n}^{x}A} \right) \right) \right| \\
    & \leq
   \frac{\text{cst}}{h_{n} \overline{F}^{x}(u_{n}^{x})}
    \left(
    \sum_{j=1}^{d_n -1} h_{n}^{2} \overline{F}^{x}(u_{n}^{x}) + \sum_{j = d_n}^{\infty} h_{n}^{2/q} \overline{F}^{x}(u_{n}^{x})^{2/q} j^{-\eta/p}
    \right) \\
    &\leq
   \text{cst}
    \left(
    d_n h_{n} + d_{n}^{-\eta/p +1} h_{n}^{2/q-1} \overline{F}^{x}(u_{n}^{x})^{2/q-1}
    \right) \sim
    \text{cst}\left(
         h_n^{1-c'_x} + h_{n}^{-1/p + c'_{x}(\eta/p -1)}\overline{F}^{x}(u_{n}^{x})^{-1/p} \right) \to 0,
  \end{align*}
  where as above, we have  $d_{n} = \lceil h_{n}^{-c'_{x}} \rceil$ for some $c'_{x} > 0$.
\end{proof}

\begin{proof}[Proof of Theorem~\ref{thm:existence_of_sequences}]
    Theorem~\ref{thm:sequence_cons_power} ensures that the first two limit conditions for consistency are fulfilled for $k_{n} \asymp n^{\delta}$ for $\delta \in (0,3\nu/(2(1+\nu)))$ and $h_{n} \asymp \overline{F}^{x}(u_{n}^{x})^{e}$ for $e \in (\delta/(3(1-\delta)), \delta/(1-\delta))$. 
    Let $r_{n} \asymp n^{a}, \ a = 1- 2\delta/3$. Then $r_{n}/n \to 0$ and
    \begin{align*}
        r_{n}h_{n}\overline{F}^{x}(u_{n}^{x})
        \asymp 
        n^{a} \left(\frac{k_{n}}{n}\right)^{1+e} 
        \asymp
        n^{a -(1-\delta)(1+e)} \to 0,
    \end{align*}
    since $(1-\delta)(1+e) > 1- 2\delta/3$.
    Let $\ell_{n} \asymp n^{\psi}, \ \psi \in (0,a)$. Then $1/\ell_{n} \to 0$, $\ell_{n}/r_{n} \to 0$ and
    \begin{align*}
        \frac{n}{r_n}\beta_{\ell_n}
        \leq \text{cst} \
        n^{1-a-\nu\psi} \to 0,
    \end{align*}
    since upper bound on $\delta$ ensures that for every $\nu > 0$ there exists a small $\varepsilon > 0$ such that for $\psi = a-\varepsilon$, it holds that $1-a-\nu\psi < 0$.
    This concludes the second limit condition. Lemma~\ref{lemma:g_n_emp_process_negligibility_alpha_mixing} yields that the third condition holds if $\lim_{n \to \infty} \sqrt{\overline{F}^{x}(u_{n}^{x})nh_{n}}h_{n}^{2} = 0$. This is indeed true since
    \begin{align*}
        \sqrt{\overline{F}^{x}(u_{n}^{x})nh_{n}}h_{n}^{2} 
        \asymp
        \sqrt{n\left(\frac{k_n}{n}\right)^{1+5e}}
      \asymp
      \sqrt{n^{1-(1+5e)(1-\delta)}},
    \end{align*}
    which vanishes asymptotically for $\delta/(5(1-\delta)) < \delta/(3(1-\delta)) < e$. According to Lemma~\ref{lemma_second_order_conditions}  it suffices to check the limits of equation~\eqref{eq:second_order_derived_sequences} in order ensure the last two of the above five conditions. The lower bound on $e$ ensures
    \begin{align*}
         \sqrt{k_{n}h_{n}^{3}}
         \log\left(n/k_{n}\right)
         \asymp
         \sqrt{n^{\delta - 3e(1-\delta)}}\log(n^{1-\delta}) \to 0.
    \end{align*}
    Further restricting, let
    \begin{align*}
        \delta \in \left(0, \ \frac{3}{2}\frac{\nu}{1+\nu} \wedge \frac{2(\zeta_{x} -1/\gamma(x))}{2\zeta_{x} - 1/\gamma(x)}\right),
    \end{align*}
    where the latter fraction of the upper bound of $\delta$ ensures that $e$ is in a non-empty interval, since $\zeta_{x}\gamma(x) > 1$, by assumption. It then holds that
    \begin{align*}
        \sqrt{h_{n}}
         k_{n}^{\zeta_{x}\gamma(x) - \frac{1}{2}}n^{1-\zeta_{x}\gamma(x)}
         \asymp
         \sqrt{h_{n}}
         n^{\delta(\zeta_{x}\gamma(x) - \frac{1}{2}) + 1-\zeta_{x}\gamma(x)}
         \to 0,
    \end{align*}
    which concludes the last two of the above five conditions.
    
    According to Lemma~\ref{lemma::anti_clustering_sequence}, the anti-clustering conditions follow if equation~\eqref{eq:sequence_anti_clustering_restrains} holds. The first limit expression of equation~\eqref{eq:sequence_anti_clustering_restrains} is given by
    \begin{align*}
        h_{n}^{1-c_{x}} + h_{n}^{c_{x}(\eta-1) -1}\overline{F}^{x}(u_{n}^{x})^{-1}
        \asymp
        \overline{F}^{x}(u_{n}^{x})^{e(1-c_{x})} + \overline{F}^{x}(u_{n}^{x})^{ec_{x}(\eta-1) + (e+1)}.
    \end{align*}
    Since $\eta > 2$, there exists $c_{x} \in (0,1)$, such that the exponents above are positive, and consequently that the expression converges to zero as $n\to \infty$.
    For the second limit expression of equation~\eqref{eq:sequence_anti_clustering_restrains}, note that
    \begin{align*}
         h_n^{1-c'_x} + h_{n}^{-1/p + c'_{x}(\eta/p -1)}\overline{F}^{x}(u_{n}^{x})^{-1/p}     
         \asymp
         \overline{F}^{x}(u_{n}^{x})^{e(1-c'_{x})} + \overline{F}^{x}(u_{n}^{x})^{ec'_{x}(\eta/p-1) + (e+1)/p}.
    \end{align*}
     For $c'_{x} \in (0,1)$ and $p \in (1, (1+e)/(ec'_{x})+ \eta) $ the exponents above are positive, and consequently the expression converges to zero as $n\to \infty$.
\end{proof}
\begin{proof}[Proof of Theorem~\ref{thm:existence_of_sequences_log}]
    Let  $h_{n}\asymp \log(n)^{-e} $ for $e \in (4/3, 2)$ and let $k_{n} \asymp \log(n)^{2} $. Then, according to Theorem~\ref{thm:sequence_cons_m_dependence}, the first two conditions for consistency is satisfied. The proof is similar to that of Theorem~\ref{thm:existence_of_sequences}, thus we only show that the statements hold, where adjustment is needed.\\
    It holds that $k_{n}/n \asymp \log(n)^{2}/n \to  0$ and hence  $\lim_{n\to \infty}u_{n}^{x} = \infty$. 
    Secondly
    \begin{align*}
        n h_{n} \overline{F}^{x}(u_{n}^{x}) 
        = k_{n}h_{n} 
       \asymp\log(n)^{2-e}\to \infty.
    \end{align*}
    As argued in the proof of Theorem~\ref{thm:sequence_cons_m_dependence}, we have $ \log\left(n/k_n\right) \sim (1/\gamma(x))\log(u_{n}^{x})$ and thus
    \begin{align*}
        h_{n}\log(u_{n}^{x}) 
        \sim
        \gamma(x)h_{n}\log(n/k_{n})
        \leq
        \text{cst}\log(n)^{1-e} \to 0.
    \end{align*}
    This concludes the three consistency conditions. Let $r_{n} \asymp n^{a}, \ a \in (0,1)$. Then $r_{n}/n \to 0$ and
    \begin{align*}
        r_{n}h_{n}\overline{F}^{x}(u_{n}^{x})
       \asymp
        n^{a-1} \log(n)^{2-e}
        \to 0.
    \end{align*}
    Let $\ell_{n} \asymp n^{\psi}, \ \psi \in (0,a)$. Then $1/\ell_{n} \to 0$, $\ell_{n}/r_{n} \to 0$ and
    \begin{align*}
        \frac{n}{r_n}\beta_{\ell_n}
        \leq \text{cst} \
        n^{1-a-\nu\psi} \to 0,
    \end{align*}
     since for every $\nu > 0$ there exists $a \in (0,1)$ and $\psi \in (0,a)$ such that $1-a-\nu\psi < 0$.
    This concludes the second normality limit condition above. Note also that
    \begin{align*}
        \sqrt{\overline{F}^{x}(u_{n}^{x})nh_{n}}h_{n}^{2}
        \asymp
        \sqrt{k_{n}h_{n}}h_{n}^{2}
        \asymp
        \log(n)^{1-5e/2}  \to 0,
    \end{align*}
    which, according to Lemma~\ref{lemma:g_n_emp_process_negligibility_alpha_mixing}, implies the third condition. The last two conditions for asymptotic normality are fulfilled by Lemma~\ref{lemma_log_second_order_conditions}, since 
    \begin{align*}
      \frac{\sqrt{k_{n}h_{n}}}
       {\log\left(n/k_{n}\right)}
       \sim
      \frac{\sqrt{k_{n}h_{n}}}
       {\log\left(n\right)}
       \asymp
       \frac{\sqrt{\log\left(n\right)^{2-e}}}
       {\log\left(n\right)} 
       = \log\left(n\right)^{-e/2} \to 0,
    \end{align*}
    and 
    \begin{align*}
        &\sqrt{k_{n}h_{n}}h_{n}\log(u_{n}^{x})
        \sim
        \gamma(x)\sqrt{k_{n}h_{n}}h_{n}\log(n/k_{n})\leq \text{cst} \sqrt{k_{n}h_{n}}h_{n}\log(n)
        \asymp
        \log\left(n\right)^{2-3e/2} \to 0.
    \end{align*}
    which concludes the last two of the above five normality conditions.
\end{proof}

\section{Anti-clustering conditions for simulations} \label{app_AC}
When checking that the time series in the simulation study satisfy the required conditions, Lemma~\ref{lemma::anti_clustering_sequence} can used to verify anti-clustering conditions~\ref{anti_clustering_condition} and~\ref{log_anti_clustering_condition}. In this appendix we prove that equation~\eqref{eq::Expected_cross_log_limit} holds for the proposed conditional Fréchet and Pareto models, and also provide a result that can be used to check the anti-clustering conditions for the proposed conditional max-stable process.

\subsection{Conditional Fréchet and Pareto}
Consider covariate process $\{X_j\}$, belonging to some stationary (possibly long-memory) time series model and $\{U_j\}$ stationary (possibly long-memory) with marginals in $[0,1]$. The constructions
\begin{align} 
    Y_{j} &= (-\log(U_j))^{-\gamma(X_j)}, \quad j \in \Z, \label{eq:cond_frechet} \\
    Y_{j} &= U_j^{-\gamma(X_j)}, \quad j \in \Z \label{eq:cond_pareto},
\end{align}
gives rise to a time series models with marginal $\text{Fréchet}(1/\gamma(X_j), 1, 0)$ and $\text{Pareto}(1, 1/\gamma(X_j))$ distributions respectively for every $j$. 

\begin{lemma}\label{Lemma:cond_par_fre_AC_requirements}
    Assume conditions~\ref{Assu:gamma_lip}-\ref{Assu:n_h_F_limit}, with either~\eqref{eq:cond_frechet} or~\eqref{eq:cond_pareto}. Then condition~\ref{Assu:Expected_cross_indicator_limit} and~\eqref{eq::Expected_cross_log_limit} holds.
\end{lemma}
\begin{proof}
    Let $s,t > 0$ be given. In both equation~\eqref{eq:cond_frechet} and~\eqref{eq:cond_pareto}, the Cauchy-Schwarz inequality yields that
    \begin{align*}
     &\CE{1_{\left\{Y_0 > su_{n}^{x}\right\}}
      1_{\left\{Y_j > tu_{n}^{x}\right\}}}{X_{0} = x + h_{n}w_{1}, X_{j} = x + h_{n}w_{2}}/\overline{F}^{x}(u_{n}^{x}) \\
      \leq 
       &\frac{\CE{1_{\left\{Y_0 > su_{n}^{x}\right\}}
      }{X_{0} = x + h_{n}w_{1}, X_{j} = x + h_{n}w_{2}}^{1/2}
      \CE{1_{\left\{Y_j > tu_{n}^{x}\right\}}}{X_{0} = x + h_{n}w_{1}, X_{j} = x + h_{n}w_{2}}^{1/2}
      }{\overline{F}^{x}(u_{n}^{x})} \\
      &=
      \left(\frac{\CE{1_{\left\{Y_0 > su_{n}^{x}\right\}}
      }{X_{0} = x + h_{n}w_{1}}
      }{\overline{F}^{x}(u_{n}^{x})}\right)^{1/2}
      \left(\frac{\CE{1_{\left\{Y_j > tu_{n}^{x}\right\}}}{ X_{j} = x + h_{n}w_{2}}
      }{\overline{F}^{x}(u_{n}^{x})}\right)^{1/2}
      = \mathcal{O}(1),
    \end{align*}
    uniformly in $w_1, w_2 \in [-1,1]$, where we have used Potter's bound and the Lipschitz conditions on $L^{x}$ and $1/\gamma(x)$ as was done in Lemma~\ref{Limit_of_tail_process}. This shows condition~\ref{Assu:Expected_cross_indicator_limit}. \\
        Showing equation~\eqref{eq::Expected_cross_log_limit} is done analogously 
    \begin{align*}
          &\frac{\CE{ \log_{+}\left(\frac{Y_{0}}{u_{n}^{x}A} \right)
         \log_{+}\left(\frac{Y_{j}}{u_{n}^{x}A} \right)}{X_{0} = x + h_{n}w_{1}, X_{j} = x + h_{n}w_{2}}}{\overline{F}^{x}(u_{n}^{x})} \\
         &\leq
         \left(\frac{\CE{ \log_{+}^{2}\left(\frac{Y_{0}}{u_{n}^{x}A} \right)
         }{X_{0} = x + h_{n}w_{1}}}
         {\overline{F}^{x}(u_{n}^{x})}\right)^{1/2}
         \left(\frac{\CE{ \log_{+}^{2}\left(\frac{Y_{0}}{u_{n}^{x}A} \right)
         }{X_{0} = x + h_{n}w_{2}}}
         {\overline{F}^{x}(u_{n}^{x})}\right)^{1/2}\\
         & \quad \to \int_{A}^{\infty} \log^{2}(t) (1/\gamma(x))t^{-1/\gamma(x)-1} \dd t,
    \end{align*}
    uniformly in $w_1, w_2 \in [-1,1]$ as shown in the proof of Theorem~\ref{Hill_normality_thm}. We conclude that also equation~\eqref{eq::Expected_cross_log_limit} holds.
\end{proof}
\subsection{Conditional Max-Stable process}
The construction of equation~\eqref{eq:MS_cond_construction}, can be summarised in the following diagram, when the centred Gaussian processes are started at zero and $(P^{X_i} \independent P^{X_j}) \vert (X_i,X_j)$.
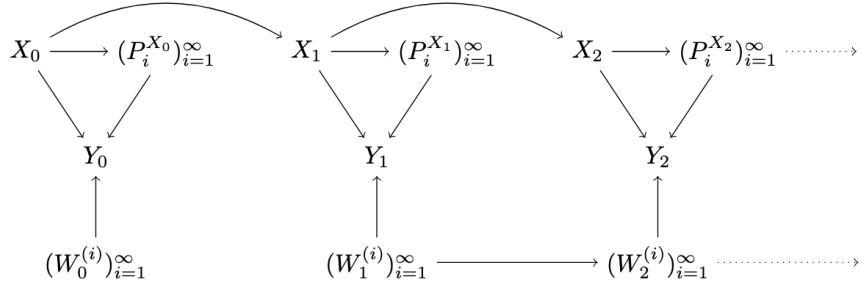
\begin{figure}[hbt!]
\centering
\begin{tikzpicture}[scale=1, every node/.style={scale=1}]

    % Nodes
    \node (X0) at (0, 3) {$X_0$};
    \node (PX0) at (2, 3) {$(P_{i}^{X_0})_{i=1}^{\infty}$};
    \node (X1) at (4, 3) {$X_1$};
    \node (PX1) at (6, 3) {$(P_i^{X_1} )_{i=1}^{\infty}$};
    \node (X2) at (8, 3) {$X_2$};
    \node (PX2) at (10, 3) {$(P_i^{X_2} )_{i=1}^{\infty}$};

    \node (Y0) at (1, 1.5) {$Y_0$};
    \node (Y1) at (5, 1.5) {$Y_1$};
    \node (Y2) at (9, 1.5) {$Y_2$};

    \node (W0) at (1, 0) {$(W_0^{(i)})_{i=1}^{\infty}$};
    \node (W1) at (5, 0) {$(W_1^{(i)})_{i=1}^{\infty}$};
    \node (W2) at (9, 0) {$(W_2^{(i)})_{i=1}^{\infty}$};

    % Arrows
    \draw[->] (X0) -- (PX0);
    \draw[->] (X1) -- (PX1);
    \draw[->] (X2) -- (PX2);
    
    \draw[->] (X0) -- (Y0);
    \draw[->] (PX0) -- (Y0);
    \draw[->] (X1) -- (Y1);
    \draw[->] (PX1) -- (Y1);
    \draw[->] (X2) -- (Y2);
    \draw[->] (PX2) -- (Y2);

    \draw[->] (W0) -- (Y0);
    \draw[->] (W1) -- (Y1);
    \draw[->] (W2) -- (Y2);

    \draw[->] (W1) -- (W2);

    % Long Arrows
    \draw[->, bend left=30] (X0) to (X1);
    \draw[->, bend left=30] (X1) to (X2);

    % Dotted arrow continuing to the right
    \draw[->, dotted] (PX2.east) -- ++(1,0);
    \draw[->, dotted] (W2.east) -- ++(2,0);

\end{tikzpicture}
\caption{
Diagram of the conditional Max-stable (CSGMS) construction.
Since the Gaussian processes $\textbf{W}^{(i)}, i \geq 1$ are all started in zero, the distribution of the processes at time $1$, namely $(W_{1}^{(i)})_{i=1}^{\infty}$, doesn't depend on the starting values $(W_{0}^{(i)})_{i=1}^{\infty}$ as they are deterministic. Hence no arrow between $(W_{0}^{(i)})_{i=1}^{\infty}$ and $(W_{1}^{(i)})_{i=1}^{\infty}$.}
\label{fig:MS_diagram}
\end{figure}
Since $(W_{0}^{(i)})_{i=1}^{\infty}$ is an infinite dimensional vector of zeroes, it holds that $( Y_{0} \independent Y_{j}) \vert (X_0, X_j)$ for every $j$ and $Y_j \vert X_j, X_i \sim Y_j \vert X_j$ for $ i\neq j$. For time series in our setup, satisfying these distributional properties, it turns out that showing the anti-clustering conditions can be done without the use of Lemma~\ref{lemma::anti_clustering_sequence}. This is content of the Lemma below,
\begin{lemma}\label{lemma:cond_ind_impl_anti_clustering}
    Assume~\eqref{def:condRV}, conditions~\ref{Assu:gamma_lip} -~\ref{Assu:n_h_F_limit} and~\ref{rate_condition}. If $( Y_{0} \independent Y_{j}) \vert (X_0, X_j)$ for every $j$, and $Y_j \vert X_j, X_i \sim Y_j \vert X_j$ for $ i\neq j$, then conditions~\ref{anti_clustering_condition} and~\ref{log_anti_clustering_condition} hold.
\end{lemma}
\begin{proof}
    Consider the numerator in condition~\ref{anti_clustering_condition}
    \begin{align*}
       &\E{K\left(\frac{x-X_0}{h_n}\right)1_{\left\{Y_0 > su_{n}^{x}\right\}}
       K\left(\frac{x-X_j}{h_n}\right)1_{\left\{Y_j > tu_{n}^{x}\right\}}} \\
       &=
       \E{K\left(\frac{x-X_0}{h_n}\right)
       K\left(\frac{x-X_j}{h_n}\right)
       \CE{1_{\left\{Y_0 > su_{n}^{x}\right\}}}{X_{0},X_{j}}
       \CE{1_{\left\{Y_j > tu_{n}^{x}\right\}}}{X_{0},X_{j}}}\\
       &=
       \E{K\left(\frac{x-X_0}{h_n}\right)
       K\left(\frac{x-X_j}{h_n}\right)
       \CE{1_{\left\{Y_0 > su_{n}^{x}\right\}}}{X_{0}}
       \CE{1_{\left\{Y_j > tu_{n}^{x}\right\}}}{X_{j}}} \\
       &= 
        (h_{n}\overline{F}^{x}(u_{n}^{x}))^{2}
        \int_{[-1,1]^2}  \biggl\{
         \frac{\CE{1_{\left\{Y_0 > su_{n}^{x}\right\}}}{X_{0} = x + w_{1}h_n}}
         {\overline{F}^{x}(u_{n}^{x})}
         \frac{ \CE{1_{\left\{Y_0 > tu_{n}^{x}\right\}}}{X_{0} = x + w_{2}h_n}}
         {\overline{F}^{x}(u_{n}^{x})}\\
         & \qquad \qquad \qquad \qquad \qquad \qquad \qquad
         \times g(x + w_{1}h_{n},x + w_{2}h_{n})K\left(w_{1}\right)K\left(w_{2}\right) \biggr\}
         \mathrm{d}(w_1, w_2)\\
         &= \mathcal{O}((h_{n}\overline{F}^{x}(u_{n}^{x}))^{2}),
    \end{align*}
    where we have used Potter's bound and the Lipschitz conditions on $L^{x}$ and $1/\gamma(x)$ as was done in Lemma~\ref{Limit_of_tail_process}. Consequently 
    \begin{align*}
       &\lim_{m\to \infty}
       \limsup_{n\to \infty} \sum_{j=m}^{r_n} 
       \frac{\E{K\left(\frac{x-X_0}{h_n}\right)1_{\left\{Y_0 > su_{n}^{x}\right\}}
       K\left(\frac{x-X_j}{h_n}\right)1_{\left\{Y_j > tu_{n}^{x}\right\}}}}
       {h_{n} \overline{F}^{x}(u_{n}^{x})} \\
       & \quad \leq 
       \limsup_{n\to \infty} \sum_{j=1}^{r_n} 
       \frac{\E{K\left(\frac{x-X_0}{h_n}\right)1_{\left\{Y_0 > su_{n}^{x}\right\}}
       K\left(\frac{x-X_j}{h_n}\right)1_{\left\{Y_j > tu_{n}^{x}\right\}}}}
       {h_{n} \overline{F}^{x}(u_{n}^{x})}  \leq
       \limsup_{n\to \infty} \mathcal{O}\left(r_{n} h_{n} \overline{F}^{x}(u_{n}^{x})\right)  = 0,
    \end{align*}
    which concludes condition~\ref{anti_clustering_condition}. Showing condition~\ref{log_anti_clustering_condition} is done analogously, using that
     \begin{align*}
        &\E{K\left(\frac{x-X_{0}}{h_n}\right) K\left(\frac{x-X_{j}}{h_n}\right) 
         \log_{+}\left(\frac{Y_{0}}{u_{n}^{x}A} \right)
         \log_{+}\left(\frac{Y_{j}}{u_{n}^{x}A} \right)} \\
         &= 
          (h_{n}\overline{F}^{x}(u_{n}^{x}))^{2}
        \int_{[-1,1]^2}  \biggl\{
         \frac{\CE{ \log_{+}\left(\frac{Y_{0}}{u_{n}^{x}A} \right)
         }{X_{0} = x + h_{n}w_{1}}}
         {\overline{F}^{x}(u_{n}^{x})}
         \frac{\CE{ \log_{+}\left(\frac{Y_{0}}{u_{n}^{x}A} \right)
         }{X_{0} = x + h_{n}w_{2}}}
         {\overline{F}^{x}(u_{n}^{x})}\\
         & \qquad \qquad \qquad \qquad \qquad \qquad 
         \times g(x + w_{1}h_{n},x + w_{2}h_{n})K\left(w_{1}\right)K\left(w_{2}\right) \biggr\}
         \mathrm{d}(w_1, w_2) \\
         &= \mathcal{O}( (h_{n}\overline{F}^{x}(u_{n}^{x}))^{2}),
  \end{align*}
  using calculations similar to those in the proof of Theorem~\ref{Hill_normality_thm}.
\end{proof}

\section{Kernel density estimation for mixing sequences}
The main task of this section is to prove consistency of the Parzen--Rosenblatt estimator for a stationary time series, under suitable assumptions involving either some form of anti-clustering or mixing condition. We also analyse the speed of convergence.

We begin with an auxiliary result which is used in all subsequent proofs.
\begin{lemma} \label{Expected_kernel_lemma}
    Consider a stationary time series $\{(X_{j}), j \in  \mathbb{Z} \}$ where each $X_j$ takes values in $[-1,1]$ and admits twice continuously differentiable density $g$. Let the kernel $K$ be a bounded, symmetric and defined on $[-1,1]$,
    and let the bandwidth sequence $\curly{h_{n}}$ satisfy
    \begin{align*}
        \lim_{n\to \infty} h_{n}  = 0.
    \end{align*}
    For any $x \in \R$, it holds that
    \begin{align*}
        \frac{\E{K^{q}\p{\frac{x-X_0}{h_n}}}}{h_{n}}
        \to g(x)\int K^{q}\p{u}\mathrm{d}u.
    \end{align*}
\end{lemma}
\begin{proof}
    The left hand side of the limit can be rewritten as follows
    \begin{align*}
          \frac{\E{K^{q}\p{\frac{x-X_0}{h_n}}}}{h_{n}} =
         \frac{1}{h_{n}}\int  
         g(z)K^{q}\p{\frac{x-z}{h_n}}
         \mathrm{d}z = 
         \int  
         K^{q}\p{u}g(x + h_{n}u)
         \mathrm{d}u. 
    \end{align*}
    The Taylor expansion $g(x + h_{n}u) = g(x) + h_{n}ug^{(1)}(x) +h_{n}^{2}u^{2}g^{(2)}(x)/2 + \mathcal{O}\p{\lvert h_{n}u \rvert^{3}}$ yields
    \begin{align*}
          \frac{\E{K^{q}\p{\frac{x-X_0}{h_n}}}}{h_{n}} = 
          &g(x)\int K^{q}\p{u}\mathrm{d}u + 
          h_{n}g^{(1)}(x)\int uK^{q}\p{u}\mathrm{d}u 
          \\
          & +
          \frac{h_{n}^{2}g^{(2)}(x)}{2}\int  u^{2}K(u)\mathrm{d}u + 
          o\p{h_{n}^{2}} \to g(x)\int  K^{q}\p{u}\mathrm{d}u,
    \end{align*}
    as $n \to \infty$.
\end{proof}
The following Lemma provides consistency when assuming anti-clustering.
\begin{lemma} \label{lemma:g_n_consistent_anti_clust}
    Assume that $h_{n} \to 0$, $nh_{n} \to  \infty$, that the time series $\{(X_{j}), j \in  \mathbb{Z} \}$ is stationary with 
    \begin{align*}
         \lim_{m\to \infty}
       \limsup_{n\to \infty} \frac{1}{n h_{n}} \sum_{j=m}^{n} 
       \frac{\E{K\left(\frac{x-X_0}{h_n}\right)
       K\left(\frac{x-X_j}{h_n}\right)}}{h_{n}} = 0,
    \end{align*}
    where each $X_j$ admits twice continuously differentiable density $g$, and the joint density $g_{X_i, X_j}$ of $(X_i, X_j)$ exists, and is bounded for all $i,j \in \Z$.\\
    Then for any $x \in \R$, it holds that $g_{n}(x) \overset{\P}{\to} g(x)$.
\end{lemma}
\begin{proof}
    Fix $x \in \R$. The bias is negligible since
    \begin{align*}
        &\E{g_{n}(x)}-g(x) =
        \int \frac{1}{h_{n}} K\left(\frac{x-y}{h_{n}}\right)\{g(y)-g(x)\} \mathrm{d} y\\
        &=\int K(u)\{g'(x)(-h_{n}u)+\frac{1}{2} g''(x)(h_{n}u)^2+o(h_{n}^2)\} \mathrm{d} u=\frac{h_{n}^2}{2}g''(x)\int u^2K(u)\mathrm{d} u+o(h_{n}^2) \to 0,
\end{align*}
as $n\to \infty$. We consider the variance
\begin{align*}
    &\text{var}\left( g_{n}(x) \right) 
    =
    \frac{1}{(nh_{n})^{2}}\text{var}\left( \sum_{i=1}^{n} K\left(\frac{x-X_{j}}{h_n}\right) \right) \\
    &= \frac{1}{(nh_{n})^{2}}\Bigg(
      n\E{K^{2}\left(\frac{x-X_{0}}{h_{n}}\right)}  - 
     n\E{K\left(\frac{x-X_{0}}{h_{n}}\right)}^{2}  + 
      2n\sum_{i=1}^{n}\left(1 - \frac{i}{n} \right) \text{cov}\left(K\left(\frac{x-X_{0}}{h_{n}}\right),K\left(\frac{x-X_{i}}{h_{n}}\right)\right)\Bigg) \\
      & \leq 
      \frac{1}{nh_{n}^{2}} \text{var}\left(K\left(\frac{x-X_{0}}{h_{n}}\right)\right)
      +
     \frac{2}{nh_{n}^{2}}      
    \sum_{i=1}^{n}\left\vert\text{cov}\left(K\left(\frac{x-X_{0}}{h_{n}}\right),K\left(\frac{x-X_{i}}{h_{n}}\right)\right)\right\vert,
\end{align*}
where we used stationarity.
The first term is negligible by Lemma~\ref{Expected_kernel_lemma}, since $nh_n \to \infty$. \\
Since the joint density $g_{X_{0},X_{i}}$ is bounded for all $i \in \N$ it holds that
\begin{align*}
    &\E{K\left(\frac{x-X_{0}}{h_{n}}\right)K\left(\frac{x-X_{i}}{h_{n}}\right)} \\
    &=
    h_{n}^{2}\int\int K\left(u\right)K\left(y\right)g_{X_{0},X_{i}}(x - uh_{n},x - yh_{n}) \mathrm{d}u\mathrm{d}y \leq
    h_{n}^{2} \sup_{x_{1},x_{2} \in \R}(g_{X_{0},X_{i}}(x_1,x_2))
    =
    \mathcal{O}\left(h_{n}^{2}\right).
\end{align*}
Similarly $\E{K^{2}((x-X_{0})/h_{n})} = \mathcal{O}(h_{n}^{2})$, and consequently, for fixed $m \in \N$ we have
    \begin{align*}
     &\limsup_{n \to \infty} \frac{2}{nh_{n}^{2}}           
    \sum_{i=1}^{m}\left\vert \text{cov}\left(K\left(\frac{x-X_{0}}{h_{n}}\right),K\left(\frac{x-X_{i}}{h_{n}}\right)\right) \right\vert
    \leq
    \text{cst} \limsup_{n \to \infty} \frac{m}{n} = 0,
    \end{align*}
    and, by assumption, for given $\varepsilon > 0$, it is possible to pick $m_0$ so large that for all $m \geq m_0$ it holds that
    \begin{align*}
     &\limsup_{n \to \infty}\frac{2}{nh_{n}^{2}}            
    \sum_{i=m}^{n}\text{cov}\left(K\left(\frac{x-X_{0}}{h_{n}}\right),K\left(\frac{x-X_{i}}{h_{n}}\right)\right) \leq \varepsilon,
    \end{align*}
    where $\varepsilon$ can be made arbitrarily small. This implies that $\text{var}\left( g_{n}(x) \right) \to 0$, and hence $g_{n}(x) \overset{\P}{\to} g(x)$.
\end{proof}
The assumptions of the following Lemma implies those of Lemma~\ref{lemma:g_n_consistent_anti_clust}, and is thus a special case thereof.
\begin{lemma} \label{lemma:g_n_consistent_alpha_mixing}
    Assume that $h_{n} \to 0$, $nh_{n} \to  \infty$, that the time series $\{(X_{j}), j \in  \mathbb{Z} \}$ is stationary, $\alpha$-mixing with $\alpha$-mixing coefficient bounded by $\mathcal{O}(j^{-\eta}), \eta > 2$ for every $j \in \Z$, and each $X_j$ admits twice continuously differentiable density $g$, and the joint density $g_{X_i, X_j}$ of $(X_i, X_j)$ exists, and is bounded for all $i,j \in \Z$. Then for any $x \in \R$, it holds that $g_{n}(x) \overset{\P}{\to} g(x)$.
\end{lemma}
\begin{proof}
    The proof is similar to that of Lemma~\ref{lemma:g_n_consistent_anti_clust}, different only in showing that the sum of covariances disappears as $n\to \infty$. \\
    Let $\left\{d_{n}, n \in \N\right\}$ be an integer sequence given by $d_{n} = \lceil h_{n}^{-2/\eta} \rceil$. Using Billingsley's inequality with the $\alpha$-mixing assumption yields
\begin{align*}
    & \frac{2}{nh_{n}^{2}}     
    \sum_{i=1}^{n}\text{cov}\left(K\left(\frac{x-X_{0}}{h_{n}}\right),K\left(\frac{x-X_{i}}{h_{n}}\right)\right)
    \leq
    \text{cst} \frac{1}{nh_{n}^{2}} 
    \left(\sum_{i=1}^{d_{n}-1} h_{n}^{2}
    +
    \sum_{i=d_{n}}^{\infty}
    \alpha(i)
    \right)\\
    & \leq
    \text{cst}\left(\frac{d_{n}}{n}
    +
    \frac{d_{n}^{-\eta +1}}{nh_{n}^{2}}  
    \right) 
    \sim
    \text{cst} \left(nh_{n}^{2/\eta}\right)^{-1} \to 0,
    \end{align*}
    since $2/\eta \leq 1$.
\end{proof}
\begin{lemma} \label{lemma:g_n_consistent_m_dependent}
    Assume that $h_{n} \to 0$, $nh_{n} \to  \infty$, that the time series $\{(X_{j}), j \in  \mathbb{Z} \}$ is stationary, $m$-dependent and each $X_j$ admits twice continuously differentiable density $g$.
    Then for any $x \in \R$, it holds that $g_{n}(x) \overset{\P}{\to} g(x)$.
\end{lemma}
\begin{proof}
    By the same argument from Lemma~\ref{lemma:g_n_consistent_anti_clust}, the bias is asymptotically negligible.
    Let $m \in \N$ be given. By $m$-dependence we split  $g_{n}(x)$ into a main component, which is a sum of independent terms, and a negligible remainder term as follows
    \begin{align*}
         g_{n}(x) = 
        \frac{1}{m}\sum_{i = 1}^{m} g_{n}^{(i)}(x) + R_{n}(x), 
    \end{align*}
    where 
    \begin{align*}
        g_{n}^{(i)}(x) = 
        \frac{m}{nh_{n}}
        \sum_{j=1}^{\floor{\frac{n}{m}}}K\left(\frac{x-X_{(j-1)m + i}}{h_n}\right), \quad
        R_{n}(s) = 
        \frac{1}{n h_{n} }
    \sum_{j=\floor{\frac{n}{m}}m}^{n}K\left(\frac{x-X_j}{h_n}\right).
    \end{align*}
    Stationarity and $m$-dependence implies that each summand in $g_{n}^{(i)}(x)$ are i.i.d. which allow for a direct variance calculation
    \begin{align*}
         \Var \p{ \frac{1}{m}\sum_{i = 1}^{m} g_{n}^{(i)}(x)} &= 
         \frac{1}{m}\Var \p{g_{n}^{(i)}(x)}
         \frac{m}{\p{nh_{n}}^2}\floor{\frac{n}{m}}
         \Var\p{K\p{\frac{x-X_0}{h_n}}} \\
         &=
          \frac{\frac{m}{n}\floor{\frac{n}{m}}}{nh_{n}}
          \p{\frac{\E{K^{2}\p{\frac{x-X_0}{h_n}}}}{h_{n}} -
           \frac{\E{K\p{\frac{x-X_0}{h_n}}}}{h_{n}}}
           \to 0,
    \end{align*}
    using Lemma~\ref{Expected_kernel_lemma}. \\
    Turning our attention to the remainder term, notice that
    \begin{align*}
    \E{R_{n}(x)^2} &= \frac{1}{n^{2}h_{n}^2}
    \E{\left( \sum_{j=\left\lfloor\frac{n}{m}\right\rfloor m}^{n}K\left(\frac{x-X_j}{h_n}\right)\right)^2} 
    \\
    &= \frac{1}{n^{2}h_{n}^2} 
     \biggl\{ \sum_{j=\left\lfloor\frac{n}{m}\right\rfloor m}^{n} \E{K^2\left(\frac{x-X_j}{h_n}\right)} 
      + 2\sum_{j=\left\lfloor\frac{n}{m}\right\rfloor m}^{n}
    \sum_{i=j}^{n} 
    \E{K\left(\frac{x-X_i}{h_n}\right)K\left(\frac{x-X_j}{h_n}\right)} \biggr\} 
    \\
    &\leq \frac{\left(n-\left\lfloor\frac{n}{m}\right\rfloor m\right)}
    {n^{2}h_{n}^2}
    \curly{  \E{K^2\left(\frac{x-X_0}{h_n}\right)}  + 2\left(n-\left\lfloor\frac{n}{m}\right\rfloor m\right) \sqrt{\E{K^2\left(\frac{x-X_0}{h_n}\right)}^2} } \\
    &= \frac{\left(n-\left\lfloor\frac{n}{m}\right\rfloor m\right)}{n^2 h_{n}} \frac{\E{K^{2}\left(\frac{x-X_0}{h_n}\right)}}{h_{n}} \left(1 + 2\left(n-\left\lfloor\frac{n}{m}\right\rfloor m\right)\right) \to 0,
\end{align*}
    where we used the Cauchy-Schwarz inequality. The above calculations imply that the remainder is negligible which finishes the proof.
\end{proof}
\begin{lemma}\label{lemma:g_n_emp_process_negligibility_anti_cl}
    Assume that that $h_{n} \to 0$, $u_{n}^{x} \to \infty$, $nh_{n} \to  \infty$, that the time series $\{(X_{j}), j \in  \mathbb{Z} \}$ is stationary, each $X_j$ admits twice continuously differentiable density $g$,  
    \begin{align}\label{Condition:sqrt_n_h3_F}
        \sqrt{k_{n}h_{n}}h_{n}^{2} \to 0,
    \end{align}
    and
    \begin{align*}
         \lim_{m\to \infty}
       \limsup_{n\to \infty} \overline{F}^{x}(u_{n}^{x}) \sum_{j=m}^{n} 
       \frac{\E{K\left(\frac{x-X_0}{h_n}\right)
       K\left(\frac{x-X_j}{h_n}\right)}}{h_{n}} = 0.
    \end{align*}
   Then
    \begin{align*}
        \sqrt{k_{n}h_{n}}\curly{g(x)- g_{n}(x)} \overset{\P }{\to} 0.
    \end{align*}
\end{lemma}
\begin{proof}
Consider the following decomposition
    \begin{align*}
         &\sqrt{k_{n}h_{n}}\curly{g_{n}(x)- g(x)} =
          \sqrt{k_{n}h_{n}}\curly{g_{n}(x)- \E{g_{n}(x)}} +  \sqrt{k_{n}h_{n}}\curly{\E{g_{n}(x)}- g(x)}=: A_{n} + B_{n}.
    \end{align*}
    In Lemma~\ref{lemma:g_n_consistent_anti_clust}, we showed that the bias is of order $h_{n}^{2}$ and hence $B_{n}$ vanishes asymptotically as
\begin{align*}
    B_{n} = \sqrt{k_{n}h_{n}}\curly{\E{g_{n}(x)}- g(x)}
    =
    \mathcal{O}\left(\sqrt{k_{n}h_{n}}h_{n}^{2}\right) \to 0,
\end{align*}
by assumption. The term $A_{n}$ has mean zero, and hence it is asymptotically negligible if its variance goes to zero. 
     For fixed $m \in \N$ we have
    \begin{align*}
     &\limsup_{n \to \infty} \frac{2\overline{F}^{x}(u_{n}^{x})}{h_{n}}      
    \sum_{i=1}^{m}\text{cov}\left(K\left(\frac{x-X_{0}}{h_{n}}\right),K\left(\frac{x-X_{i}}{h_{n}}\right)\right) 
    \leq
    \text{cst} \limsup_{n \to \infty} m \overline{F}^{x}(u_{n}^{x})  h_{n} = 0,
    \end{align*}
    and for given $\varepsilon > 0$, it is possible to pick $m_0$ so large that for all $m \geq m_0$ it holds that
    \begin{align*}
     &\limsup_{n \to \infty} \frac{2\overline{F}^{x}(u_{n}^{x})}{h_{n}}      
    \sum_{i=m}^{n}\text{cov}\left(K\left(\frac{x-X_{0}}{h_{n}}\right),K\left(\frac{x-X_{i}}{h_{n}}\right)\right) \leq \varepsilon,
    \end{align*}
    where $\varepsilon$ can be made arbitrarily small. 
\end{proof}
\begin{lemma} \label{lemma:g_n_emp_process_negligibility_alpha_mixing}
    Assume that~\eqref{Condition:sqrt_n_h3_F} holds, that $h_{n} \to 0$, $u_{n}^{x} \to \infty$, $nh_{n} \to  \infty$, that the time series $\{(X_{j}), j \in  \mathbb{Z} \}$ is stationary, $\alpha$-mixing with $\alpha$-mixing coefficient bounded by $\mathcal{O}(j^{-\eta}), \eta > 2$ for every $j \in \Z$, and each $X_j$ admits twice continuously differentiable density $g$.\\
    Then
    \begin{align*}
        \sqrt{k_{n}h_{n}}\curly{g(x)- g_{n}(x)} \overset{\P }{\to} 0.
    \end{align*}
\end{lemma}
\begin{proof}
The proof is similar to that of Lemma~\ref{lemma:g_n_emp_process_negligibility_anti_cl}, different only in showing that the sum of covariances disappears as $n\to \infty$.
    By similar calculations as in the proof of Lemma~\ref{lemma:g_n_consistent_anti_clust} it holds that
\begin{align*}
    &\text{var}\left(A_{n}\right)
      \leq
      \frac{\overline{F}^{x}(u_{n}^{x})}{h_{n}}
      \text{var}\left(K\left(\frac{x-X_{0}}{h_{n}}\right)\right) +
    \frac{2\overline{F}^{x}(u_{n}^{x})}{h_{n}}      
    \sum_{i=1}^{n} \left\vert \text{cov}\left(K\left(\frac{x-X_{0}}{h_{n}}\right),K\left(\frac{x-X_{i}}{h_{n}}\right)\right)\right\vert.
\end{align*}
The first term has leading term of order $\overline{F}^{x}(u_{n}^{x})$, which vanishes asymptotically. For the second term, we again consider the integer sequence $\left\{d_{n}, n \in \N\right\}$ given by $d_{n} = \lceil h_{n}^{-2/\eta} \rceil$. Using Davydov's and Billingsley's inequality with the $\alpha$-mixing assumption yields
\begin{align*}
    &\frac{2\overline{F}^{x}(u_{n}^{x})}{h_{n}}      
    \sum_{i=1}^{n}\left\vert\text{cov}\left(K\left(\frac{x-X_{0}}{h_{n}}\right),K\left(\frac{x-X_{i}}{h_{n}}\right)\right)\right\vert \leq
    \text{cst}\frac{\overline{F}^{x}(u_{n}^{x})}{h_{n}} 
    \left(\sum_{i=1}^{d_{n}-1} h_{n}^{2}
    +
    \sum_{i=d_{n}}^{\infty}
    \alpha(i)
    \right)\\
    & \leq
    \text{cst}\left(\overline{F}^{x}(u_{n}^{x})h_{n}d_{n}
    +
    \frac{\overline{F}^{x}(u_{n}^{x})}{h_{n}} d_{n}^{-\eta +1}
    \right) 
    \sim
    \text{cst}\overline{F}^{x}(u_{n}^{x})h_{n}^{1-\frac{2}{\eta}} \to 0.
    \end{align*}
This yields that $\text{var}\left(A_{n}\right) \to 0$, and hence by Chebyshev's inequality we conclude the desired result.
\end{proof}

\section{Background on fundamental limit theorems}
\subsection*{Mixing}
We introduce $\alpha$- and $\beta$-mixing using~\citet{Rosenblatt1956ACL,Volkonskii1959} as primary references (see also~\citet{Bradley_05,Kulik2020}) and state useful results on $\beta$-mixing sequences.

Let $\p{\Omega, \mathcal{F}, \P}$ be a probability space and let $\mathcal{A}$ and $\mathcal{B}$ be sub-sigma-algebras of $\mathcal{F}$.
\begin{definition}
    Define 
    \begin{align*}
        &\alpha \p{\mathcal{A}, \mathcal{B}} = 
        \sup_{A \in \mathcal{A}, B \in \mathcal{B}}\left\vert \P\p{A \cap B} - \P\p{A}\P\p{B} \right\vert, \\
        &\beta \p{\mathcal{A}, \mathcal{B}} = 
        \sup \frac{1}{2} \sum_{i=1}^{I}\sum_{j=1}^{J} \left\vert \P\p{A_{i} \cap B_{j}} - \P\p{A_{i}}\P\p{B_{j}} \right\vert,
    \end{align*}
    where the latter supremum is over all pairs of finite partitions $\p{A_1, A_2, \ldots, A_I}$ and $\p{B_1, B_2, \ldots, B_J}$ of $\Omega$ such that $A_i \in \mathcal{A}$ for each $i$ and $B_j \in \mathcal{B}$ for each $j$.
\end{definition}
The functions $\alpha$ and $ \beta$ are measures of dependence.

\begin{definition}
    Let $\curly{X_j, j \in \Z}$ be a stationary sequence and let $\mathcal{F}_{a}^{b}$ be the $\sigma$-field generated by $\p{X_a, X_{a+1}, \ldots, X_b}, -\infty \leq a < b \leq \infty$. We introduce the sequence of $\alpha$-mixing coefficients and $\beta$-mixing coefficients respectively as 
    \begin{align*}
        \alpha(n) = \alpha \p{\mathcal{F}_{-\infty}^{0}, \mathcal{F}_{n}^{\infty}},\quad\beta(n) = \beta \p{\mathcal{F}_{-\infty}^{0}, \mathcal{F}_{n}^{\infty}}.
    \end{align*}
    The sequence $\curly{X_j, j \in \Z}$ is said to be 
    \begin{enumerate}
        \item $\alpha$-mixing (strongly mixing) if $\lim_{n \to \infty}\alpha(n) = 0$;
        \item $\beta$-mixing (absolutely regular) if $\lim_{n \to \infty}\beta(n) = 0$.
    \end{enumerate}    
\end{definition}
It holds that $2\alpha \p{\mathcal{A}, \mathcal{B}} \leq \beta \p{\mathcal{A}, \mathcal{B}}$ and thus $\beta$-mixing implies $\alpha$-mixing. Absolute regularity allows for comparison between the laws of a times series and another sequence based on independent blocks. The following Lemma validates the blocking method.
\begin{lemma}\label{blocking_lemma}
    Let $l \leq r$ and $m$ be integers. Let $\p{\Omega, \mathcal{F}, \P}$ be a probability space on which the following sequences are defined, (i) a stationary $\R^{d}$-valued sequence $\{X_j, j \in \Z\}$ with $\beta$-mixing coefficients $\beta_{j}$, (ii) a sequence $\{X_j, j \in \Z\}$ such that the vectors  $(X^{\dag}_{(i-1)r +1},\ldots,X^{\dag}_{ir}),  i \in \Z$ are mutually independent and each have the same distribution as $(X_{1},\ldots,X_{r})$, (iii) a sequence $\{\xi_j, j \in \Z\}$ of i.i.d. random elements taking value in a measurable set $(E, \mathcal{E})$, independent of the previous sequences.
    Let $Q_m$ and $Q_{m}^{\dag}$ be the respective laws of $\{(\xi_{i}, X_{(i-1)r +1},\ldots, X_{ir-l}), 1 \leq i \leq m\}$ and $\{(\xi_{i}, X^{\dag}_{(i-1)r +1},\ldots, X^{\dag}_{ir-l}), 1 \leq i \leq m\}$. Then 
    \begin{align*}
        \mathrm{d}_{TV}\p{Q_m, Q_{m}^{\dag}} \leq m\beta_{l},
    \end{align*}
    where $ \mathrm{d}_{TV}(\cdot, \cdot)$ denotes the total variation distance between measures.
\end{lemma}
For $n,k \geq 1$, let $h_{n,k}:\p{\R^{d}}^{k} \rightarrow \R$ be measurable function and define
\begin{align*}
    &\mathbb{H}_n = \sum_{i=1}^{m_n} h_{n,r_{n}}\p{X_{(i-1)r_{n} +1},\ldots,X_{ir_{n}}}, \quad \mathbb{H}^{\dag}_n = \sum_{i=1}^{m_n} h_{n,r_{n}}\p{X^{\dag}_{(i-1)r_{n} +1},\ldots,X^{\dag}_{ir_{n}}},\\
    &\mathbb{H}_{n}^{0} = \mathbb{H}_n - \E{\mathbb{H}_n}, \quad 
     \mathbb{H}_{n}^{\dag, 0} = \mathbb{H}^{\dag}_n - \E{\mathbb{H}^{\dag}_n}.
\end{align*}
\begin{lemma} \label{coinciding_fidis_lemma}
    Let $\curly{X_j, j \in \Z}$ be a stationary sequence in $\R^{d}$ with $\beta$-mixing coefficients $\beta_{j}$. Let $\ell_{n}, m_n, r_n$ be non-decreasing sequences of integers such that $\ell_n/r_n \to 0$. Assume that for each $n$ there exists a measurable function $f_n$ such that
    \begin{align*}
        \left\vert
        h_{n,r_{n}}\p{X_{1},\ldots,X_{r_n}} -
        h_{n,r_{n}-\ell_{n}}\p{X_{1},\ldots,X_{r_n-\ell_{n}}}
        \right\vert
        \leq f_{n}\p{X_{r_n-\ell_{n} + 1},\ldots,X_{r_n}},
    \end{align*}
    where $
        \lim_{n \to \infty}m_{n}\E{f_{n}^{2}\p{X_{1},\ldots,X_{\ell_n}}} = 0.$\\
    If $\lim_{n \to \infty}m_{n}\beta_{l_{n}} = 0,$
    and $\mathbb{H}_{n}^{\dag, 0}$ converges weakly, then $\mathbb{H}_{n}^{0}$ converges weakly to the same distribution.
\end{lemma}
\subsection*{Central Limit Theorems}
In this section we present existing results on regarding Central Limit Theorems using~\citet{BillingsleyPatrick2009Copm} and~\citet{vanderVaart2023} as primary reference, see also~\citet{Kulik2020}. 
\\
For $a < b$, define by $\mathbb{D}\left[a,b\right]$ the space of cadlag functions indexed on the compact interval $\left[a,b\right]$. We are interested in establishing weak convergence in $\mathbb{D}\left[a,b\right]$ endowed with the $J_1$-topology (also called the Skorohod topology)
\begin{definition}
    For $x$ and $y$ in $\mathbb{D}$, define
    \begin{align*}
        d_{J_{1}}(x,y) = \inf_{\lambda \in \Lambda}\curly{\norm{\lambda - I} \vee \norm{x-y(\lambda)}}, 
    \end{align*}
    where $\Lambda$ is the set of strictly increasing and continuous mappings defined on the index set of $\mathbb{D}$ onto it self and $I$ is the identity map on the index set. Then $d_{J_{1}}$ is a metric on $\mathbb{D}$ and we denote by $J_1$ the topology metrizable by this metric.
\end{definition}

\begin{theorem}[Asymptotic equicontinuity with uniform modulus of continuity]\label{equicontinuity_theorem}
Assume that a sequence of $\mathbb{D}[a,b]$-valued stochastic processes $(\mathbb{X}_n)$ has converging finite-dimensional distributions. If moreover, for any $\varepsilon>0$
\begin{align}\label{equicontinuity_condition}
\lim_{\delta\to0} \limsup_{n\to\infty}\P \Big(\sup_{\substack{a\le s\le t \le b \\ |s-t|\le \delta}}|\mathbb{X}_n(t)-\mathbb{X}_n(s)|>\varepsilon\Big)=0,
\end{align}
then there exists a stochastic process $\mathbb{X}$, with $\P (\mathbb{X}\in \mathbb{C}[a,b]) = 1$, such that $\mathbb{X}_n$ converges weakly to $\mathbb{X}$
in $\mathbb{D}[a,b]$ endowed with the $J_1$ topology.
\end{theorem}
As the condition of equation~\eqref{equicontinuity_condition} is difficult to show directly, one often invokes the covering Central Limit Theorem when dealing with arrays of stochastic processes. In the theorem we below we present an adaption to the covering CLT in the case of separable processes.

\begin{definition}[Separability]
    Let $\p{\mathcal{G}, \rho}$ be a semi-metric space. A process $\curly{\mathbb{X}(f), f \in \mathcal{G}}$ is said to be separable if there exists a countable sub-class $\mathcal{G}_{0} \subset \mathcal{G}$ an $\Omega_{0} \subseteq \Omega$ with $\P\p{\Omega_{0}}=1$ such the for all $\omega \in \Omega_{0}, f \in \mathcal{G}$ and $\varepsilon > 0$,
    \begin{align*}
        \mathbb{X}\p{f,w} \in \overline{\curly{ \mathbb{X}\p{g,w}: g \in \mathcal{G}_{0} \cap B\p{f,\varepsilon}}}, 
    \end{align*}
    where $B\p{f,\varepsilon}$ is the open ball around $f$ with radius $\varepsilon$ wrt $\rho$.
\end{definition}

\begin{definition}[Covering number]
   The covering number $N\p{\mathcal{G},  \mathrm{d}, \varepsilon}$ is the minimal number of balls $\left\{g: d(g,f) < \varepsilon \right\}$ of radius $\varepsilon$ needed to cover the set $\mathcal{G}$. The centres of the balls need not belong to $\mathcal{G}$, but the should have finite norms.
\end{definition}

\begin{lemma}\label{separable_measurable}
Assume $\mathbb{X}$ is a separable stochastic process indexed by the set $\mathcal{F}$. Then $||\mathbb{X}||_\mathcal{F}=\sup_{f\in\mathcal{F}}\{\mathbb{X}(f)\}$ is measurable.
\end{lemma}

Let $\mathbb{Z}_{n}$ be the centred empirical process indexed by a class of functions $\mathcal{G}$ defined by 
\begin{align*}
    \mathbb{Z}_{n}\p{f} = \sum_{i=1}^{m_{n}} \curly{Z_{n,i}\p{f} - \E{Z_{n,i}\p{f}}}, \ f \in \mathcal{G},
\end{align*}
where $m_n \to \infty$ as $n \to \infty$ and $\curly{Z_{n,i}, n \in \N}, i = 1,\ldots, m_{n}$ are i.i.d. separable stochastic processes. We define the random semi-metric $\mathrm{d}_{n}$ on $\mathcal{G}$ by 
\begin{align*}
    \mathrm{d}_{n}^{2}\p{f,g} = 
    \sum_{i=1}^{m_{n}} \curly{Z_{n,i}\p{f} - Z_{n,i}\p{g}}^{2}, \ f,g \in \mathcal{G}. 
\end{align*}

\begin{theorem}\label{bracketing_clt}
        Consider the sequence of processes  
\begin{align*}
    \mathbb{Z}_{n}\p{f} = \sum_{i=1}^{m_{n}} \curly{Z_{n,i}\p{f} - \E{Z_{n,i}\p{f}}}, 
\end{align*}
where for each $n\in \N, \curly{Z_{n,i}, i = 1,\ldots, m_{n}}, $ are i.i.d. separable stochastic processes and assume that the space $\p{\mathcal{G},\rho}$ is totally bounded. Suppose that
\begin{enumerate}
    \item For all $\eta > 0$,
    \begin{align*}
        \lim_{n \to \infty}m_{n}\E{\norm{Z_{n,1}}_{\mathcal{G}}^{2} \mathrm{1}\curly{\norm{Z_{n,1}}_{\mathcal{G}}^{2} > \eta}} = 0.
    \end{align*}
    \item For every sequence $\curly{\delta_{n}}$ decreasing to zero,
    \begin{align*}
        \lim_{n\to \infty} 
        \sup_{\substack{f,g \in \mathcal{G}\\ \rho\p{f,g} \leq \delta_{n}}}
        \E{ \mathrm{d}_{n}^{2}\p{f,g}} = 0.
    \end{align*}
    \item There exists a measurable majorant $N^{\ast}\p{\mathcal{G},  \mathrm{d}_{n}, \varepsilon}$ of the covering number $N\p{\mathcal{G},  \mathrm{d}_{n}, \varepsilon}$ such that for every sequence $\curly{\delta_{n}}$ which decreases to zero, 
    \begin{align}\label{entropy_condition}
        \int_{[0,\delta_{n}]}\sqrt{\log\left(N^{\ast}\p{\mathcal{G},  \mathrm{d}_{n}, \varepsilon}\right)}\mathrm{d}\varepsilon \overset{\P}{\to} 0. 
    \end{align}
\end{enumerate}
Then $\curly{Z_{n}, n \in \N}$ is asymptotically $\rho$-equicontinuous.
\end{theorem}
Consider a sequence of stochastic processes
\begin{align*}
    Z_{n,i}\p{f} = \frac{1}{\sqrt{v_n}}f\p{Y_{n,i}}, \ \ i = 1,\ldots, m_n,
\end{align*}
where $Y_{n,i}$ are i.i.d. random elements in a measurable space $E$ and $v_n$ is non-decreasing tending to infinity. For a class of real-valued functions $\mathcal{G}$ defined on $E$, define the envelope function $\boldsymbol{\mathrm{G}}:E\rightarrow \R$ by
\begin{align*}
    \boldsymbol{\mathrm{G}}(y) = \sup_{f \in \mathcal{G}}\left\vert f(y) \right\vert.
\end{align*}
Denote by $\mathcal{P}$ the set of finite discrete probability measures on $E$. For $P \in \mathcal{P}$ define
    $L^{2}\p{P} = \{f:E \rightarrow \R, P(f^{2})<\infty\}.$
For each $f,g \in L^{2}\p{P}$ write
\begin{align*}
    \norm{f}_{L^{2}\p{P}} = \sqrt{P\p{f^{2}}}, \quad \mathrm{d}_{L^{2}\p{P}}(f,g) = \sqrt{P\p{\curly{f-g}^{2}}}.
\end{align*}

The entropy condition of equation~\eqref{entropy_condition} in Theorem~\ref{bracketing_clt} can be replaced with a uniform metric entropy condition with respect to the deterministic semi-metric $\mathrm{d}_{L^{2}\p{P}}$.
\begin{lemma}\label{uniform_entropy_lemma}
    Assume that 
    \begin{align}
        \limsup_{n \to \infty}\frac{m_n\E{\boldsymbol{\mathrm{G}}^{2}(Y_{n,1})}}{v_n} < \infty,
    \end{align}
    and \begin{align}\label{uniform_entropy_condition}
        \lim_{\delta \to 0}
        \int_{[0,\delta]}\sup_{Q \in \mathcal{P}}\sqrt{\log\left(N\p{\mathcal{G},  \mathrm{d}_{L^{2}\p{P}},  \norm{\boldsymbol{\mathrm{G}}}_{L^{2}\p{Q}}\varepsilon}\right)}\mathrm{d}\varepsilon = 0. 
    \end{align}
    \hspace{4pt}
    Then equation~\eqref{entropy_condition} holds.
\end{lemma}
\begin{lemma}\label{linearly_ordered_class_lemma}
    If $\mathcal{G}$ is a VC-subgraph class with measurable envelope function $\boldsymbol{\mathrm{G}}$, then~\eqref{uniform_entropy_condition} holds. In particular,~\eqref{uniform_entropy_condition} holds for linearly ordered classes.
\end{lemma}

\begin{theorem}[Vervaat's Lemma]\label{Vervaat}
    Consider non-empty intervals $I, J$ and let $\curly{\mathbb{X}, \mathbb{X}_{n}, n \geq 1}$ be $\mathbb{D}\p{I}$-valued stochastic processes and assume that $\mathbb{X}$ has almost surely continuous sample paths. Let $h:I \rightarrow J$ be a surjective, continuous and increasing function with positive derivative $h'$. If there exists a deterministic sequence $c_n \to \infty$ such that $c_{n}\p{\mathbb{X}_{n} - h} \overset{d}{\to} \mathbb{X}$ in $\mathbb{D}\p{I}$ endowed with the $J_1$-topology and $\mathbb{X}_{n}, n \geq 1$ have non-decreasing sample paths, then 
    \begin{align*}
        c_{n}\p{\mathbb{X}_{n} - h, \mathbb{X}^{\leftarrow}_{n} - h^{\leftarrow} } 
        \overset{d}{\to} \p{\mathbb{X}, -\p{h^{\leftarrow}}' \times \mathbb{X}\circ h^{\leftarrow}},
    \end{align*}
    in $\mathbb{D}\p{I} \times \mathbb{D}\p{J}$ endowed with the product $J_1$ topology.
\end{theorem}
A simple application of the Continuous Mapping Theorem yields that Vervaat's Lemma holds in the case where $h$ is decreasing and $\mathbb{X}_{n}, n \geq 1$ have non-increasing sample paths.

\end{appendix}
\end{document}